\documentclass[twoside]{article}

\usepackage[accepted]{aistats2025}
\usepackage[round]{natbib}

\usepackage{graphicx}
\usepackage{float}

\usepackage[utf8]{inputenc} 
\usepackage[T1]{fontenc}    
\usepackage{hyperref}       
\hypersetup{colorlinks=true,linkcolor=blue,citecolor=blue}
\usepackage{url}            
\usepackage{booktabs}       
\usepackage{amsfonts}       
\usepackage{microtype}      
\usepackage{xcolor}         

\usepackage{algorithm}
\usepackage{algorithmic}

\usepackage{booktabs}
\usepackage{arydshln}
\usepackage{multirow}
\usepackage{colortbl}
\usepackage{threeparttable}

\usepackage{amsmath}
\usepackage{amsxtra,amsfonts,amscd,amssymb,bm,bbm}
\usepackage{nicefrac}       
\usepackage{amsthm,thmtools,thm-restate}

\newtheorem{theorem}{Theorem}[section]
\newtheorem{proposition}[theorem]{Proposition}
\newtheorem{lemma}[theorem]{Lemma}

\theoremstyle{definition}
\newtheorem{definition}[theorem]{Definition}
\newtheorem{assumption}[theorem]{Assumption}

\usepackage[capitalize,noabbrev]{cleveref}
\Crefname{assumption}{Assumption}{Assumptions}

\usepackage{enumitem}

\newcommand{\innerp}[1]{\left\langle {#1} \right\rangle }
\newcommand{\norm}[1]{\left\| {#1} \right\| }
\newcommand{\abs}[1]{\left | {#1} \right |}
\newcommand{\zkh}[1]{\left[ {#1} \right]}
\newcommand{\hkh}[1]{\left\{ {#1} \right\}}
\newcommand{\kh}[1]{\left( {#1} \right)}

\DeclareMathOperator*{\argmin}{arg\,min}
\DeclareMathOperator*{\argmax}{arg\,max}

\newcommand{\todo}[1]{{#1}}
\usepackage[normalem]{ulem} 
\usepackage{xcolor}         

\newcommand{\revise}[1]{{#1}}

\definecolor{comblue}{RGB}{2,0,255}

\newcommand{\mdp}{\mathcal{M}}
\newcommand{\mdps}{\mathcal{S}}
\newcommand{\mdpa}{\mathcal{A}}

\newcommand{\sqrta}{\sqrt{\left| \mathcal{A} \right|}}
\newcommand{\Exp}[1]{\mathbb{E}\zkh{#1}}

\definecolor{Gray}{gray}{0.85}

\begin{document}

%

%
\runningauthor{Yan Yang, Bin Gao, Ya-xiang Yuan}

\twocolumn[

\aistatstitle{Bilevel Reinforcement Learning via the Development of Hyper-gradient without Lower-Level Convexity}

\aistatsauthor{ Yan Yang\,\textsuperscript{$1$}\,\textsuperscript{$2$}\qquad Bin Gao\,\textsuperscript{$1$}\qquad  Ya-xiang Yuan\,\textsuperscript{$1$}}
\vspace{0.3cm}
\aistatsaddress{\textsuperscript{$1$}\,Academy of Mathematics and Systems Science, Chinese Academy of Sciences \\\textsuperscript{$2$}\,{University of Chinese Academy of Sciences}}]


\begin{abstract}
Bilevel reinforcement learning (RL), which features intertwined two-level problems, has attracted growing interest recently. The inherent non-convexity of the lower-level RL problem is, however, to be an impediment to developing bilevel optimization methods. By employing the fixed point equation associated with the regularized RL, we characterize the hyper-gradient via fully first-order information, thus circumventing the assumption of lower-level convexity. This, remarkably, distinguishes our development of hyper-gradient from the general AID-based bilevel frameworks since we take advantage of the specific structure of RL problems. Moreover, we design both model-based and model-free bilevel reinforcement learning algorithms, facilitated by access to the fully first-order hyper-gradient. Both algorithms enjoy the convergence rate $\mathcal{O}\kh{\epsilon^{-1}}$. \todo{To extend the applicability, a stochastic version of the model-free algorithm is proposed, along with results on its iteration and sample complexity.} In addition, numerical experiments demonstrate that the hyper-gradient indeed serves as an integration of exploitation and exploration.
\end{abstract}

\section{INTRODUCTION}\label{sec:intro}
Bilevel optimization, aiming to solve problems with a hierarchical structure, achieves success in~a wide range of machine learning applications, e.g., hyper-parameter optimization~\citep{franceschi2017forward,franceschi2018bilevel,ye2023difference}, meta-learning~\citep{bertinetto2019meta}, computer vision \citep{liu2021CVBiO}, neural architecture search~\citep{liu2019darts,wang2022zarts}, adversarial training~\citep{wang2021adver}, reinforcement learning~\citep{chakraborty2024parl,shen2024bilevelRL,thoma2024HPGD}, data poisoning \citep{liu2024poisoning}. To address the bilevel optimization problems, a line of research has emerged recently~\citep{ghadimi2018bsa,liu2020generic,liu2021value,ji2022will,hu2023cdb,kwon2023f2sa,shen2023penalty,hao2024Unbounded,yao2024Hessianfree}. Generally, the upper-level (resp. the lower-level) problem optimizes the decision taken by the leader (resp. the follower), which exhibits potential for handling complicated decision-making processes such as Markov decision processes (MDPs).

Reinforcement learning (RL) \citep{sutton2018reinforcement,szepesvari2022reinforcement} serves as an effective way of learning to make sequential decisions in MDPs, and has seen plenty of applications \citep{silver2017go,berner2019dota,mirhoseini2021graph,miki2022robotics,sun2024haisor}. The central task of RL is to find the optimal policy that maximizes the expected cumulative rewards in an MDP. Bilevel RL enriches the framework of RL by considering a two-level problem: the follower solves a standard RL problem within an environment parameterized by the decision variable taken by the leader; meanwhile, the leader optimizes the decision variable based on the response policy from the lower level. Recently, bilevel RL has gained increasing attention in practice, including RL from human feedback, \citep{christiano2017rlhf}, inverse RL \citep{brown2019IRL}, and reward shaping \citep{hu2020rewardshaping}. 

\begin{table*}[t]
    \setlength{\tabcolsep}{5pt}
    \centering
    {
    \caption{Comparison among bilevel reinforcement learning algorithms.}
    \label{tab:comparision_BiRL}
    \vspace{0.1cm}
    \begin{threeparttable}
        \begin{tabular}{lcllc}
            \toprule
            { Algorithm } & {Deter. or Stoch.} & { Conv. Rate } & { Inner Iter. } & {Oracle}\\
            \midrule
            PARL \citep{chakraborty2024parl} & Deter.  & $\mathcal{O}(\epsilon^{-1})$ & $\mathcal{O}\kh{\log{\epsilon^{-1}}}$ & 1st$+$2nd \\
            PBRL \citep{shen2024bilevelRL} & Deter. & $\mathcal{O}(\lambda\epsilon^{-1})$  & $\mathcal{O}\kh{\log{\lambda^2\epsilon^{-1}}}$  & 1st\\
            HPGD \citep{thoma2024HPGD} & Stoch. & $\mathcal{O}(\epsilon^{-2})$  & $\mathcal{O}\kh{\log{\epsilon^{-1}}}$  & 1st\\
            M-SoBiRL (this work) & Deter. & $\mathcal{O}(\epsilon^{-1})$ & $\mathcal{O}\kh{1}$ & 1st\\           
            SoBiRL (this work) & Deter. & $\mathcal{O}(\epsilon^{-1})$ & $\mathcal{O}\kh{\log{\epsilon^{-1}}}$ & 1st \\
            Stoc-SoBiRL (this work) & Stoch. & $\widetilde{\mathcal{O}}(\epsilon^{-1.5})$ & $\mathcal{O}\kh{\log{\epsilon^{-1}}}$ & 1st \\
            \bottomrule	
        \end{tabular}
    \end{threeparttable}
    }
\end{table*}

In this paper, we focus on the bilevel reinforcement learning problem:
\begin{equation}\label{eq:informal_biRL}
    \begin{array}{cl}
    \min\limits_{x \in \mathbb{R}^n}& \phi(x):=f\left(x,\pi^*(x)\right)
    \\
    \mathrm{s.\,t.}& \pi^*(x) \in \argmin\limits_{\pi \in \Pi} g(x,\pi),    
    \end{array}
\end{equation}
where $\Pi$ is the policy set of interest, the upper-level function~$f$ is defined on~$\mathbb{R}^{n}\times\Pi$, the univariate function $\phi$~is called the \emph{hyper-objective}, and the gradient of $\phi(x)$~is referred to as the \emph{hyper-gradient} \citep{pedregosa2016hyperparameter,grazzi2020FPBiO,yang2023accelerating,chen2023hyper} if it exists. The approximate implicit differentiation~{(AID)} based method which resorts to the hyper-gradient has become flourishing recently~\citep{ghadimi2018bsa,ji2021stocbio,arbel2022amortized,dagreou2022soba,liu2023average,hu2024contextualstoc,gao2024lancbio}. Specifically, in each \emph{outer iteration} of the AID-based method, one implements an inexact hyper-gradient descent step $x_{k+1}=x_k-\beta\widetilde{\nabla} \phi (x_k)$, where the estimator $\widetilde{\nabla} \phi (x_k)$ of the hyper-gradient is obtained with the help of a few \emph{inner iterations}. We consider the extension of AID-based methods to the bilevel reinforcement learning problem~\eqref{eq:informal_biRL}. Note that AID-based methods depend on the lower-level strong convexity~\citep{khanduri2021sustain,huang2022enhancedbregman,li2022fsla} or uniform Polyak--\L{}ojasiewicz~{(PL)} condition~\citep{huang2023momentum,chakraborty2024parl} to ensure the existence of the hyper-gradient. However, the lower-level problem in \eqref{eq:informal_biRL}---always an RL problem---is inherently non-convex even with strongly-convex regularization \citep{agarwal2020tabular,lan2023mirror}, and only the non-uniform PL condition has been established \citep{mei2020softmax}, which renders $\nabla \phi$ ambiguous, as stated in~\citep{shen2024bilevelRL}. 

\revise{Recently, \citet{chen2022adaptivemodel} assumed the convexity of the hyper-objective and proposed a deterministic algorithm for the bilevel RL problem.}
\citet{chakraborty2024parl} presented an AID-based framework, assuming that the lower-level problem satisfies the uniform PL condition and the Hessian non-singularity condition. \citet{shen2024bilevelRL} proposed a penalty-based bilevel RL algorithm to bypass the requirement of lower-level convexity by constructing two penalty functions. The convergence rate relies on the penalty parameter $\lambda$ that is at least the order of~${\epsilon^{-0.5}}$. \todo{\citet{thoma2024HPGD} designed a stochastic bilevel RL method, achieving the convergence rate $\mathcal{O}(\epsilon^{-2})$. In this paper, we develop model-based and model-free bilevel RL algorithms and extend the model-free algorithm to stochastic settings, all of which are provable and exhibit an enhanced convergence rate, without the lower-lower convex assumption; see \cref{tab:comparision_BiRL} for a detailed comparison with the existing works, where the notation $\widetilde{\mathcal{O}}(\cdot)$ hides logarithmic terms of $\epsilon^{-1}$.}

{\textbf{Contributions}} The main contributions are summarized as follows.

Firstly, we characterize the hyper-gradient of bilevel RL problem via fully first-order information and unveil its properties by investigating the fixed point equation associated with the entropy-regularized RL problem \citep{nachum2017bridging,geist2019regRLtheory,yang2019regsparse}, which extends the spirits in \citep{christianson1994fixedpoints,grazzi2020FPBiO,grazzi2021stochasticFPBiO,grazzi2023FPBiO}.

Secondly, understanding the hyper-gradient enables us to construct its estimators, upon which we devise a model-based bilevel RL algorithm, M-SoBiRL, together with a model-free version, SoBiRL. \todo{For broader applicability, we extend SoBiRL in stochastic settings by designing a sampling scheme to estimate the hyper-gradient and a framework Stoc-SoBiRL aided with a momentum technique to accommodate the stochastic hyper-gradient estimator}. Specifically, the implementation only requires first-order oracles, which circumvents complicated second-order queries in general AID-based bilevel methods. 

Finally, we offer an analysis to illustrate the efficiency of amortizing the hyper-gradient approximation through outer iterations in M-SoBiRL, i.e., it enjoys the convergence rate $\mathcal{O}(\epsilon^{-1})$ with the inner iteration number $N=\mathcal{O}(1)$ independent of the solution accuracy $\epsilon$. In the model-free scenario, an enhanced convergence property is also established. \todo{Moreover, investigating the statistical properties of Stoc-SoBiRL, we show its iteration complexity $\widetilde{\mathcal{O}}(\epsilon^{-1.5})$ and sample complexity $\widetilde{\mathcal{O}}(\epsilon^{-3.5})$; refer to \cref{tab:comparision_BiRL} for a detailed comparison. To the best of our knowledge, it is the first sample complexity result established for bilevel RL problems. In addition, a synthetic experiment verifies the convergence of M-SoBiRL, and the favorable performance of the proposed SoBiRL is validated on Atari games and Mujoco simulations, which implies that the hyper-gradient is an aggregation of exploitation and exploration.}

\section{RELATED WORK}
The introduction to related bilevel optimization methods can be found in~\cref{sec:related_bio}.

\textbf{Entropy-regularized Reinforcement Learning}  Entropy regularization is commonly considered in the RL community. Specifically, the goal of entropy-regularized RL is to maximize the expected reward augmented with the policy entropy, thereby boosting both task success and behavior stochasticity. It facilitates exploration and robustness~\citep{ziebart2010modeling,haarnoja2017softlearning,haarnoja2018sac,lan2021lanquantum}, smoothens the optimization landscape~\citep{ahmed2019impact}, and enhances the convergence property~\citep{mnih2016A3C,cen2022fastNPG,zhan2023mirror,lan2023mirror,yang2023accelerating,li2024preconditionmdp}. Moreover, the policy optimality condition under entropy-regularized setting is equivalent to the \emph{softmax temporal value consistency} in~\citep{nachum2017bridging}. Therefore, the term \emph{soft} is prefixed to quantities in this scenario, also resonating with the concept in~\citep{sutton2018reinforcement}, where ``soft'' means that the policy ensures positive probabilities across all state-action pairs.

\textbf{Bilevel Reinforcement Learning} Extensive reinforcement learning applications hold the bilevel structure, e.g., reward shaping \citep{sorg2010rewarddesign,zheng2018intrinsicrewards,hu2020rewardshaping,devidze2022explorationguided,gupta2023behavioralign}, offline reward correction \citep{li2023rewardcorrection}, preference-based RL \citep{christiano2017rlhf,lee2021pebble,saha2023duelingrl,shen2024bilevelRL}, apprenticeship learning \citep{arora2021apprenticeship}, Stackleberg games \citep{fiez2020implicitstack,song2023nash,shen2024bilevelRL}, etc. In terms of provable bilevel reinforcement learning frameworks, \citet{chakraborty2024parl} proposed a policy alignment algorithm demonstrating performance improvements, \citet{shen2024bilevelRL} designed a penalty-based method, \todo{and \citet{thoma2024HPGD} investigated the problem in the stochastic setting.} However, the research related to the hyper-gradient in bilevel RL remains constrained.

\section{PROBLEM FORMULATION}\label{sec:pro_formulation}
In this section, we will introduce the setting of entropy-regularized Markov decision processes~\citep{williams1991entropy,ziebart2010modeling,nachum2017bridging,haarnoja2017softlearning}, and the formulation of bilevel reinforcement learning. Moreover, two specific instances of bilevel RL are introduced. The notation is listed in \cref{sec:notations}.

\subsection{Entropy-regularized MDPs}\label{sec:reg_RL}
{\textbf{Discounted Infinite-horizon MDPs}} An MDP is characterized by a~tuple~$\mdp_\tau=(\mdps,\mdpa,P,r,\gamma,\tau)$ with $\mdps$ and $\mdpa$ serving as the state space and the action space, respectively. In this paper, we restrict the focus to the tabular setting, where both $\mdps$ and $\mdpa$ are finite,~i.e.,~$\abs{\mdps}<+\infty$ and $\abs{\mdpa}<+\infty$. Furthermore, $P\in\mathbb{R}^{\abs{\mdps}\abs{\mdpa}\times\abs{\mdps}}$ is the transition matrix with $P_{sas^\prime}=P(s^\prime|s,a)$ representing the transition probability from state $s$ to state $s^\prime$ under action~$a$. The vector $r\in\mathbb{R}^{\abs{\mdps}\abs{\mdpa}}$ specifies the reward $r_{sa}$ received when action~$a$~is carried out at state $s$. Note that $\gamma \in (0,1)$ is the discount factor, adjusting the importance of immediate versus future reward, and the \emph{temperature parameter} $\tau\ge0$ balances regularization and reward.

{\textbf{Soft Value Function and Q-function}} A policy $\pi\in\mathbb{R}^{\abs{\mdps}\abs{\mdpa}}$ provides an action selection rule, that is, for any $s\in \mdps,\ a\in \mdpa$, $\pi_{sa}=\pi(a|s)$ is the probability of performing action~$a$ at state~$s$, and we denote~$\pi_s\in \mathbb{R}^{\abs{\mdpa}}$ to be the corresponding distribution over~$\mdpa$ at state~$s$, which means $\pi_s$ belongs to the probability simplex $\Delta(\mdpa):=\{x\in\mathbb{R}^{\abs{\mdpa}}:x_i\ge 0,\ \sum_{i=1}^{\abs{\mdpa}}x_i=1\}$. Collecting $\pi_s$ over $s\in\mdps$ defines $\Delta^{\abs{\mdps}}(\mdpa):=\{\pi=(\pi_s)_{s\in\mdps}\in\mathbb{R}^{\abs{\mdps}\abs{\mdpa}}:\pi_s\in\Delta(\mdpa)\text{ for }s\in\mdps\}$ as the feasible set of policies. For simplicity, the above two sets are abbreviated as $\Delta$ and $\Delta^{\abs{\mdps}}$, respectively. A policy $\pi$ induces a $P^\pi\in\mathbb{R}^{\abs{\mdps}\times\abs{\mdps}}$, measuring the transition probability from one state to another, i.e., $P^\pi_{ss^\prime}=\sum_a\pi_{sa}P_{sas^\prime}$. To promote stochasticity and discourage premature convergence to sub-optimal policies \citep{ahmed2019impact}, it commonly resorts to the entropy function $h(\cdot)$,
\begin{equation*}
    h(\pi_s):=-\sum_{a\in\mdpa}\pi_{sa}\log{\pi_{sa}},\ \text{ for }\pi_s\in\mathbb{R}_+^{\abs{\mdpa}}\cap \Delta,
\end{equation*}
where $\mathbb{R}_+^{m}:=\{z\in\mathbb{R}^{m}:z_i>0\}$. Consequently, the \emph{soft value function}~$V^\pi\in\mathbb{R}^{\abs{\mdps}}$ represents the expected discounted reward augmented with policy entropy, i.e., 
\begin{equation*}
    V^\pi_s:=\mathbb{E}\Big[\sum_{t=0}^\infty \gamma^t \kh{r_{s_ta_t}+\tau h(\pi_{s_t})}\ \big |\ s_0=s,\pi\Big],
\end{equation*}
for $s\in\mdps$, where the expectation is taken over the trajectory~$(s_0,a_0\sim\pi(\cdot|{s_0}),s_1\sim P(\cdot |s_0,a_0),\ldots)$. Given the initial state distribution $\rho\in \mathbb{R}_+^{\abs{\mdps}}\cap\Delta\kh{{\mdps}}$, the objective of RL is to find the optimal policy $\pi\in\Delta^{\abs{\mdps}}$ that solves the problem:
\begin{equation}\label{eq:reg_exp_mdp}
    \max_{\pi\in\Delta^{\abs{\mdps}}}\ \ \  V^\pi(\rho) := \mathbb{E}_{s\sim\rho}\zkh{V^\pi_s}.
\end{equation}
The {soft Q-function} $Q^\pi\in\mathbb{R}^{\abs{\mdps}\abs{\mdpa}}$ associated with policy $\pi$, couples with~$V^\pi$ in the following fashion,
\begin{align}
    \forall(s, a) \in \mathcal{S} \times \mathcal{A}: Q^\pi_{sa} \!&=r_{sa}+\gamma \mathbb{E}_{s^{\prime} \sim P(\cdot | s, a)}\left[V^\pi_{s^{\prime}}\right],    \label{eq:V_to_Q}
    \\
    \forall s \in \mathcal{S}:  V^\pi_s & =\mathbb{E}_{a \sim \pi(\cdot | s)}\left[-\tau \log \pi(a | s)+Q^\pi_{sa}\right] .  \nonumber
\end{align}
It is worth noting that the optimization problem \eqref{eq:reg_exp_mdp} admits a unique \emph{optimal soft policy} $\pi^*$ independent of~$\rho$~\citep{nachum2017bridging}. The corresponding \emph{optimal soft value
function} (resp. \emph{optimal soft Q-function}) is denoted by $V^*:=V^{\pi^*}$ (resp. $Q^*:=Q^{\pi^*}$). If necessary, we add a subscript to clarify the environment in which the policy is evaluated, e.g., $V^\pi_{\mathcal{M}_\tau}$ and $Q^\pi_{\mathcal{M}_\tau}$.

\subsection{Bilevel RL Formulation}
The single-level reinforcement learning task in ~\cref{sec:reg_RL} trains an agent with fixed rewards. In~contrast, this work considers the scenario where the reward function is parameterized by a decision variable, a task dubbed ``policy alignment'' in~\citep{chakraborty2024parl}. Generally, we parameterize the reward $r(x)\in\mathbb{R}^{\abs{\mdps}\abs{\mdpa}}$ with~$x\in \mathbb{R}^n$, which results in a parameterized MDP $\mdp_\tau(x)=(\mdps,\mdpa,P,r(x),\gamma,\tau)$. In parallel with~\cref{sec:reg_RL}, given $\mdp_\tau(x)$, we define the soft value function $V^\pi(x)$, the soft Q-function $Q^{\pi}(x)$, the objective $V_{\mathcal{M} _{\tau}\left( x \right)}^{\pi}\left( \rho \right) $, the optimal soft policy $\pi^*(x)$, the optimal soft value-function $V^*(x)$, and the optimal soft Q-function $Q^*(x)$, associated with $x$.

Consequently, we formulate the bilevel reinforcement learning problem:
\begin{equation}\label{eq:standar_biRL}
    \begin{array}{cl}
    \min\limits_{x \in \mathbb{R}^n}& \phi(x):=f\left(x,\pi^*(x)\right)
    \\
    \mathrm{s.\,t.}& \pi^*(x) = \argmin\limits_{\pi \in \Delta^{\abs{\mdpa}}} -V_{\mathcal{M} _{\tau}\left( x \right)}^{\pi}\left( \rho \right) ,    
    \end{array}
\end{equation}
where the upper-level~$f$ is defined on~$\mathbb{R}^{n}\times\mathbb{R}^{\abs{\mdps}\abs{\mdpa}}$.

\subsection{Applications of Bilevel RL}\label{sec:biRL_app}
We introduce two examples unified in the formulation~\eqref{eq:standar_biRL}. To make a distinction between environments at two levels, we denote the upper-level MDP by $\bar{\mathcal{M}}_{\bar{\tau}} = \{\mdps, \mdpa, \bar{P}, \bar{r}, \bar{\gamma}, \bar{\tau}\}$ with an initial state distribution $\bar{\rho}$, and the lower-level MDP by $\mathcal{M}_{\tau}(x)$ with $\rho$. \revise{More applications can be found in~\cref{sec:append_app}.}

\textbf{Reward Shaping} In reinforcement learning, the reward function acts as the guiding signal to motivate agents to achieve specified goals. However, in many cases, the rewards are sparse, which impedes the policy learning, or are partially incorrect, which leads to inaccurate policies. To this end, from the perspective of bilevel RL, it is advisable to shape an auxiliary reward function $r(x)$ at the lower level for efficient agent training, while maintaining the original environment at the upper level to align with the initial task evaluation~\citep{hu2020rewardshaping}, which is established as
\begin{equation*}
    \begin{aligned}
        \min _{x\in\mathbb{R}^n}-V_{\bar{\mathcal{M}}_{\bar{\tau}}}^{\pi ^*\left( x \right)}\left( \bar{\rho} \right),\quad \mathrm{s.\,t.}\ \pi ^*\left( x \right) =\argmin _{\pi \in \Delta^{\abs{\mdpa}}}-V_{\mathcal{M} _{\tau}\left( x \right)}^{\pi}\left( \rho \right).
    \end{aligned}
\end{equation*}

\textbf{Reinforcement Learning from Human Feedback (RLHF)} The target of RLHF is to learn the intrinsic reward function that incorporates expert knowledge, from simple labels only containing human preferences. Drawing from the original framework \citep{christiano2017rlhf}, \citet{chakraborty2024parl} and \citet{shen2024bilevelRL} have formulated it in a bilevel form, which optimizes a policy under $r(x)$ at the lower level, and adjusts $x$ to align the preference predicted by the reward model $r(x)$ with the true labels at the upper level.
\begin{equation*}
    \begin{array}{cl}
    \min \limits_{x\in\mathbb{R}^n} & \mathbb{E} _{y,d_1,d_2\sim \rho \left( d;\pi ^*\left( x \right) \right)}\left[ l_h\left( d_1,d_2,y;x \right)  \right]
    \\
    \mathrm{s.\,t.}& \pi ^*\left( x \right) =\argmin\limits_{\pi \in \Delta^{\abs{\mdpa}}}-V_{\mathcal{M} _{\tau}\left( x \right)}^{\pi}\left( \rho \right),  
    \end{array}
\end{equation*}
where each trajectory $d_i=\{\kh{s_h^i,a_h^i}\}_{h=0}^{H-1}$ ($i=1,2$) is sampled from the distribution~$\rho\kh{d;\pi^*(x)}$ generated by the policy $\pi^*(x)$ in the upper-level~$\bar{\mathcal{M}}_{\bar{\tau}}$, i.e., 
\begin{align*}
    P(d_i)=&\bar{\rho}\kh{s_0^i}\zkh{\Pi_{h=0}^{H-2}\pi^*(x)\kh{a_h^i|s_h^i}\bar{P}\kh{s_{h+1}^i|s_h^i,a_h^i}}
    \\
    &\times\pi^*(x)\kh{a_{H-1}^i|s_{H-1}^i},
\end{align*}
and the preference label $y\in\{0,1\}$, indicating preference for $d_1$ over $d_2$, obeys human feedback distribution $y\sim D_{human}(y| d_1,d_2)$. Moreover, $l_h$ is the binary cross-entropy loss, $l_h\left( d_1,d_2,y;x \right) = -y\log P\left( d_1\succ d_2;x \right)-\left( 1-y \right) \log P\left( d_2\succ d_1 ; x\right)$, with the preference probability $P\kh{d_1\succ d_2;x}$ built by the Bradley--Terry model; see~\eqref{eq:BT_model} in the appendix.

\section{MODEL-BASED SOFT BILEVEL REINFORCEMENT LEARNING}\label{sec:MSoBiRL}
Addressing the bilevel RL problem \eqref{eq:standar_biRL} is challenging from two aspects: 1)~analyzing the properties of $\pi^*(x)$, which has not been well studied but plays~a~crucial role in AID-based methods; 2)~characterizing the hyper-gradient $\nabla \phi(x)$ which remains unclear. To this end, we take advantage of the specific structure of entropy-regularized RL to investigate $\pi^*(x)$ and identify~$\nabla \phi(x)$. As a result, we propose a model-based algorithm to solve the bilevel RL problem \eqref{eq:standar_biRL}.

Recall that the optimal soft quantities adhere to the following softmax temporal value consistency conditions~\citep{nachum2017bridging}: $\forall(s, a) \in \mathcal{S} \times \mathcal{A}$,
\begin{footnotesize}
\begin{align}
    \!\!\pi_{sa}^*(x)
    \!&=\!{\exp \left\{ Q_{sa}^{*}\left( x \right) /\tau \right\}}\big/\big({\sum_{a^{\prime}}{\exp \left( Q_{sa^{\prime}}^{*}\left( x \right) /\tau \right)}}\big), \label{eq:consis_policy}
    \\
    \!\!V_{s}^{*}\left( x \right)\! &=\!\tau \log\!\left( \!\sum\limits_a{\exp \left( \frac{r_{sa}\left( x \right) +\gamma \sum_{s^{\prime}}{P_{sas^{\prime}}V_{s^{\prime}}^{*}\left( x \right)}}{\tau} \right)}\!\!\right)\!\label{eq:consis_v}
    \\
    &=\!\tau \log \left( \sum\nolimits_a{\exp \left( Q_{sa}^{*}\left( x \right) /\tau \right)} \right).  \label{eq:comply_V}
\end{align}    
\end{footnotesize}

By differentiating \eqref{eq:consis_policy} and incorporating \eqref{eq:V_to_Q}, we assemble $\{\nabla\pi_{sa}^*(x)\}$ in matrix form and obtain
\begin{equation}\label{eq:pi_v_nabla}
    \nabla \pi ^*(x)={\tau^{-1}}\mathrm{diag}\left( \pi^*(x)\right) \left( \nabla r(x)-{U}\nabla V^*\left( x \right) \right) 
\end{equation}
with $U\in\mathbb{R}^{\abs{\mdps}\abs{\mdpa}\times\abs{\mdps}}$ defined by $U_{sas^\prime}:=1-\gamma P_{sas^\prime}$ for $s=s^\prime$ and $U_{sas^\prime}:=-\gamma P_{sas^\prime}$, otherwise. Hence, differentiating $\pi^*(x)$ boils down to considering the \emph{implicit differentiation} $\nabla V ^*(x)$. In view of~\eqref{eq:consis_v}, we notice that $V^*(x)$~is a fixed point related to the mapping $\varphi$,
\begin{align*}
    \varphi(x,v):\ \mathbb{R}^{n}\times\mathbb{R}^{\abs{\mdps}} &\longrightarrow \mathbb{R}^{\abs{\mdps}}
    \\
    (x,v)&\longmapsto\tau \log\!\kh{\!\exp\!\kh{\frac{r(x)+\gamma{P}v}{\tau}}\cdot{\bf 1}\!},
\end{align*}
where $\log(\cdot)$ and $\exp(\cdot)$ are element-wise operations, and ${\bf 1}\in\mathbb{R}^{\abs{\mdps}}$ is an all-ones vector. The structure of $\varphi(\cdot,\cdot)$, in consequence, can facilitate the characterization of $\nabla V ^*(x)$, as outlined in the following proposition.
\begin{proposition}\label{pro:gradient_V}
    For any $x\in\mathbb{R}^n$, $\varphi(x,\cdot)$ is a contraction mapping, i.e., $\|\nabla_v \varphi(x,v)\|_{\infty}=\gamma<1$, and the matrix ${I-\nabla_{v}\varphi(x,v)}$ is invertible. Consequently, $V^*(x)$ is the unique fixed point of $\varphi(x,\cdot)$, with a well-defined derivative $\nabla V^*(x)$, given by
    \begin{equation*}\label{eq:inverse_Jac_product}
         \nabla V^*(x) = \kh{I-\nabla_{v}\varphi(x,V^*(x))}^{-1}\nabla_{x}\varphi(x,V^*(x)).
    \end{equation*}    
    Additionally, $\nabla_v \varphi\kh{x,V^*(x)}$ coincides with the $\gamma$-scaled transition matrix induced by the optimal soft policy $\pi^*(x)$, i.e., $\nabla_{v}\varphi\kh{x,V^*(x)} = \gamma P^{\pi^*(x)}$.
\end{proposition}

Combining the above proposition with \eqref{eq:pi_v_nabla}, we identify the hyper-gradient $\nabla\phi (x)$ to the problem~\eqref{eq:standar_biRL}; \todo{its expression is delayed in \eqref{eq:exact_hypergrad2} of \cref{sec:ana_MSoBiRL}}. In this manner, we unveil the hyper-gradient in the context of bilevel RL, by harnessing the softmax temporal value consistency and derivatives of a fixed-point equation.

In light of the characterization of $\nabla\phi$, a model-based soft bilevel reinforcement learning algorithm,~called M-SoBiRL, is proposed in~\cref{alg:M-SoBiRL}. In summary, the $k$-th outer iteration performs an inexact hyper-gradient descent step on $x_k$, with the aid of auxiliary iterates $\kh{\pi_k,Q_k,V_k,w_k}$. Specifically, given $x_{k+1}$, the inner iterations aim to (approximately) solve the lower-level entropy-regularized MDP problem. To this end, the \emph{soft policy iteration} studied by \citet{haarnoja2018sac,cen2022fastNPG} inspires us to define the following soft Bellman optimality operator $\mathcal{T} _{\mathcal{M} _{\tau}\left( x \right)}:\mathbb{R}^{\abs{\mdps}\abs{\mdpa}} \to\mathbb{R}^{\abs{\mdps}\abs{\mdpa}}$ associated with $x\in\mathbb{R}^n$,
\begin{small}
 \begin{equation*}
    \mathcal{T} _{\mathcal{M} _{\tau}\left( x \right)}(Q)(s,a):=r_{sa}(x)+\gamma {\mathbb{E}}\left[ \tau \log \left( \left\| \exp \left( Q_{s^{\prime},\cdot} /\tau \right) \right\| _1 \right) \right],
\end{equation*}   
\end{small}
\!\!where the expectation is taken with respect to $s^{\prime}\sim {P(\cdot \mid s,a)}$. Applying this operator iteratively leads to the optimal soft Q-value function, with a linear convergence rate in theory (see \cref{sec:Bellman}). Therefore, $N$ inner iterations are invoked in line~7-9 of \cref{alg:M-SoBiRL} to estimate $Q^*\kh{x_{k+1}}$, and the warm-start strategy is adopted in line~10 of \cref{alg:M-SoBiRL} to initialize the next outer iteration with historical information, where it amortizes the computation of $Q^*(x)$ through the outer iterations. Additionally, in order to recover $\pi_{k+1}$ from $Q_{k+1}$, we take into account the consistency condition~\eqref{eq:consis_policy} which involves the softmax mapping as follows,
\begin{align*}
    \texttt{softmax}:\ \ \mathbb{R}^{\abs{\mdps}\abs{\mdpa}} &\longrightarrow\mathbb{R}^{\abs{\mdps}\abs{\mdpa}}
    \\
    \theta &\longmapsto \kh{\frac{\exp \left( \theta \left( s,a \right) \right)}{\sum_{a^{\prime}}{\exp \left( \theta \left( s,a^{\prime} \right) \right)}}}_{sa}.
\end{align*}
It follows from \eqref{eq:consis_policy} that $\pi_{k+1}=\texttt{softmax}\kh{Q_{k+1}/\tau}$.
As for the approximation of the hyper-gradient~\eqref{eq:exact_hypergrad2} which requires an inverse matrix vector product, we denote
\begin{equation*}
A_k=\kh{I-\gamma P^{\pi_k}}^{\top},\quad b_k=U^\top\operatorname{diag}\kh{\pi_k}\nabla_\pi f(x_k,\pi_k),
\end{equation*}
and employ $w_k$ to track the quantity $A_k^{-1}b_k$ in a similar principle of amortization. Concretely, $w_k$ is regarded as the (approximate) solution of the least squares problem, $\min_{w\in \mathbb{R}^{\abs{\mdps}}}\frac{1}{2}\norm{A_kw-b_k}^2$, and implements one gradient descent step (line~4 of \cref{alg:M-SoBiRL}) in each outer iteration. Furthermore, elements of $Q_k$ are assembled to estimate the soft value function:
\begin{equation}\label{eq:V_from_Q}
    \!\forall s \in \mathcal{S}:\  V_k\left( s \right) =\tau \log \left(\sum\nolimits_a{\exp \left( Q_k\left( s,a \right) /\tau \right)} \right),\!
\end{equation}
which complies with the form of~\eqref{eq:comply_V}. Finally, collecting the iterates $\left(x_k,\pi_k,V_k,w_{k}\right)$ yields the model-based hyper-gradient estimator $\widehat{\nabla}\phi$, i.e.,
\begin{small}
\begin{equation*}
    \begin{aligned}
         \!\widehat{\nabla} \phi \left(x_k,\pi_k,V_k,w_{k}\right):=&\,\nabla_x f\left(x_k, \pi_k\right) \!+\! \frac{1}{\tau} {\nabla r(x_k)}^\top\!\operatorname{diag}
         (\pi_k)
         \\
         &\times\!\nabla_\pi f(x_k,\pi_k)\!-\!\frac{1}{\tau}\nabla_{x}\varphi(x_k,V_k)^\top\!w_k.
    \end{aligned}
\end{equation*}    
\end{small}

\section{MODEL-FREE SOFT BILEVEL REINFORCEMENT LEARNING}\label{sec:SoBiRL}
In many real-world applications, agents lack access to the accurate model $P$ of the environment, underscoring the need for efficient model-free algorithms \citep{schulman2015TRPO,schulman2016GAE,lillicrap2016DDPG,mnih2016A3C,schulman2017PPO,haarnoja2018sac}. When $P$ is unknown, two obstacles appear from the hyper-gradient computation~\eqref{eq:exact_hypergrad2}. Specifically, it explicitly relies on the black-box transition matrix~$P$ to compute $\nabla V^*(x)$, i.e.,
\begin{equation}\label{eq:nabla_V_star}
    \nabla V^*(x) = \kh{I-\gamma P^{\pi^*(x)}}^{-1}\nabla_{x}\varphi(x,V^*(x)),
\end{equation}
and the computation involves complicated large-scale matrix multiplications. To circumvent them, we demonstrate how to derive the hyper-gradient in an expectation, which allows us to estimate $\nabla \phi(x)$ via sampling fully first-order information. Following these developments, we offer a model-free soft bilevel reinforcement learning algorithm.

In the model-free scenario, we concentrate on the upper-level function
\begin{equation}\label{eq:upper_level_f}
    f\left( x,\pi \right) =\mathbb{E} _{d_i\sim \rho \left( d;\pi  \right)}\left[ l\left( d_1,d_2,\ldots,d_I;x \right) \right],
\end{equation}
where each trajectory $d_i=\{\kh{s_h^i,a_h^i}\}_{h=0}^{H-1}$ is independently sampled from the trajectory distribution~$\rho\kh{d;\pi}$ generated by the policy $\pi$ in the upper-level~$\bar{\mathcal{M}}_{\bar{\tau}}$, and $l(\cdot;x)$ associated with~$x$ is a function of trajectories. Notice that it incorporates the applications in \cref{sec:biRL_app}, with $I=1$, $H=1$, $l=-Q^{\pi^*(x)}_{\bar{\mdp}_{\bar{\tau}}}\kh{s_0^1,a_0^1}$ for reward shaping, and $I=2$, finite $H$, $l=\mathbb{E}_y\zkh{l_h\kh{d_1,d_2,y;x}}$ for RLHF.

We extend \cref{pro:gradient_V} to absorb the transition matrix $P$ in \eqref{eq:nabla_V_star} into an expectation, with the result that the implicit differentiations $\nabla V^*(x)$ and $\nabla Q^*(x)$ can be estimated by sampling the reward gradient under the policy $\pi^*(x)$.
\begin{proposition}
    For any $x\in\mathbb{R}^n$, $s\in\mdps$, and $a\in\mdpa$, 
    \begin{align*}
         \nabla V_s^*(x) &= \mathbb{E}\big[ \sum_{t=0}^\infty \gamma^t\nabla r_{s_t,a_t}(x)\ \big |\ s_0=s, \pi^*(x)\big],
         \\
          \nabla Q_{sa}^*(x) &= \mathbb{E}\big[ \sum_{t=0}^\infty \gamma^t\nabla r_{s_t,a_t}(x)\ \big |\ s_0\!=\!s,a_0\!=\!a,\pi^*(x) \big].
    \end{align*}       
\end{proposition}
Nevertheless, it is still intractable to construct the ${\abs{\mdps}\times n}$ matrix $\nabla V^*(x)$ based on the above element-wise calculation, not to mention the large-scale matrix multiplications in \eqref{eq:exact_hypergrad2}. To bypass these matrix computations, we resort to the ``log probability trick'' \citep{sutton2018reinforcement} and the consistency condition~\eqref{eq:consis_policy}. The subsequent proposition confirms that we can evaluate the hyper-gradient by interacting with the environment and collecting fully first-order information. 
\begin{proposition}[Hyper-gradient]\label{pro:modelfree_hypergrad_chara}
    The hyper-gradient $\nabla \phi(x)$ of the bilevel reinforcement learning problem \eqref{eq:standar_biRL} with upper-level function~\eqref{eq:upper_level_f} can be computed by
    \begin{align}
        \nabla \phi(x) =& \ \mathbb{E} _{d_i\sim \rho \left( d;\pi ^*\left( x \right) \right)}\left[ \nabla l\left( d_1,d_2,\ldots,d_I;x \right) \right]   \label{eq:mf_hypergrad1}
        \\
        &+\tau ^{-1}\mathbb{E} _{d_i\sim \rho \left( d;\pi ^*\left( x \right) \right)} \Big [ l\left( d_1,d_2,\ldots,d_I;x \right) \nonumber
        \\
        & \Big( \sum_i{\sum_h{\nabla \left( Q_{s_{h}^{i}a_{h}^{i}}^{*}\left( x \right) -V_{s_{h}^{i}}^{*}\left( x \right) \right)}} \Big ) \Big].     \label{eq:mf_hypergrad2}
    \end{align}
\end{proposition}
Essentially, the first term \eqref{eq:mf_hypergrad1} in the hyper-gradient contributes to decreasing the upper-level function, and the second term~\eqref{eq:mf_hypergrad2} aggregates the gradient information transmitted from the lower level. Assembling these two directions constructs the hyper-gradient, and updating the upper-level variable $x$ along $-\nabla \phi(x)$ will lead to a first-order stationary point (see \cref{the:sobirl}). 

In line with the above propositions, to facilitate the $k$-th (inexact) hyper-gradient descent step, the construction of a hyper-gradient estimator $\widetilde{\nabla}\phi(x_k,\pi_k)$ is devided into two steps: 1)~evaluating the implicit differentiations based on $\pi_k$,
    \begin{align*}
         \widetilde{\nabla} V_s(x_k,\pi_k) &= \mathbb{E}\big[ \sum_{t=0}^\infty \gamma^t\nabla r_{s_t,a_t}(x_k)\,\big |\,s_0=s, \pi_k\big],
         \\
          \widetilde{\nabla} Q_{sa}(x_k,\pi_k) &= \mathbb{E}\big[ \sum_{t=0}^\infty \gamma^t\nabla r_{s_t,a_t}(x_k)\,\big |\,s_0\!=\!s,a_0\!=\!a,\pi_k\big];
    \end{align*}  
2)~absorbing these components into the upper-level sampling process induced by $\pi_k$,
\begin{align}
    \widetilde{\nabla} \phi(x_k,\pi_k) =&\ \mathbb{E} _{d_i\sim \rho \left( d;\pi_k \right)}\left[ \nabla l\left( d_1,d_2,\ldots,d_I;x_k \right) \right]  \label{eq:free_hyper_esti}
    \\
    &+\tau ^{-1}\mathbb{E} _{d_i\sim \rho \left( d;\pi_k \right)}\Big[ l\left( d_1,d_2,\ldots,d_I;x_k \right) \nonumber
    \\
    & \Big( \sum_i{\sum_h{\widetilde{\nabla} \left( Q_{s_{h}^{i}a_{h}^{i}}-V_{s_{h}^{i}}\right)\kh{x_k,\pi_k}}} \Big) \Big].     \nonumber
\end{align}
Consequently, we propose a model-free soft bilevel reinforcement learning algorithm, called SoBiRL, which is outlined in~\cref{alg:SoBiRL}. In the $k$-th outer iteration, we search for an approximate optimal soft policy $\pi_k$ satisfying ${\norm{{\pi}_k-\pi^*(x_k)}^2_2\le \epsilon}$, which can be achieved by executing $\mathcal{O}\kh{\log\epsilon^{-1}}$ iterations of the policy mirror descent algorithm~\citep{lan2023mirror,zhan2023mirror}. Then, we utilize $\widetilde{\nabla}\phi(x_k,\pi_k)$, whose estimation only involves first-order oracles, for an inexact hyper-gradient descent~step.
\begin{algorithm}[htbp]
    \caption{Model-Free Soft Bilevel Reinforcement Learning Algorithm (SoBiRL)}
    \label{alg:SoBiRL}
    \begin{algorithmic}[1]
        \REQUIRE iteration number $K$, step size $\beta$, initialization $x_1, \pi_0$, required accuracy $\epsilon$
        \FOR{$k = 1, \dots, K$}
            \STATE Solve the lower-level problem with the initialization $\pi_{k-1}$ to get ${\pi}_k$ with $\norm{{\pi}_k-\pi^*_k}_2^2\le \epsilon$\!\!\!\!\!\!\!
            \STATE Compute $\widetilde{\nabla}\phi(x_k,\pi_k)$
            \STATE Implement an inexact hyper-gradient descent step $x_{k+1} = x_k - \beta \widetilde{\nabla}\phi(x_k,\pi_k)$
        \ENDFOR
        \ENSURE $(x_{K+1},\pi_{K+1})$
    \end{algorithmic}
\end{algorithm}

\section{THE STOCHASTIC EXTENSION}\label{sec:ext_stocsobirl}
The preceding sections establish bilevel methods through the lens of deterministic optimization. However, in practical implementations of reinforcement learning, quantities are typically approximated via sampling data, e.g., $\widetilde{\nabla}\phi(x_k,\pi_k)$ defined in an expected form. As a result, the sampling process introduces stochasticity to the algorithm. Therefore, to enrich the developed framework SoBiRL both practically and theoretically, it is essential to extend it to stochastic settings. In detail, we approach this in two steps. (1)~Designing a scheme to estimate the hyper-gradient $\nabla\phi$ by sampling trajectories and then characterizing the resulting stochasticity, i.e., the bias and variance. (2)~Replacing the hyper-gradient in SoBiRL with its stochastic counterpart and incorporating a momentum technique to accommodate the stochasticity, leading to our stochastic variant, Stoc-SoBiRL.

In the stochastic scenario, we still focus on the problem formulation \eqref{eq:standar_biRL} with the upper-level function \eqref{eq:upper_level_f} as discussed in \cref{sec:SoBiRL}. Denote the trajectory tuple $(d_1,d_2,\ldots,d_I)$ as $\bm{d}$ for simplicity. In the $k$-th outer iteration, to estimate $\nabla\phi$, an expectation with respect to the random variable $\bm{d}$, we can generate $M$ independent tuples $\bm{d}^m:=(d^m_1,d^m_2,\ldots,d^m_I)$ for $m=1,2,\ldots,M$, evaluate corresponding quantities along each $\bm{d}^m$ and average them. To this end, it is necessary to tackle the terms $\nabla Q^*$ and $\nabla V^*$ in \eqref{eq:mf_hypergrad2}. Assuming access to a
generative model \citep{kearns1998finite,li2024generativemodel}, from any initial state-action pair $(s,a)$, we sample $J$ independent trajectories of length $T$ by implementing $\pi_k$, i.e., for $j=1,2,\ldots,J$,
\begin{small}
\begin{equation*}
\xi _{k}^{j}( s,a) :=\{( s_{0}^{j}=s,a_{0}^{j}=a);(s_{1}^{j},a_{1}^{j} ) ,\ldots,( s_{T-1}^{j},a_{T-1}^{j})\}.    
\end{equation*}
\end{small}
\!\!\!Collecting all these random variables yields $\bm{\xi}_k:=\{\xi _{k}^{j}\left( s,a \right) , j=1,\ldots,J,s\in \mathcal{S} ,a\in \mathcal{A} \}$, and the estimator for $\nabla Q^*$ can be constructed as follows.
\begin{equation*}
    \bar{\nabla}Q_{sa}^{\bm{\xi }_k}\left( x_k,\pi _k \right) :=\frac{1}{J}\sum_{j=1}^{J}{\sum_{t=0}^{T-1}{\gamma ^t\nabla r_{s_{t}^{j}a_{t}^{j}}\left( x_k \right)}}.
\end{equation*}
The random variable $\bm{\zeta}_k$ and the associated estimator $\bar{\nabla}V^{\bm{\zeta }_k}$ for $\nabla V^*$ are constructed similarly. Consequently, denoting $\bm{D}_k:=(\bm{d}_k^1,\bm{d}_k^2,\ldots,\bm{d}_k^M)$, we obtain the stochastic hyper-gradient:
\begin{align}
    &\bar{\nabla}\phi \left( \bm{D}_k,\bm{\xi }_k,\bm{\zeta }_k;x_k,\pi _k \right) =\frac{1}{M}\sum_{m=1}^M{\nabla l\left( \bm{d}_k^m;x_k \right)} +\frac{1}{\tau M} \nonumber
    \\
    &\times\!\sum_{m=1}^M{l\left( \bm{d}^m_k;x_k \right)}\sum_i\sum_h{\bar{\nabla}\!\left( Q_{s_{h}^{m,i}a_{h}^{m,i}}^{\bm{\xi }_k}-V_{s_{h}^{m,i}}^{\bm{\zeta }_k} \right)}( x_k,\pi _k ),   \nonumber
\end{align}
which is abbreviated as $\bar{\nabla}\phi _k$ in the following discussion. The sampling scheme is outlined in \cref{alg:samplingprocess}. 

Momentum techniques are known to be beneficial for reducing variance and accelerating algorithms \citep{cutkosky2019momentum}. To this end, we maintain a momentum-instructed $h_k$ in the $k$-th outer iteration,
\begin{equation*}
    h_k\! = \! \bar{\nabla}\phi_k+ (1-\mu_k)(h_{k-1}-\bar{\nabla}\phi(\bm{D}_k,\bm{\xi}_k,\bm{\zeta}_k;x_{k-1},\pi_{k-1})),
\end{equation*}
which tracks $\bar{\nabla}\phi_k$ via current and historical hyper-gradient estimates. Using $h_k$ to update the upper-level $x_k$, we obtain the stochastic \cref{alg:StocSoBiRL}, Stoc-SoBiRL.

\begin{figure*}[tbp]
\centering
\begin{minipage}{0.49\textwidth}
    \centering
    \includegraphics[width=1\linewidth]{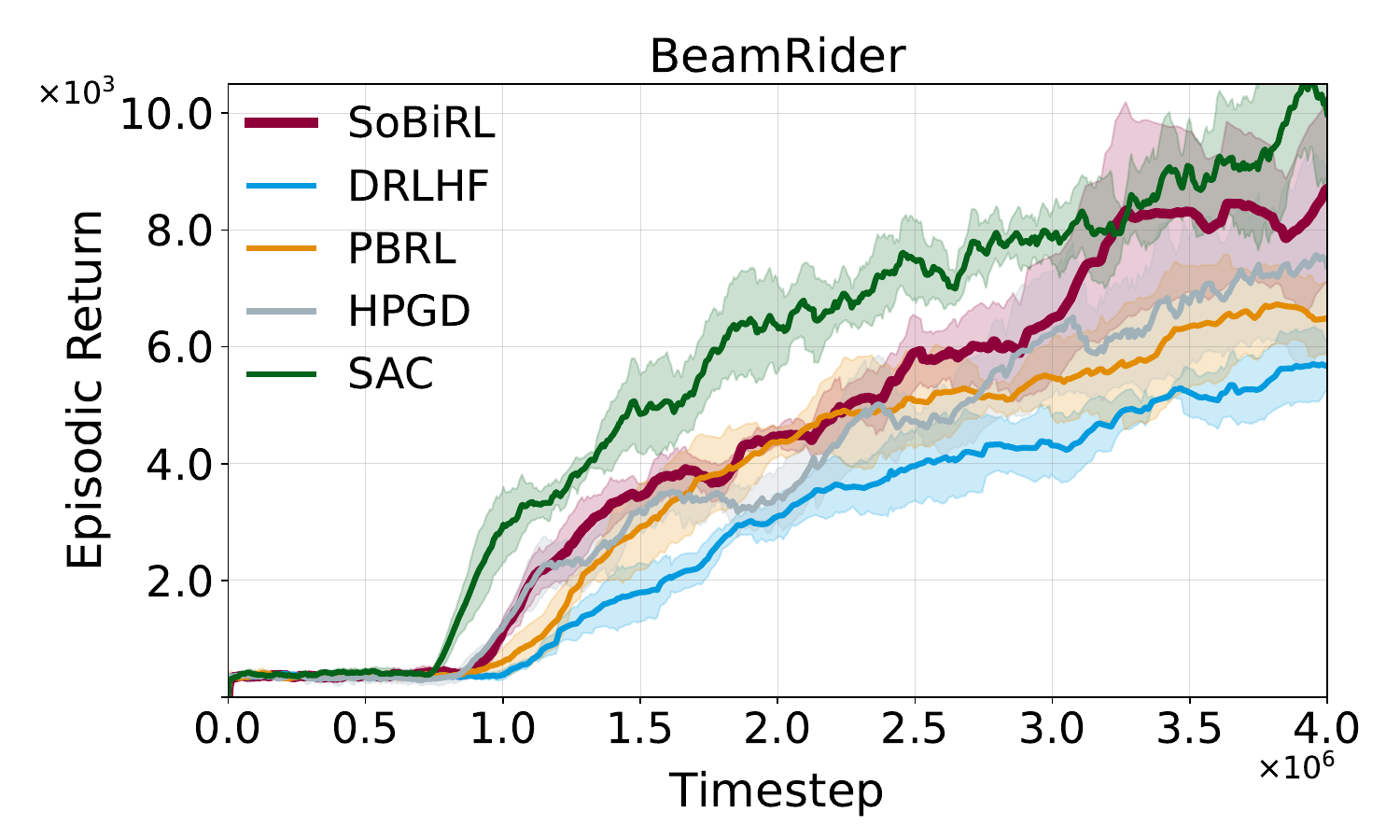}
\end{minipage}
\,
\begin{minipage}{0.49\textwidth}
    \centering
    \includegraphics[width=1\linewidth]{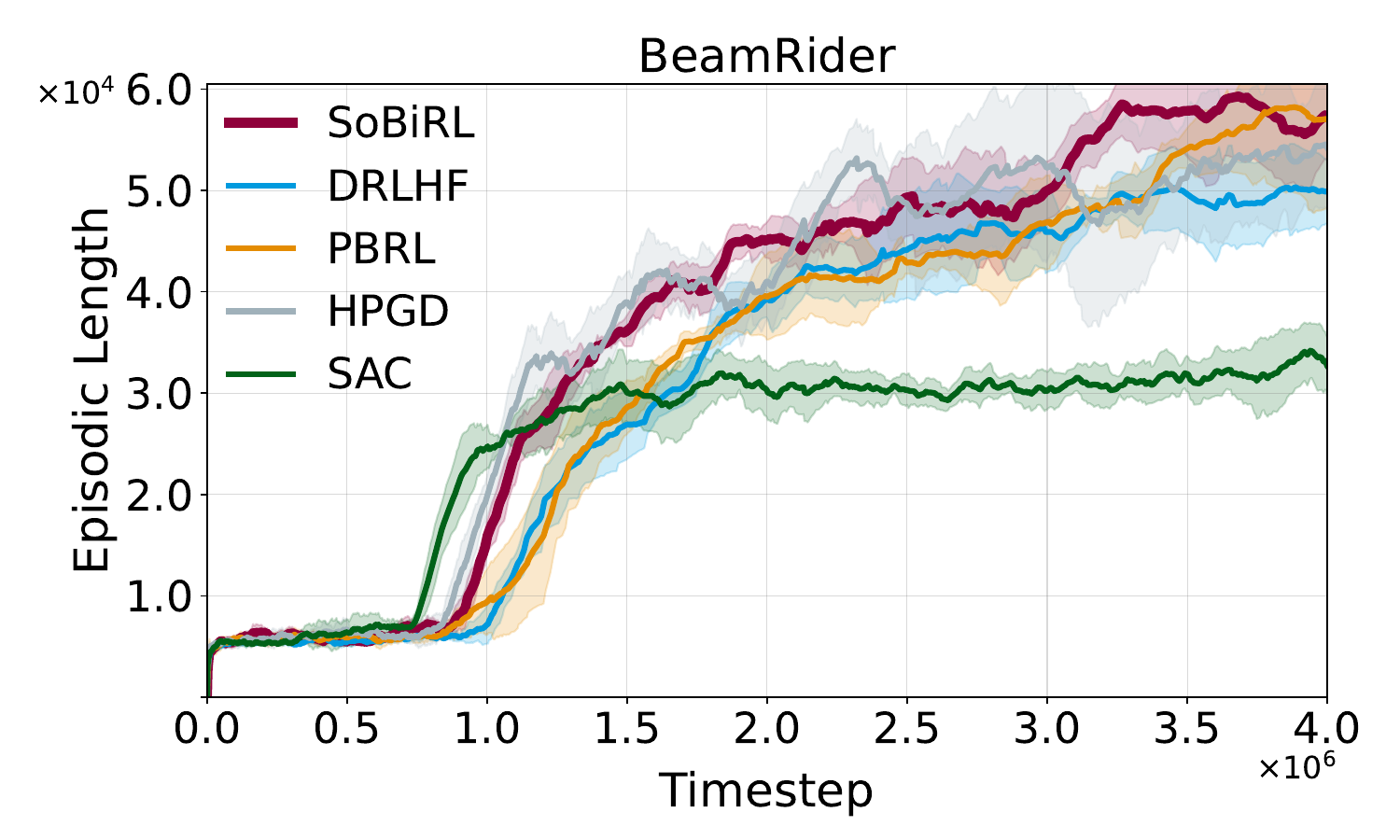}
\end{minipage}
\caption{Comparison of algorithms on the Atari game, BeamRider, evaluated by the ground-truth reward. Each bilevel algorithm collects a total of $3000$ trajectory pairs. The running average over~$15$ consecutive episodes is adopted for the presentation, and results are averaged over $5$ seeds.}
\label{fig:beamrider}
\end{figure*}

\section{THEORETICAL ANALYSIS}\label{sec:theory}
In this section, we prove global convergence and give the iteration and sample complexity for the proposed algorithms. The analysis requires only the regularity conditions---specifically, boundedness and Lipschitz continuity---of the first-order information. This distinguishes from general AID-based methods, which require additional smoothness assumptions regarding second-order derivatives \citep{ghadimi2018bsa,chen2021closing,khanduri2021sustain,dagreou2022soba,ji2022will}. The properties of $\varphi$, $\pi^*(x)$, and $V^*(x)$ are outlined in \cref{sec:general_smooth}. Moreover, it is worth emphasizing that the analysis for M-SoBiRL is different from that for SoBiRL. Thus, we organize their properties and the detailed proofs in \cref{sec:ana_MSoBiRL} and \cref{sec:ana_SoBiRL}. \todo{The analysis for Stoc-SoBiRL is referred to \cref{sec:stocsobirl}.}

\begin{assumption}\label{assu:f} 
    In the model-based scenario \eqref{eq:standar_biRL}, $f$ is continuously differentiable. The gradient~$\nabla f(x,\pi)$ is $L_{f}$-Lipschitz continuous, i.e., for any {$x_1,x_2\in\mathbb{R}^n,\pi_1,\pi_2\in\Delta^{\abs{\mdps}}$}, $\norm{\nabla f(x_1,\pi_1) - \nabla f(x_2,\pi_2)}_2\le L_f\kh{\norm{x_1-x_2}_2 + \norm{\pi_1-\pi_2}_\infty}$. \!\!Additionally, the boundedness condition holds, $\norm{\nabla_\pi f(x,\pi^*(x))}_2\!\leq\! C_{f\pi}$.
\end{assumption}

\begin{assumption}\label{assu:l} 
    In the model-free scenario with upper-level function~\eqref{eq:upper_level_f}, denote $\bm{d}=(d_1,d_2,\ldots,d_I)$. $l(\bm{d};x)$ parameterized by $x$ is a function of $I$ trajectories,  each consisting of $H$ steps. It is bounded, i.e., $\abs{l\kh{\bm{d};x}}\le C_l$, and is continuously differentiable with respect to $x$. $l(\bm{d};x)$ is $L_{l}$-Lipschitz continuous, i.e., for any {$x_1,x_2\in\mathbb{R}^n$} and $\bm{d}$, $\norm{l(\bm{d};x_1) - l(\bm{d};x_2) }_2 \le L_l\| x_1-x_2\|_2.$ Moreover, $\nabla l(\bm{d};x)$ is $L_{l1}$-Lipschitz continuous, i.e., for any {$x_1,x_2\in\mathbb{R}^n$} and $\bm{d}$, $\norm{\nabla l(\bm{d};x_1)-\nabla l(\bm{d};x_2) }_2\le L_{l1}\norm{x_1\!-\!x_2}_2$.
\end{assumption}
Assumptions~\ref{assu:f} and~\ref{assu:l} specify the requirements for the upper-level problems in model-based and model-free scenarios, respectively. Generally, the boundedness and Lipschitz continuity assumptions related to $f$ are standard in bilevel optimization \citep{ghadimi2018bsa,ji2021stocbio,arbel2022amortized,chen2022stable,li2022fsla}.
\begin{assumption}\label{assu:r}
    The reward is bounded by $C_r$, i.e., $\abs{r_{sa}(x)}\le C_r$, and the gradient is bounded by $C_{rx}$, i.e., $\norm{\nabla r_{sa}(x)}_2\le C_{rx}$. Additionally, $r_{sa}(x)$~is $L_r$-Lipschitz smooth, i.e., for any $x_1,x_2\in\mathbb{R}^n$ and $(s,a)\in\mdps\times\mdpa$, $\norm{\nabla r_{sa}(x_1) - \nabla r_{sa}(x_2)}_2\le L_r\norm{x_1-x_2}_2$.
\end{assumption}
This assumption on the parameterized reward function is common in RL \citep{zhang2020globalPG,lan2022rewardgrad,chakraborty2024parl,shen2024bilevelRL}.

In the model-based scenario, the following theorem reveals that the proposed algorithm M-SoBiRL benefits from amortizing the hyper-gradient approximation in which it enjoys the $\mathcal{O}\kh{\epsilon^{-1}}$ convergence rate, and the inner iteration number $N$ can be selected as a constant independent of the accuracy~$\epsilon$. 
\begin{theorem}[Model-based]\label{the:M_sobirl}
    Under Assumptions~\ref{assu:f} and~\ref{assu:r}, in M-SoBiRL, we can choose constant step sizes~$\beta,\eta$, and the inner iteration number $N\sim\mathcal{O}(1)$. Then the iterates $\{x_k\}$~satisfy
    \begin{equation*}
        \frac{1}{K}\sum_{k=1}^K \norm{\nabla \phi(x_k)}_2^2 = \mathcal{O}\kh{K^{-1}}.
    \end{equation*}
    Detailed parameter setting is referred to \cref{the:proof_MSoBiRL}.
\end{theorem}

Coupled with \cref{pro:modelfree_hypergrad_chara}, the next proposition marks a substantial development to clarify the Lipschitz property of the hyper-gradient.
\begin{proposition}
    Under~\cref{assu:r}, we have $\norm{\nabla \phi(x_1) - \nabla \phi(x_2)}_2\le L_{\phi}\norm{x_1-x_2}_2$ for any $x_1,x_2\in\mathbb{R}^n$, with $L_\phi$ specified in \cref{pro:hypergrad_Lipz}.
\end{proposition}

\begin{theorem}[Model-free]\label{the:sobirl}
    Under Assumptions~\ref{assu:l} and~\ref{assu:r} and given the accuracy $\epsilon$ in \cref{alg:SoBiRL}, we can set the constant~$\beta < \frac{1}{2L_\phi}$, and then the iterates $\{x_k\}$~satisfy
    \begin{equation*}
        \frac{1}{K}\sum_{k=1}^K{\left\| \nabla \phi \left( x_k \right) \right\|_2 ^2}\le \frac{\phi \left( x_1 \right) -\phi ^*}{K\left( \frac{\beta}{2}-\beta ^2L_{\phi} \right)}+\frac{1+2\beta L_{\phi}}{1-2\beta L_{\phi}}L^2_{\widetilde{\phi }}\ \epsilon\,,
    \end{equation*}
    where $\phi^*$ is the minimum of $\phi(x)$, and $L_\phi,\ L_{\widetilde{\phi }}$ are specified in Propositions~\ref{pro:hypergrad_Lipz} and \ref{pro:lipz_hp_estimator}, {respectively}.
\end{theorem}
Combined with~\cref{tab:comparision_BiRL}, the convergence property in~\cref{the:sobirl} implies that SoBiRL realizes better iteration complexity than PBRL \citep{shen2023penalty} with the same computation cost at each outer and inner iteration. Moreover, SoBiRL attains the convergence rate of the same order as PARL \citep{chakraborty2024parl}, but only employs first-order oracles and gets rid of the convexity assumption on the lower-level problem.

\todo{In the context of stochastic bilevel RL, the distribution of samples $(\bm{D}_k,\bm{\xi}_k,\bm{\zeta}_k)$ relies on the variable $\pi_k$, which differs from related work employing momentum techniques \citep{cutkosky2019momentum,khanduri2021sustain,huang2023momentum}. To circumvent this misalignment in Stoc-SoBiRL, we give the statistical properties (bias and variance of the stochastic hyper-gradient) in \cref{sec:stoc_estimator} and the convergence analysis in \cref{sec:conver_stocsobirl}.}
\begin{theorem}[Stochastic]
    Under Assumptions~\ref{assu:l} and~\ref{assu:r}, we can choose appropriate sampling configurations $M\sim\mathcal{O}(K^{4/3}),\,J\sim\mathcal{O}(1),\,T\sim\mathcal{O}(\log K)$.
    Then the iterates $\{x_k\}$ generated by \cref{alg:StocSoBiRL} satisfy
    \begin{equation*}
        \frac{1}{K}\sum_{k=1}^K{\mathbb{E} \left[ \left\| \nabla \phi (x_k) \right\|_2 ^2 \right]}=\mathcal{O} \left( \frac{\log K}{K^{2/3}} \right) =\widetilde{\mathcal{O} }\left( K^{-\frac{2}{3}} \right).
    \end{equation*}
     Detailed parameter setting is given in \cref{the:convergent_StocSoBiRL}.
\end{theorem}
Enlightened by the momentum acceleration, the average norm square of the hyper-gradient achieves $\epsilon$ within $\widetilde{\mathcal{O}}(\epsilon^{-1.5})$ outer iterations. Therefore, the samples required by the upper-level leader are of the order $\widetilde{\mathcal{O}}(\epsilon^{-3.5})$, which can be alleviated in practice, as the leader can leverage vectorized environments \citep{makoviychuk2isaac} within each outer iteration; \revise{see \cref{sec:distributed} for further discussion.}

\section{EXPERIMENTS}\label{sec:exp}
\vspace{-0.3cm}
In this section, we conduct RLHF experiments on Atari games and Mujoco environments to validate the efficiency of SoBiRL, and one synthetic experiment to verify the convergence of M-SoBiRL. The experiment details and parameter settings are provided in \cref{sec:details_exp}. We have made the code available on \revise{\href{https://github.com/UCAS-YanYang/SoBiRL}{https://github.com/UCAS-YanYang/SoBiRL}}.

\begin{table*}[t]
\caption{Comparison of algorithms on the Atari games and Mujoco simulations, evaluated by the ground-truth reward. Each bilevel algorithm collects a total of 3000 trajectory pairs. The results are recorded after appropriate timesteps and averaged over 5 seeds.}
\label{tab:overall_exp}
\begin{center}
\begin{footnotesize}
\setlength{\tabcolsep}{2.5pt}
\begin{tabular}{lcccccc}
\toprule 
Algorithm  & BeamRider (4M )  & Seaquest  (1M ) & SpaceInvaders (1M ) & HalfCheetah (1M ) & Walker2d (1M ) & Hopper (500k )  \\
\midrule
SoBiRL      & $ 8459.7\pm1989.1$ &  $ 746.7\pm16.4$  & $595.6\pm23.8$           & $ 4283.6\pm363.0$   & $ 2996.0\pm\phantom{0}97.4$ & $ 2812.4\pm217.5$ \\
DRLHF            &   $6124.5\pm\phantom{0}433.1$    &    $561.9\pm38.0$    & $483.9\pm27.3$           & $3466.7\pm424.2$       & $2389.4\pm101.8$    & $2455.8\pm457.2$     \\
PBRL    &   $6798.2\pm\phantom{0}784.0$    &    $651.2\pm24.6$    & $554.0\pm33.4$           & $4022.8\pm311.8$       & $2257.8\pm127.9$    & $2799.1\pm155.5$     \\
HPGD             &   $7503.9\pm1863.6$   &    $728.1\pm57.5$    & $ 601.8\pm36.0$       & $3975.8\pm476.1$       & $2487.5\pm223.8$    & $2803.3\pm354.5$     \\
SAC &   $9780.0\pm1877.4$   &    $891.1\pm31.3$    & $783.8\pm56.7$           & $5769.9\pm765.3$       & $4552.5\pm\phantom{0}23.3$     & $2207.6\pm\phantom{0}18.4$      \\
\bottomrule
\end{tabular}
\end{footnotesize}
\end{center}
\end{table*}

In RLHF experiments, we compare SoBiRL \revise{with the bilevel algorithms DRLHF \citep{christiano2017rlhf}, PBRL \citep{shen2024bilevelRL}, HPGD \cite{thoma2024HPGD}, and a baseline algorithm SAC \citep{haarnoja2018sac,haarnoja2018sacauto}. All the bilevel solvers harness deep neural networks to predict rewards, while the baseline SAC receives ground-truth rewards for training.} The results are reported in Figures \ref{fig:beamrider}, \ref{fig:rlhf2}, \ref{fig:Mujoco}, \revise{and \cref{tab:overall_exp}}, showing that even in the face of unknown rewards, the performance of SoBiRL is on a par with the baseline SAC. Additionally, SoBiRL achieves a higher episodic return than the compared methods within the same timesteps. Another interesting observation is that, taking BeamRider for example, compared to the ground truth, the preference prediction based on the reward model of DRLHF achieves an accuracy of approximately $85\%$, while it hovers around $54\%$ with the reward model of SoBiRL. This stems from different update rules. Specifically, DRLHF alternates between learning a policy given the parameterized reward $r(x)$ and minimizing~$\mathbb{E}\zkh{l_h}$ based on collecting trajectories, which decouples the task into two separate phases. It aligns preferences, but overlooks the exploration for rewards promoting better policy. Nevertheless, SoBiRL glues the two levels via the hyper-gradient ${\nabla}{\phi}$, with the first term~\eqref{eq:mf_hypergrad1} exploiting the preference information to align the reward predictor and the second term~\eqref{eq:mf_hypergrad2} exploring the implicit reward to unearth a more favorable policy. Therefore, SoBiRL produces superior results in reward prediction, although it exhibits lower accuracy in alignment of trajectory preference.

Regarding M-SoBiRL, the results of a synthetic experiment are shown in \cref{fig:test_M}. The curves exhibit the benign convergence property of the proposed algorithm.

\subsection*{Acknowledgments}
Bin Gao was supported by the Young Elite Scientist Sponsorship Program by CAST. Ya-xiang Yuan was supported by the National Natural Science Foundation of China (grant No. 12288201). The authors are grateful to the Program and Area Chairs and Reviewers for their detailed and valuable comments and suggestions.

\bibliography{lib}

\newpage
\onecolumn
\aistatstitle{Bilevel Reinforcement Learning via the Development of Hyper-gradient without Lower-Level Convexity: \\ Supplementary Materials}

\appendix

\section{RELATED WORK IN BILEVEL OPTIMIZATION}\label{sec:related_bio}
By the implicit function theorem, the approximate implicit differentiation (AID) based method regards the optimal lower-level solution as~a function of the upper-level variable, which yields the hyper-gradient to instruct the upper-level update. It implements alternating (inexact) gradient descent steps between two levels~\citep{ghadimi2018bsa,ji2021stocbio,chen2022stable,dagreou2022soba,li2022fsla,hong2023ttsa}. In addition, generalizing the lower-level strong convexity, the studies~\citep{grazzi2020FPBiO,grazzi2021stochasticFPBiO,grazzi2023FPBiO} focused on the bilevel problem where the lower-level problem is formulated as a fixed-point equation. Moreover, a line of research has been dedicated to the bilevel problem with a nonconvex lower-level objective~\citep{liu2021towards,shen2023penalty,huang2023momentum,lu2023slm,liu2024moreau}. Specifically, \citet{shen2023penalty}~proposed a penalty-based method to bypass the requirement of lower-level convexity. In~\citep{huang2023momentum}, a momentum-based approach was designed to solve the bilevel problem with a lower-level function satisfying the PL condition. \citet{liu2024moreau}~utilized the Moreau envelope based reformulation to provide a single-loop algorithm for nonconvex-nonconvex bilevel optimization. In stochastic settings, bilevel algorithms have been developed in \citep{ghadimi2018bsa,ji2021stocbio,khanduri2021sustain,chen2022stable}.
 
\section{VECTOR AND MATRIX NOTATION}\label{sec:notations}
Throughout this paper, we adhere to the following conventions of vector and matrix notation.
\begin{itemize}[label=$\bullet$,leftmargin=*]
  \item For any vector $q\in\mathbb{R}^{\abs{\mdps}\abs{\mdpa}}$, we denote its component associated with $s\in\mdps$ as $q_s\in\mathbb{R}^{\abs{\mdpa}}$ or $q(s,\cdot)\in\mathbb{R}^{\abs{\mdpa}}$ in the following way,
    \begin{equation*}
        \forall\,(s,a)\in\mdps\times\mdpa,\quad q_s(a)=q(s,a)=q_{sa}.
    \end{equation*}
  Additionally, for a vector $v\in\mathbb{R}^{\abs{\mdps}}$, we use the notation of entry $v_s=v(s)$ interchangeably for convenience of exposition.
  \item $P\in\mathbb{R}^{\abs{\mdps}\abs{\mdpa}\times\abs{\mdps}}$, $r\in\mathbb{R}^{\abs{\mdps}\abs{\mdpa}}$, $\pi\in\mathbb{R}^{\abs{\mdps}\abs{\mdpa}}$: the transition matrix, the reward function, and the policy.
  \item $P^\pi\in\mathbb{R}^{\abs{\mdps}\times\abs{\mdps}}$: the transition matrix induced by~$\pi$. 
  \item $V^\pi\in\mathbb{R}^{\abs{\mdps}}$, $Q^\pi\in\mathbb{R}^{\abs{\mdps}\abs{\mdpa}}$: the soft value functions  associated with~$\pi$.
  \item $V^*(x)\in\mathbb{R}^{\abs{\mdps}}$, $Q^*(x)\in\mathbb{R}^{\abs{\mdps}\abs{\mdpa}}$, $\pi^*(x)\in\mathbb{R}^{\abs{\mdps}\abs{\mdpa}}$: the optimal soft value functions and the optimal soft policy given $r(x)$.
  \item $\nabla\phi(x)\in\mathbb{R}^n$: the hyper-gradient.
  \item $\nabla_x \varphi(x,v) \in\mathbb{R}^{\abs{\mdps}\times n}$, $\nabla_v \varphi(x,v) \in\mathbb{R}^{\abs{\mdps}\times \abs{\mdps}}$: the partial derivatives of $\varphi(x,v)$.
  \item $\nabla_x f(x,\pi) \in\mathbb{R}^n$, $\nabla_\pi f(x,\pi) \in\mathbb{R}^{\abs{\mdps}\abs{\mdpa}}$: the partial derivatives of $f(x,\pi)$.
  \item $\nabla r(x)\in\mathbb{R}^{\abs{\mdps}\abs{\mdpa}\times n}$: the derivative of $r(x)$. 
  \item $\nabla V^*(x)\in\mathbb{R}^{\abs{\mdps}\times n}$, $\nabla Q^*(x)\in\mathbb{R}^{\abs{\mdps}\abs{\mdpa}\times n}$, $\nabla\pi^*(x)\in\mathbb{R}^{\abs{\mdps}\abs{\mdpa}\times n}$: the implicit differentiations.
\end{itemize}

\section{PRELIMINARIES ON NUMERICAL LINEAR ALGEBRA}
\revise{We first revisit some basics for non-negative matrices and Markov chains; see \citep{seneta2006nonnegative}.}
\begin{proposition}\label{prop:I_gamma}
    Given a matrix $B\in\mathbb{R}^{m\times m}$ satisfying $\norm{B}_{\infty}\le\gamma$, then the matrix $I-B$ is invertible with the magnitudes of all its eigenvalues located within the interval $[1-\gamma,1+\gamma]$.
\end{proposition}
\begin{proof}
Denote the entries of $B$ and $I-B$ by $\{b_{ij}\}_{i.j=1}^m$ and $\{l_{ij}\}_{i.j=1}^m$, respectively. Set $\lambda\in\mathbb{C}$ as an eigenvalue of $I-B$, and consider the Gershgorin circle theorem, for $i=1,2,\ldots,m$,
\begin{equation}\label{eq:yuanpan}
    \abs{\lambda-l_{ii}}\le \sum_{j=1,j\neq i}^m \abs{l_{ij}}.
\end{equation}
Incorporating $l_{ii}=1- b_{ii}$ and $l_{ij}=- b_{ij}$ into \eqref{eq:yuanpan}, we obtain
\begin{equation*}
    \abs {\lambda - \kh{1- b_{ii}}} \le \sum_{j=1,j\neq i}^m \abs{b_{ij}},
\end{equation*}
The absolute value inequality leads to
\begin{equation*}
    -\sum_{j=1,j\neq i}^m \abs{b_{ij}} \le \abs{\lambda - 1} - \abs{b_{ii}} \le \sum_{j=1,j\neq i}^m \abs{b_{ij}},
\end{equation*}
which completes the proof since
\begin{equation*}
    \max_{i=1,2,\ldots,m} \sum_{j=1}^m \abs{b_{ij}} = \norm{B}_{\infty} = \gamma.
\end{equation*}
\end{proof}

\begin{definition}\label{def:tranmatrix}
    A matrix $P\in\mathbb{R}^{m\times m}=(p_{ij})$ is defined to be a transition matrix if all entries are non-negative and each row sums to $1$, i.e., for any $i=1,2,\ldots,m$,
    \begin{equation*}
        \sum_{j} p_{ij} = 1.
    \end{equation*}
\end{definition}

\begin{proposition}\label{pro:tranmatrix}
A transition matrix $P\in\mathbb{R}^{m\times m}=(p_{ij})$ always has an eigenvalue of $1$, while all the remaining eigenvalues have absolute values less than or equal to $1$.
\end{proposition}

\begin{definition}
    A matrix $X=(x_{ij})\in\mathbb{R}^{m\times m}$ is called non-negative if all the entries are equal to or greater than zero, i.e.,
    \begin{equation*}
    x_{ij} \ge 0,\ \text{ for }1\le i,j\le m.
    \end{equation*}
\end{definition}

\begin{definition}
    A matrix $A\in\mathbb{R}^{m \times m}$ is an $M$-matrix if it can be expressed in the form $A=\mu I-B$, where $B=\left(b_{i j}\right)$ is non-negative, and $\mu\ge 0$ is equal to or greater than the moduli of any eigenvalue of $B$.
\end{definition}

\begin{proposition}\label{pro:inverse_nonnegative}
    Given a constant $\gamma\in(0,1)$ and a transition matrix $P\in\mathbb{R}^{m\times m}=(p_{ij})$, then the matrix $A=I-\gamma P$ is an $M$-matrix with the moduli of every eigenvalue in $[1-\gamma,1+\gamma]$, Additionally, the inverse of $A$ is
    \begin{equation*}
        A^{-1} = \sum_{t=0}^\infty \gamma^t P^t,
    \end{equation*}
    which is a non-negative matrix and holds all diagonal entries greater than or equal to $1$. 
\end{proposition}
\begin{proof}
    Combined with Definition~\ref{def:tranmatrix} which reveals $\norm{P}_{\infty} = 1$ and Proposition~\ref{prop:I_gamma}, it follows that any $\lambda\in\mathbb{C}$ as an eigenvalue of $A$ satisfies $\abs{\lambda}\in[1-\gamma,1+\gamma]$, yielding that $A$ is invertible. Since~${\rho(\gamma P)=\gamma<1}$ by Proposition~\ref{pro:tranmatrix}, the infinite series $\sum_{t=0}^\infty \gamma^t P^t$ converges and coincides with the inverse of $A$, that is,
    \begin{equation}\label{eq:inverse_series}
        A^{-1} = \sum_{t=0}^\infty \gamma^t P^t.
    \end{equation}
    Coupled with $P$ being non-negative and $\kh{\gamma P}^0=I$,  \eqref{eq:inverse_series}~implies that $A^{-1}$ is non-negative with all diagonal entries greater than or equal to $1$.
\end{proof}

The next proposition serves as a technical tool in the later analysis.
\begin{proposition}\label{pro:lognorm}
    For any $\pi_1,\pi_2\in \mathbb{R}_+^\mathbb{\abs{\mdps}\abs{\mdpa}}\cap\Delta^{\abs{\mdps}}$, the vector norms satisfy
    \begin{align*}
        \norm{\pi_1-\pi_2}_\infty&\le\norm{\log \pi_1-\log\pi_2}_\infty,
        \\
        \norm{\pi_1-\pi_2}_2&\le\norm{\log \pi_1-\log\pi_2}_2.
    \end{align*}
\end{proposition}
\begin{proof}
    It suffices to show for any $0<a,b<1$, $\abs{a-b}\le \abs{\log a-\log b}$. Without loss of generality, suppose that $0<b\le a<1$. Consider the function $G\left( z \right) :=z-\log z$, with the derivative
    \begin{equation*}
        G^{\prime}\left( z \right) =1-\frac{1}{z}<0,\quad\text{for}\ z\in(0,1),
    \end{equation*}
    which means $G(z)$ is monotonically decreasing on $(0,1)$. It leads to $a-\log a\le b-\log b$, equivalently, $a-b\le \log a-\log b$.
\end{proof}

\section{PROOF IN SECTION~\ref{sec:MSoBiRL}}
The structure of $\varphi(\cdot,\cdot)$ facilitates the characterization of $\nabla V ^*(x)$, as outlined in the following proposition.
\begin{proposition}
    For any $x\in\mathbb{R}^n$, $\varphi(x,\cdot)$ is a contraction map, i.e., $\|\nabla_v \varphi(x,v)\|_{\infty}=\gamma<1$, and the matrix ${I-\nabla_{v}\varphi(x,v)}$ is invertible. Consequently, $V^*(x)$ is the unique fixed point of $\varphi(x,\cdot)$, with a well-defined derivative $\nabla V^*(x)$, given by
    \begin{equation}\label{eq:hyper_result1}
         \nabla V^*(x) = \kh{I-\nabla_{v}\varphi(x,V^*(x))}^{-1}\nabla_{x}\varphi(x,V^*(x)).
    \end{equation}    
    Additionally, $\nabla_v \varphi\kh{x,V^*(x)}$ coincides with the $\gamma$-scaled transition matrix induced by the optimal soft policy $\pi^*(x)$, i.e.,
    \begin{equation}\label{eq:hyper_result2}
        \nabla_{v}\varphi\kh{x,V^*(x)} = \gamma P^{\pi^*(x)}.
    \end{equation}    
\end{proposition}
\begin{proof}
    Consider the mapping $F:\mathbb{R}^n\times\mathbb{R}^{\abs{\mdps}}\to\mathbb{R}^{\abs{\mdps}}$ defined by
    \begin{equation*}
        F(x,v):= v - \varphi(x,v).
    \end{equation*}
    Then the partial derivative of $F(x,v)$ with respect to $v$ reads $I-\nabla_v\varphi(x,v)$. Computing
    \begin{equation}\label{eq:phi_grad_v}
        \left( \nabla _v\varphi \left( x,v \right) \right) _{ss^{\prime}} = \frac{\partial \varphi _s\left( x,v \right)}{\partial v_{s^{\prime}}}=\gamma \frac{\sum_a{P_{sas^{\prime}}\exp \left( \tau ^{-1}\left( r_{sa}\left( x \right) +\gamma \sum_{s^{\prime\prime}}{P_{sas^{\prime\prime}}v_{s^{\prime\prime}}} \right) \right)}}{\sum_a{\exp \left( \tau ^{-1}\left( r_{sa}\left( x \right) +\gamma \sum_{s^{\prime\prime}}{P_{sas^{\prime\prime}}v_{s^{\prime\prime}}} \right) \right)}}
    \end{equation}
    leads to
    \begin{equation*}
    \norm{\nabla_v\varphi(x,v)}_\infty = \max_s \sum_{s^{\prime}}{\left| \left( \nabla _v\varphi \left( x,v \right) \right) _{ss^{\prime}} \right|} = \max_s \sum_{s^\prime} \frac{\gamma \sum_a P _{a s s^{\prime}} \exp \left(\tau^{-1}\left(r^x_{s a}+\gamma \sum_{s^{\prime \prime}} P _{a s s^{\prime \prime}} v_{s^{\prime \prime}}\right)\right)}{\sum_a \exp \left(\tau^{-1}\left(r^x_{s a}+\gamma \sum_{s^{\prime \prime}} P _{a s s^{\prime \prime}} v_{s^{\prime \prime}}\right)\right)} = \gamma,
    \end{equation*}
    which means for any $x\in\mathbb{R}^n$, $\varphi(x,\cdot)$ is a contraction mapping, admiting a unique $V^*(x)$ such that $F(x,V^*(x))=0$. Additionally, by applying \cref{prop:I_gamma}, we can derive that $\nabla_v F(x,v)$ is invertible for any $(x,v)\in \mathbb{R}^n\times\mathbb{R}^{\abs{\mdps}}$. Consequently, the implicit function theorem implies that there exists a differentiable function $V^*:\mathbb{R}^n\to\mathbb{R}^{\abs{\mdps}}$ such that, for any $v\in\mathbb{R}^{\abs{\mdps}}$ and $x\in\mathbb{R}^n$,
    \begin{equation*}
        F(x,V^*(x)) = V^*(x) - \varphi(x,V^*(x)) = {\bf 0},
    \end{equation*}
    with the derivative $\nabla V^*(x)$ satisfying
    \begin{equation*}
        \nabla V^*(x) - \nabla_x\varphi(x,V^*(x)) - \nabla_v \varphi(x,V^*(x)) \nabla V^*(x) = {\bf 0}.
    \end{equation*}
    Rearranging this yields the result \eqref{eq:hyper_result1}. On the other hand, take \eqref{eq:phi_grad_v} at $\kh{x,V^*(x)}$,
    \begin{align}
        \left( \nabla _v\varphi \left( x,v \right) \right) _{ss^{\prime}}|_{v=V^*\left( x \right)}&=\gamma \frac{\sum_a{P_{sas^{\prime}}\exp \left( \tau ^{-1}\left( r_{sa}\left( x \right) +\gamma \sum_{s^{\prime\prime}}{P_{sas^{\prime\prime}}V_{s^{\prime\prime}}^{*}\left( x \right)} \right) \right)}}{\sum_a{\exp \left( \tau ^{-1}\left( r_{sa}\left( x \right) +\gamma \sum_{s^{\prime\prime}}{P_{sas^{\prime\prime}}V_{s^{\prime\prime}}^{*}\left( x \right)} \right) \right)}}  \nonumber
        \\
        &=\gamma \frac{\sum_a{P_{sas^{\prime}}\exp \left( \tau ^{-1}Q_{sa}^{*}\left( x \right) \right)}}{\exp \left( \tau ^{-1}V_{s}^{*}\left( x \right) \right)} \label{eq:phi_grad_v_1}
        \\
        &=\gamma \sum_a{P_{sas^{\prime}}\exp \left( \tau ^{-1}\left( Q_{sa}^{*}\left( x \right) -V_{s}^{*}\left( x \right) \right) \right)}    \nonumber
        \\
        &=\gamma \sum_a{P_{sas^{\prime}}\pi _{sa}^{*}\left( x \right)}  \label{eq:phi_grad_v_2},
    \end{align}
    where \eqref{eq:phi_grad_v_1} comes from \eqref{eq:V_to_Q} and \eqref{eq:consis_v}, and \eqref{eq:phi_grad_v_2} follows considering \eqref{eq:consis_policy}. Consequently, revisiting the definition of $P^{\pi^*(x)}$ yields the conclusion $\nabla_{v}\varphi\kh{x,V^*(x)} = \gamma P^{\pi^*(x)}$.
\end{proof}

\section{PROOFS IN SECTION~\ref{sec:SoBiRL}}
In this section, we provide proofs of the propositions in \cref{sec:SoBiRL}. The following proposition indicates that the implicit differentiations $\nabla V ^*(x)$ and $\nabla Q ^*(x)$ can be estimated by sampling the reward gradient under the policy~$\pi^*(x)$.
\begin{proposition}\label{pro:first_order_nablaV}
    For any $x\in\mathbb{R}^n$, $s\in\mdps$, and $a\in\mdpa$, 
    \begin{align*}
         \nabla V_s^*(x) &= \mathbb{E}\left[ \sum_{t=0}^\infty \gamma^t\nabla r_{s_t,a_t}(x)\ \Big |\ s_0=s, \pi^*(x),\mdp_\tau\kh{x} \right],
         \\
          \nabla Q_{sa}^*(x) &= \mathbb{E}\left[ \sum_{t=0}^\infty \gamma^t\nabla r_{s_t,a_t}(x)\ \Big |\ s_0=s,a_0=a,\pi^*(x),\mdp_\tau\kh{x} \right].
    \end{align*}       
\end{proposition}
\begin{proof}
    Compute the partial derivative
    \begin{equation*}
        \frac{\partial \varphi _s\left( x,v \right)}{\partial x_i}=\frac{\sum_a{\partial _{x_i}r_{sa}\left( x \right) \exp \left( \tau ^{-1}\left( r_{sa}\left( x \right) +\gamma \sum_{s^{\prime\prime}}{P_{sas^{\prime\prime}}v_{s^{\prime\prime}}} \right) \right)}}{\sum_a{\exp \left( \tau ^{-1}\left( r_{sa}\left( x \right) +\gamma \sum_{s^{\prime\prime}}{P_{sas^{\prime\prime}}v_{s^{\prime\prime}}} \right) \right)}},
    \end{equation*}
    and substitute $v=V^*(x)$ into it,
    \begin{equation}\label{eq:phi_grad_x_1}
        \nabla _x\varphi _s\left( x,V^*\left( x \right) \right) = \sum_a{\nabla r_{sa}\left( x \right) \pi _{sa}^{*}\left( x \right)},
    \end{equation}
    where we adopt the similar derivation as \eqref{eq:phi_grad_v_2}. Applying \cref{pro:inverse_nonnegative} to \eqref{eq:hyper_result2}, we obtain \begin{equation}\label{eq:phi_grad_v_3}
       \left( I-\nabla _v\varphi \left( x,V^*\left( x \right) \right) \right) ^{-1}=\sum_{t=0}^{\infty}{\left( \gamma P^{\pi ^*\left( x \right)} \right) ^t}.
    \end{equation}
    Employ the subscript $s\in\mdps$ to denote the corresponding row of a matrix. It follows from \eqref{eq:hyper_result1}, \eqref{eq:phi_grad_x_1} and \eqref{eq:phi_grad_v_3} that
    \begin{align*}
        \nabla V_{s}^{*}\left( x \right) =&\left( I-\nabla _v\varphi \left( x,V^*(x) \right) \right) _{s}^{-1}\nabla _x\varphi \left( x,V^*\left( x \right) \right) 
        \\
        =&\sum_{t=0}^{\infty}{\left( \gamma P^{\pi ^*\left( x \right)} \right) _{s}^{t}}\nabla _x\varphi \left( x,V^*\left( x \right) \right) 
        \\
        =&\sum_{t=0}^{\infty}{\gamma ^t\sum_{s^{\prime}}{P\left( s_t=s^{\prime}|s_0=s,\pi ^*\left( x \right) \right)}}\sum_a{\nabla r_{s^{\prime}a}\left( x \right) \pi _{s^{\prime}a}^{*}\left( x \right)}
        \\
        =&\sum_{t=0}^{\infty}{\gamma ^t}\sum_{s^{\prime},a}{P\left( s_t=s^{\prime},a_t=a|s_0=s,\pi ^*\left( x \right) \right) \nabla r_{s^{\prime}a}\left( x \right)}
        \\
        =&\ \mathbb{E} \left[ \sum_{t=0}^{\infty}{\gamma ^t\nabla _xr_{s_ta_t}\left( x \right)}\,\,| s_0=s,\pi ^*\left( x \right) ,\mathcal{M} _{\tau}\left( x \right) \right],
    \end{align*}
    where $P\left( s_t=s^{\prime},a_t=a|s_0=s,\pi ^*\left( x \right) \right) $ is the probability of taking action $a$ at state $s^\prime$ at time $t$, given that the process starts from state $s_0=s$ in the MDP $\mdp_\tau(x)$. Differentiating both sides of 
    \begin{equation*}
         Q_{sa}^{*}\left( x \right) = r_{sa}\left( x \right) +\gamma \sum_{s^{\prime}}{P_{sas^{\prime}} V_{s^{\prime}}^{*}\left( x \right)}
    \end{equation*}
    with respect to~$x$, we have
    \begin{align*}
        \nabla Q_{sa}^{*}\left( x \right) =&\nabla r_{sa}\left( x \right) +\gamma \sum_{s^{\prime}}{P_{sas^{\prime}}\nabla V_{s^{\prime}}^{*}\left( x \right)}
        \\
        =&\ \mathbb{E} \left[ \sum_{t=0}^{\infty}{\gamma ^t\nabla _xr_{s_ta_t}\left( x \right)}\,\,| s_0=s,a_0=a,\pi ^*\left( x \right) ,\mathcal{M} _{\tau}\left( x \right) \right] .
    \end{align*}
\end{proof}

Subsequently, the next proposition characterizes the hyper-gradient via fully first-order information.
\begin{proposition}
    The hyper-gradient $\nabla \phi(x)$ of the bilevel reinforcement learning problem \eqref{eq:standar_biRL} with the upper-level function~\eqref{eq:upper_level_f} can be computed by
    \begin{align*}
        \nabla \phi(x) =&\ \mathbb{E} _{d_i\sim \rho \left( d;\pi ^*\left( x \right) \right)}\left[ \nabla l\left( d_1,d_2,\ldots,d_I;x \right) \right]   
        \\
        &+\tau ^{-1}\mathbb{E} _{d_i\sim \rho \left( d;\pi ^*\left( x \right) \right)}\left[ l\left( d_1,d_2,\ldots,d_I;x \right) \left( \sum_i{\sum_h{\nabla \left( Q_{s_{h}^{i}a_{h}^{i}}^{*}\left( x \right) -V_{s_{h}^{i}}^{*}\left( x \right) \right)}} \right) \right].    
    \end{align*}
\end{proposition}
\begin{proof}
Recall that
\begin{equation*}
    P(d_i;\pi ^*\left( x \right) )=\bar{\rho}\kh{s_0^i}\zkh{\Pi_{h=0}^{H-2}\pi ^*_{s_h^i,a_h^i}\left( x \right) \bar{P}\kh{s_{h+1}^i|s_h^i,a_h^i}}\pi^*_{s_{H-1}^i,a_{H-1}^i}(x),
\end{equation*}
and 
\begin{equation*}
    \phi(x) = \mathbb{E} _{d_i\sim \rho \left( d;\pi ^*\left( x \right) \right)}\left[ l\left( d_1,d_2,\ldots,d_I;x \right) \right] =\sum_{\left( d_1,\ldots,d_I \right)}{l\left( d_1,d_2,\ldots,d_I;x \right) \Pi _{i=1}^{I}}P\left( d_i;\pi ^*\left( x \right) \right).
\end{equation*}
It follows that
\begin{align*}
    &\nabla \phi \left( x \right)
    \\
    =&\sum_{\left( d_1,\ldots,d_I \right)}{\nabla l\left( d_1,d_2,\ldots,d_I;x \right) \Pi _{i=1}^{I}}P\left( d_i;\pi ^*\left( x \right) \right) +\sum_{\left( d_1,\ldots,d_I \right)}{l\left( d_1,d_2,\ldots,d_I;x \right) \nabla \Pi _{i=1}^{I}}P\left( d_i;\pi ^*\left( x \right) \right) 
    \\
    =&\ \mathbb{E} _{d_i\sim \rho \left( d;\pi ^*\left( x \right) \right)}\left[ \nabla l\left( d_1,d_2,\ldots,d_I;x \right) \right] +\mathbb{E} _{d_i\sim \rho \left( d;\pi ^*\left( x \right) \right)}\left[ l\left( d_1,d_2,\ldots,d_I;x \right) \nabla \log \Pi _{i=1}^{I}P\left( d_i;\pi ^*\left( x \right) \right) \right] 
    \\
    =&\ \mathbb{E} _{d_i\sim \rho \left( d;\pi ^*\left( x \right) \right)}\left[ \nabla l\left( d_1,d_2,\ldots,d_I;x \right) \right] +\mathbb{E} _{d_i\sim \rho \left( d;\pi ^*\left( x \right) \right)}\left[ l\left( d_1,d_2,\ldots,d_I;x \right) \nabla \left( \sum_i{\log P\left( d_i;\pi ^*\left( x \right) \right)} \right) \right] 
    \\
    =&\ \mathbb{E} _{d_i\sim \rho \left( d;\pi ^*\left( x \right) \right)}\left[ \nabla l\left( d_1,d_2,\ldots,d_I;x \right) \right] +\mathbb{E} _{d_i\sim \rho \left( d;\pi ^*\left( x \right) \right)}\left[ l\left( d_1,d_2,\ldots,d_I;x \right) \left( \sum_i{\sum_h{\nabla \log \pi _{s_{h}^{i}a_{h}^{i}}^{*}\left( x \right)}} \right) \right] 
    \\
    =&\ \mathbb{E} _{d_i\sim \rho \left( d;\pi ^*\left( x \right) \right)}\left[ \nabla l\left( d_1,d_2,\ldots,d_I;x \right) \right] 
    \\
    &+\tau ^{-1}\mathbb{E} _{d_i\sim \rho \left( d;\pi ^*\left( x \right) \right)}\left[ l\left( d_1,d_2,\ldots,d_I;x \right) \left( \sum_i{\sum_h{\nabla \left( Q_{s_{h}^{i}a_{h}^{i}}^{*}\left( x \right) -V_{s_{h}^{i}a_{h}^{i}}^{*}\left( x \right) \right)}} \right) \right], 
\end{align*}
where the first equality comes from the chain rules, the second equality is obtained by ``log probability tricks'' in REINFORCE \citep{sutton2018reinforcement}, and the last equality takes advantege of \eqref{eq:consis_policy}.
\end{proof}

\section{PROPERTIES OF $V^*(x)$ AND $\pi^*(x)$}\label{sec:general_smooth}
At the beginning, we uncover the Lipschitz properties of $V^*(x)$, based on the results of Propositions \ref{pro:first_order_nablaV} and \ref{pro:QV_Lipz}. The proof of \cref{pro:QV_Lipz} is deferred to \cref{sec:ana_SoBiRL}, since it needs more analytical tools, and this section primarily concentrates on the properties of $V^*(x)$ and $\pi^*(x)$. It is worth stressing that all results in this section only rely on \cref{assu:r}.

\begin{lemma}\label{lem:vsmooth1}
    Under \cref{assu:r}, $V^*(x)$ is ${L_V}$-Lipschitz continuous with $L_V=\frac{\sqrt{\left| \mathcal{S} \right|}C_{rx}}{1-\gamma}$, i.e., for any $x_1,x_2\in\mathbb{R}^{n}$,
    \begin{equation*}
        \norm{V^*(x_1)-V^*(x_2)}_2\le L_V\norm{x_1-x_2}_2.
    \end{equation*}
\end{lemma}
\begin{proof}
    Combining the characterization from \cref{pro:first_order_nablaV}
    \begin{equation*}
        \nabla V_{s}^{*}(x)=\mathbb{E} \left[ \sum_{t=0}^{\infty}{\gamma ^t}\nabla r_{s_t,a_t}(x) | s_0=s,\pi ^*(x),\mathcal{M} _{\tau}\left( x \right) \right], 
    \end{equation*}
    and the boundedness of $\nabla r(x)$ revealed by \cref{assu:r} yields
    \begin{equation*}
        \left\| \nabla V_{s}^{*}(x) \right\| _2\le \frac{C_{rx}}{1-\gamma},\quad\text{for}\ s\in\mdps.
    \end{equation*}
    Put differently, the $2$-norm of each row vector in $\nabla V^*(x)$ does not exceed $\frac{C_{rx}}{1-\gamma}$, which means
    \begin{equation*}
        \left\| \nabla V^*(x) \right\| _2\le \left\| \nabla V^*(x) \right\| _F\le \frac{\sqrt{\left| \mathcal{S} \right|}C_{rx}}{1-\gamma}.
    \end{equation*}
    Here, $\norm{\cdot}_F$ is the Frobenius norm of a matrix. Setting $L_V=\frac{\sqrt{\left| \mathcal{S} \right|}C_{rx}}{1-\gamma}$ completes the proof.
\end{proof}

The Lipschitz continuity of $\pi^*(x)$ is connected to that of the value functions by the consistency condition \eqref{eq:consis_policy}.
\begin{proposition}\label{lem:pismooth}
    Under \cref{assu:r}, $\pi^*(x)$ is ${L_\pi}$-Lipschitz continuous, i.e., for any $x_1,x_2\in\mathbb{R}^{n}$,
    \begin{equation*}
        \norm{\pi^*(x_1)-\pi^*(x_2)}_2\le\norm{\log\pi^*(x_1)-\log\pi^*(x_2)}_2\le L_\pi\norm{x_1-x_2},
    \end{equation*}
    where $L_{\pi}=\frac{2\sqrt{\left| \mathcal{S} \right|\left| \mathcal{A} \right|}C_{rx}}{\tau \left( 1-\gamma \right)}$.
\end{proposition}
\begin{proof}
For any $s\in\mdps$, $a\in\mdpa$, from the results of \cref{pro:first_order_nablaV},
    \begin{align*}
    \left\| \nabla Q_{sa}^{*}(x) \right\| _2\le \frac{C_{rx}}{1-\gamma}\quad\text{and}\quad\left\| \nabla V_{s}^{*}(x) \right\| _2\le \frac{C_{rx}}{1-\gamma},
    \end{align*}
and the consistency condition \eqref{eq:consis_policy},
\begin{equation*}
    \log \pi _{sa}^{*}(x)=\tau ^{-1}\left( Q_{sa}^{*}\left( x \right) -V_{s}^{*}\left( x \right) \right) ,
\end{equation*}
it follows that
\begin{align*}
    \left\| \log \pi _{sa}^{*}(x_1)-\log \pi _{sa}^{*}(x_2) \right\| \le& \tau ^{-1}\left( \left\| Q_{sa}^{*}\left( x_1 \right) -Q_{sa}^{*}\left( x_2 \right) \right\| _2+\left\| V_{s}^{*}\left( x_1 \right) -V_{s}^{*}\left( x_2 \right) \right\|_2 \right) 
    \\
    \le& \frac{2C_{rx}}{\tau \left( 1-\gamma \right)}\left\| x_1-x_2 \right\| _2.
\end{align*}
The shape of $\pi\in\mathbb{R}^{\abs{\mdps}\abs{\mdpa}}$ implies
\begin{equation*}
    \left\| \log \pi ^*(x_1)-\log \pi ^*(x_2) \right\| _2\le \frac{2\sqrt{\left| \mathcal{S} \right|\left| \mathcal{A} \right|}C_{rx}}{\tau \left( 1-\gamma \right)}\left\| x_1-x_2 \right\| _2.
\end{equation*}
Applying \cref{pro:lognorm} achieves the conclusion.
\end{proof}

We derive the Lipschitz constant $L_{V1}$ in \cref{pro:QV_Lipz} focusing on the smoothness of each $V_s^*(x)$, which lays a foundation for the subsequent proposition. The following Lipschitz constant $L^M_{V1}$ is noted with the superscript $M$ to emphasize we consider the continuity of the whole matrix $\nabla V^*(x)$ here.
\begin{proposition}
Under~\cref{assu:r}, $V^*(x)$ is $L^M_{V1}$-Lipschitz smooth, i.e., for any $x_1,x_2\in\mathbb{R}^n$,
\begin{equation*}
\norm{\nabla V^*(x_1) - \nabla V^*(x_2)}_2\le L^M_{V1}\norm{x_1-x_2}_2,
\end{equation*}
where $L_{V1}^{M}=\sqrt{\left| \mathcal{S} \right|}\left( \frac{(1+\gamma)C_{rx}L_\pi\sqrta}{\left( 1-\gamma \right) ^2}+\frac{L_r}{1-\gamma}\right)$.    
\end{proposition}\label{lem:vsmooth2}
\begin{proof}
    \cref{pro:QV_Lipz} implies that for any $s\in\mdps$,
    \begin{align*}
        \left\| \nabla V_{s}^{*}(x_1)-\nabla V_{s}^{*}(x_2) \right\| _2&\le L_{V1}\left\| x_1-x_2 \right\| _2,
    \end{align*}
    with $L_{V1}=\frac{(1+\gamma)C_{rx}L_\pi\sqrta}{\left( 1-\gamma \right) ^2}+\frac{L_r}{1-\gamma}$. The property of matrix norms, $\norm{A}_2\le\norm{A}_F$ leads to
    \begin{equation*}
        \left\| \nabla V^*(x_1)-\nabla V^*(x_2) \right\| _2\le \left\| \nabla V^*(x_1)-\nabla V^*(x_2) \right\| _F\le \sqrt{\left| \mathcal{S} \right|}L_{V1}\left\| x_1-x_2 \right\| _2.
    \end{equation*}
    In this way, we can choose
    \begin{equation*}
        L_{V1}^{M}=\sqrt{\left| \mathcal{S} \right|}L_{V1}=\sqrt{\left| \mathcal{S} \right|}\left( \frac{(1+\gamma)C_{rx}L_\pi\sqrta}{\left( 1-\gamma \right) ^2}+\frac{L_r}{1-\gamma} \right) .
    \end{equation*}
\end{proof}

\section{ANALYSIS OF M-SOBIRL}\label{sec:ana_MSoBiRL}
\begin{algorithm}[htbp]
    \caption{Model-based soft bilevel reinforcement learning algorithm (M-SoBiRL)}
    \label{alg:M-SoBiRL}
    \begin{algorithmic}[1]
        \REQUIRE outer and inner iteration numbers $K,N$, step sizes $\beta,\eta$, initialization $x_0, \pi_0, Q_{0},w_0$
        \FOR{$k = 1, \dots, K$}
            \STATE $A_k=\kh{I-\gamma P^{\pi_k}}^{\top}$,\ $b_k=U^\top\operatorname{diag}\kh{\pi_k}\nabla_\pi f(x_k,\pi_k)$
            \STATE Compute $V_k$ by \eqref{eq:V_from_Q}
            \STATE $w_{k} = w_{k-1} - \eta\kh{\kh{A_k^\top A_k}w_{k-1}-A_k^\top b_k}$
            \STATE $x_{k+1} = x_k - \beta \widehat{\nabla} \phi \left(x_k,\pi_k,V_k,w_{k}\right)$
            \STATE $Q_{k,0}=Q_k$
            \FOR{$n = 1, \dots, N$}
                \STATE $Q_{k,n}=\mathcal{T}_{\mdp_\tau(x_{k+1})}(Q_{k,n-1})$
            \ENDFOR
            \STATE $Q_{k+1}=Q_{k,N}$, $\pi_{k+1}=\texttt{softmax}\kh{Q_{k+1}/\tau}$ 
        \ENDFOR
        \ENSURE $(x_{K+1},\pi_{K})$
    \end{algorithmic}
\end{algorithm}
Enlightened by the hyper-gradient, 
\begin{align}
    \nabla \phi(x)=&\ \nabla_x f(x,\pi^*(x)) +  \nabla \pi^*(x)^\top \nabla_\pi f(x,\pi^*(x)) \nonumber
    \\
    =&\ \nabla_x f(x,\pi^*(x)) + \tau^{-1} \nabla r(x)^\top \operatorname{diag}\kh{\pi^*(x)} \nabla_\pi f(x,\pi^*(x))   \nonumber  
    \\
    &-\tau^{-1}\nabla_{x}\varphi\kh{x,V^*(x)}^\top \kh{I- \gamma P^{\pi^*(x)}}^{-\top} U^\top\operatorname{diag}\kh{\pi^*(x)}\nabla_\pi f(x,\pi^*(x)), \label{eq:exact_hypergrad2}
\end{align}
we propose the model-based algorithm, M-SoBiRL; see \cref{alg:M-SoBiRL}. \revise{Note that in practice, the special structure of RL will facilitate the computation. Specifically, the matrix $A_k=(I-\gamma P^{\pi_k})^\top$ is usually sparse in RL, and the calculation of $b_k=U^\top\mathrm{diag}(\pi_k)\nabla_\pi f(x_k,\pi_k)$ also involves a sparse constant matrix $U$ and a diagonal $\mathrm{diag}(\pi_k)$. Therefore, the multiplications $A^\top_k A_k$ and $A_k^\top b_k$ for updating $w_k$ in line 4 can be accelerated by sparse multiplication oracles, the cost of which is denoted as $c_{\mathrm{sparse}}$. Additionally, evaluating $\hat{\nabla}\phi$ and updating $x$ in line~5 cost $\mathcal{O}(|\mathcal{S}||\mathcal{A}|n)$. Moreover, the other lines in \cref{alg:M-SoBiRL} totally cost $\mathcal{O}(|\mathcal{S}||\mathcal{A}|)$ since they are element-wise operations on $Q_k$ or $V_k$. In summary, the computation cost of each outer iteration is $\mathcal{O}(|\mathcal{S}||\mathcal{A}|n+c_{\mathrm{sparse}})$.}

Next, we focus on the convergence analysis for M-SoBiRL, starting with a short proof sketch for guidance.
\subsection{Proof Sketch of \cref{the:M_sobirl}}
 Denote $\pi _{k}^{*}:=\pi ^*\left( x_k \right)$, $Q_{k}^{*}:=Q^*\left( x_k \right) $, $V_{k}^{*}:=V^*\left( x_k \right) $, and
\begin{equation*}
    w_{k}^{*}:=\left( I- \gamma P^{\pi^*_k} \right) ^{-\top}U^{\top}\mathrm{diag}\left( \pi ^*_k \right) \nabla _{\pi}f(x_k,\pi ^*_k ).
\end{equation*}
The proof is structured into four main steps.
\paragraph{Step1: Preliminary Properties}
 \hspace*{\fill} \\
 \cref{sec:basic_lip} studies the Lipschitz and boundedness properties of the quantities related to \cref{alg:M-SoBiRL}, laying a foundation for the following analysis.
\paragraph{Step2: Upper-bounding the Residual $\norm{w_{k}-w_{k}^*}_2$}
 \hspace*{\fill} \\
 \cref{sec:converge_w} bounds the term $\norm{w_{k}-w_{k}^*}_2$,
\begin{align*}
    \left\| w_k-w_{k}^{*} \right\| _{2}^{2}\le&\left( 1-\frac{\eta \left( 1-\gamma \right) ^2}{2} \right) \left\| w_{k-1}-w_{k-1}^{*} \right\| _{2}^{2}+\eta ^2C_{\sigma \pi}^{2}\left( \frac{1}{\left( 1-\gamma \right) ^2}+4 \right) \left\| \pi _k-\pi _{k}^{*} \right\| _{\infty}^{2}
    \\
    &+2L_{w}^{2}\left( 1+\frac{1}{\eta \left( 1-\gamma \right) ^2} \right) \left\| x_k-x_{k-1} \right\| _{2}^{2}+4\eta ^2\left| \mathcal{S} \right|^2\left( 1+\gamma \right) ^2\left\| w_{k-1}-w_{k-1}^{*} \right\| _{2}^{2},
\end{align*}
where the coefficient $\left( 1-\frac{\eta \left( 1-\gamma \right) ^2}{2} \right)$ implies its descent property, and the constants $C_{\sigma \pi},L_w$ will be specified later.

\paragraph{Step3: Upper-bounding the Errors $\norm{Q_k-Q_k^*}_\infty$}
 \hspace*{\fill} \\
 \cref{sec:Bellman} measures the quality of the policy evaluation, $\norm{Q_k-Q_k^*}_\infty$, which can be bounded by
\begin{align*}
    \left\| Q_{k}-Q_{k}^{*} \right\| _{\infty}\le \gamma ^N\left( \left\| Q_{k-1}-Q_{k-1}^{*} \right\| _{\infty}+\frac{C_{rx}}{1-\gamma}\left\| x_{k}-x_{k-1} \right\| _2 \right),
\end{align*}
with the contraction factor $\gamma^N$ responsible for the convergence property.

\paragraph{Step4: Assembling the Estimations above and Achieving the Conclusion}
 \hspace*{\fill} \\
Considering the merit function 
\begin{equation*}
    \mathcal{L} _k:=\phi \left( x_k \right)+\zeta_w\norm{w_k-w^*_k}^2_2 + \zeta _Q\norm{Q_k-Q^*_k}_2^2,
\end{equation*}
where the coefficients $\zeta_w$ and $\zeta_Q$ are to be determined later, \cref{sec:converge_MSoBiRL} evaluates $\mathcal{L} _{k+1}-\mathcal{L} _k$ to reveal the decreasing property of $\mathcal{L} _k$. As a result, it proves that M-SoBiRL enjoys the convergence rate $\mathcal{O}(\epsilon^{-1})$ with the inner iteration number $N=\mathcal{O}(1)$ independent of the solution accuracy $\epsilon$. 

\subsection{Lipschitz Properties of Quantities Related to \cref{alg:M-SoBiRL}}\label{sec:basic_lip}
In this subsection, we study the Lipschitz and boundedness properties of the quantities related to \cref{alg:M-SoBiRL}, setting the groundwork for convergence analysis.
\begin{lemma}\label{lem:Ppi_Lip}
    $P^\pi$ is $\sqrt{\left| \mathcal{A} \right|}$-Lipschitz continuous with respect to $\pi$, i.e., for any $\pi^1,\pi^2$,
    \begin{equation*}
        \left\| P^{\pi ^1}-P^{\pi ^2} \right\| _2\le \sqrt{\left| \mathcal{A} \right|}\left\| \pi ^1-\pi ^2 \right\| _2
    \end{equation*}
\end{lemma}
\begin{proof}
    The element of $P^{\pi ^1}-P^{\pi ^2}$ in position $(s,s^\prime)$ is $\sum_a{\left( \pi _{sa}^{1}-\pi _{sa}^{2} \right) P_{sas^{\prime}}}$. Consider the $2$-norm of the $s$-th row in $P^{\pi ^1}-P^{\pi ^2}$,
    \begin{align*}
        &\sum_{s^{\prime}}{\left( \sum_a{\left( \pi _{sa}^{1}-\pi _{sa}^{2} \right) P_{sas^{\prime}}} \right)}^2
        \\
        \le& \sum_{s^{\prime}}{\left( \sum_a{\left( \pi _{sa}^{1}-\pi _{sa}^{2} \right) ^2} \right)}\left( \sum_a{P_{sas^{\prime}}^{2}} \right) 
        \\
        \le& \left( \sum_a{\left( \pi _{sa}^{1}-\pi _{sa}^{2} \right) ^2} \right) \sum_{s^{\prime}}{\left( \sum_a{P_{sas^{\prime}}} \right)}
        \\
        =& \left| \mathcal{A} \right|\left\| \pi _{s}^{1}-\pi _{s}^{2} \right\| _{2}^{2},
    \end{align*}
    which leads to
    \begin{equation*}
        \left\| P^{\pi ^1}-P^{\pi ^2} \right\| _2\le \left\| P^{\pi ^1}-P^{\pi ^2} \right\| _F\le \sqrt{\left| \mathcal{A} \right|}\left\| \pi ^1-\pi ^2 \right\| _2.
    \end{equation*}
\end{proof}

\begin{lemma}\label{lem:K_bound}    $U\in\mathbb{R}^{\abs{\mdps}\abs{\mdpa}\times\abs{\mdps}}$ has full column rank, and 
\begin{equation*}
    {\norm{U}_2\in \zkh{\sqrt{\abs{\mdpa}}(1+\gamma), \sqrt{\abs{\mdps}\abs{\mdpa}(1+\gamma)}}}.
\end{equation*}
\end{lemma}
\begin{proof}
Define that for any $\ a\in\mdpa$, $U_a = I-\gamma P_a$, where $P_a\in\mathbb{R}^{\abs{\mdps}\times\abs{\mdps}}$ with $P_a(s,s^\prime) = P_{sas^\prime}$. By the triangle inequality, $\norm{U_a}_\infty\in\zkh{1-\gamma.1+\gamma}$. Subsequently, the property of the matrix norms, $\norm{A}_2\le\sqrt{\norm{A}_1\norm{A}_\infty}$, produces
\begin{equation}\label{eq:U_singular_bound_up}
    \norm{U_a}_2\le \sqrt{\norm{U_a}_1\norm{U_a}_\infty}\le\sqrt{\abs{\mdps}(1+\gamma)}.
\end{equation}
Applying \cref{prop:I_gamma} to $U_a$ deduces that the magnitude of every eigenvalue of $U_a$ belongs to $[1-\gamma,1+\gamma]$. The property that $\norm{U_a}_2$ is greater than the magnitude of any of its eigenvalues leads to
\begin{equation}\label{eq:U_singular_bound_low}
    \norm{U_a}_2\ge 1+\gamma.
\end{equation}
The definition of $U_a$ reveals that the $\abs{\mdps}$ rows of the matrix $U_a$ are distributed in $U$ with a spacing of $\abs{\mdpa}$. In this way, $U$ has full column rank since $U_a$ is invertible. Consider the singular value of $U\in\mathbb{R}^{\abs{\mdps}\abs{\mdpa}\times\abs{\mdps}}$. Specifically,
\begin{equation*}
U^\top U=\kh{\sum_{a\in\mdpa}\kh{\sum_{s^\prime\in\mdps}U_{a}^\top(s^\prime)U_{a}(s^\prime)}}=\kh{\sum_{a\in\mdpa}U_a^\top U_a},
\end{equation*}
where $U_{a}(s^\prime)$ denotes the $s^\prime$-th row of $U_{a}$. Combining~\eqref{eq:U_singular_bound_up} with \eqref{eq:U_singular_bound_low}, it follow that
\begin{equation*}
    \norm{U}_2\in \zkh{\sqrt{\abs{\mdpa}}(1+\gamma), \sqrt{\abs{\mdps}\abs{\mdpa}(1+\gamma)}}.
\end{equation*}
\end{proof}

\begin{lemma}
    Define
\begin{equation*}
    \vartheta \left( x \right) :=U^{\top}\mathrm{diag}\left( \pi ^*\left( x \right) \right) \nabla _{\pi}f(x,\pi ^*\left( x \right) ).
\end{equation*}
    Under Assumptions~\ref{assu:f} and \ref{assu:r}, $\vartheta \left( x \right) $ is $L_\vartheta$-Lipschitz continuous, i.e., for any $x_1,x_2\in\mathbb{R}^n$,
    \begin{equation*}
        \left\| \vartheta \left( x_1 \right) -\vartheta \left( x_2 \right) \right\| _2 \le L_\vartheta\left\| x_1-x_2 \right\| _2,
    \end{equation*}
    where $L_\vartheta = \sqrt{\abs{\mdps}\abs{\mdpa}(1+\gamma)} \left( L_f\left( 1+L_{\pi} \right) +C_{f\pi}L_{\pi} \right) $.
\end{lemma}
\begin{proof}
    Given the boundedness of $\nabla_\pi f(x,\pi^*(x))$ from Assumption~\ref{assu:f} and the boundedness of~$\norm{U}_2$ derived in~\cref{lem:K_bound}, it follows that
    \begin{align*}
        &\left\| \vartheta \left( x_1 \right) -\vartheta \left( x_2 \right) \right\| _2
        \\
        \le&\sqrt{\abs{\mdps}\abs{\mdpa}(1+\gamma)} \left\| \mathrm{diag}\left( \pi _{1}^{*} \right) \nabla _{\pi}f(x_1,\pi _{1}^{*})-\mathrm{diag}\left( \pi _{2}^{*} \right) \nabla _{\pi}f(x_2,\pi _{2}^{*}) \right\| _2
        \\
        \le& \sqrt{\abs{\mdps}\abs{\mdpa}(1+\gamma)} \left\| \mathrm{diag}\left( \pi _{1}^{*} \right) \left( \nabla _{\pi}f(x_1,\pi _{1}^{*})-\nabla _{\pi}f(x_2,\pi _{2}^{*}) \right) \right\| _2
        \\
        &+\sqrt{\abs{\mdps}\abs{\mdpa}(1+\gamma)}\left\| \left( \mathrm{diag}\left( \pi _{1}^{*} \right) -\mathrm{diag}\left( \pi _{2}^{*} \right) \right) \nabla _{\pi}f(x_2,\pi _{2}^{*}) \right\| _2
        \\
        \le& \sqrt{\abs{\mdps}\abs{\mdpa}(1+\gamma)} \left( L_f\left( \left\| x_1-x_2 \right\| _2+\left\| \pi _{1}^{*}-\pi _{2}^{*} \right\| _{\infty} \right) +C_{f\pi}\left\| \pi _{1}^{*}-\pi _{2}^{*} \right\| _{\infty} \right) 
        \\
        \le& \sqrt{\abs{\mdps}\abs{\mdpa}(1+\gamma)} \left( L_f\left( 1+L_{\pi} \right) +C_{f\pi}L_{\pi} \right) \left\| x_1-x_2 \right\| _2.
    \end{align*}
\end{proof}

Consequently, we concentrate on the Lipschitz property of the hyper-objective $\phi(x)$, which plays a vital role in analyzing the convergence properties of \cref{alg:M-SoBiRL}. The superscript $M$ is appended to the Lipschitz constant, $L^M_\phi$ to make a distinction from the constant used in the model-free scenario.
\begin{proposition}\label{pro:L_M_phi}
Under Assumptions~\ref{assu:f} and~\ref{assu:r}, $\nabla \phi(x)$ is $L^M_\phi$-Lipschitz continuous, i.e., for any $x_1,x_2\in\mathbb{R}^n$,
\begin{equation*}
\norm{\nabla\phi(x_1) - \nabla\phi(x_2)}_2\le L^M_\phi\norm{x_1-x_2}_2,
\end{equation*}
where $L^M_\phi=\mathcal{O}\kh{\frac{\left| \mathcal{S} \right|^{{3}/{2}}\left| \mathcal{A} \right|^{{3}/{2}}}{\left( 1-\gamma \right) ^3}}$ is specified in \eqref{eq:L_phi_M}.
\end{proposition}
\begin{proof}
    Revisit the expression of $\nabla \phi(x)$ and incorporate the notation $\vartheta(x)$,
    \begin{align}
        \!\!\nabla \phi(x)  =& \nabla_x f(x,\pi^*(x)) + {\tau^{-1}} \Big ( \nabla r(x)-U\nabla V^*\left( x \right) \Big )^\top \operatorname{diag}\kh{\pi^*(x)} \nabla_\pi f(x,\pi^*(x))    \nonumber
        \\
        =&\nabla_x f(x,\pi^*(x)) + {\tau^{-1}}\nabla r(x)^\top\operatorname{diag}\kh{\pi^*(x)} \nabla_\pi f(x,\pi^*(x)) - {\tau^{-1}}\nabla V^*\left( x \right)^\top\vartheta(x).   \label{eq:M_decom_hypergrad}
    \end{align}
    Consider the first term in \eqref{eq:M_decom_hypergrad},
    \begin{align}
        &\left\| \nabla _xf(x_1,\pi ^*(x_1))-\nabla _xf(x_2,\pi ^*(x_2)) \right\| _2    \nonumber
        \\
        \le&\ L_f\left( \left\| x_1-x_2 \right\| _2+\left\| \pi ^*(x_1)-\pi ^*(x_2) \right\| _2 \right)     \nonumber
        \\
        \le& \left( L_f+L_fL_{\pi} \right) \left\| x_1-x_2 \right\| _2. \label{eq:term_1_Lip_hyper}
    \end{align}
       Subsequent analysis involves the Lipschitz property of the second term,
    \begin{align}
        &\left\| \nabla r(x_1)^{\top}\mathrm{diag}\left( \pi ^*(x_1) \right) \nabla _{\pi}f(x_1,\pi ^*(x_1))-\nabla r(x_2)^{\top}\mathrm{diag}\left( \pi ^*(x_2) \right) \nabla _{\pi}f(x_2,\pi ^*(x_2)) \right\| _2    \nonumber
        \\
        \le& \left\| \nabla r(x_1)^{\top}\mathrm{diag}\left( \pi ^*(x_1) \right) \nabla _{\pi}f(x_1,\pi ^*(x_1))-\nabla r(x_2)^{\top}\mathrm{diag}\left( \pi ^*(x_1) \right) \nabla _{\pi}f(x_1,\pi ^*(x_1)) \right\| _2    \nonumber
        \\
        &+\left\| \nabla r(x_2)^{\top}\mathrm{diag}\left( \pi ^*(x_1) \right) \nabla _{\pi}f(x_1,\pi ^*(x_1))-\nabla r(x_2)^{\top}\mathrm{diag}\left( \pi ^*(x_2) \right) \nabla _{\pi}f(x_1,\pi ^*(x_1)) \right\| _2   \nonumber
        \\
        &+\left\| \nabla r(x_2)^{\top}\mathrm{diag}\left( \pi ^*(x_2) \right) \nabla _{\pi}f(x_1,\pi ^*(x_1))-\nabla r(x_2)^{\top}\mathrm{diag}\left( \pi ^*(x_2) \right) \nabla _{\pi}f(x_2,\pi ^*(x_2)) \right\| _2   \nonumber
        \\
        \le&\ \sqrt{\left| \mathcal{S} \right|\left| \mathcal{A} \right|}L_rC_{f\pi}\left\| x_1-x_2 \right\| _2+\sqrt{\left| \mathcal{S} \right|\left| \mathcal{A} \right|}C_{rx}C_{f\pi}L_{\pi}\left\| x_1-x_2 \right\| _2  \nonumber
        \\
        &+\ \sqrt{\left| \mathcal{S} \right|\left| \mathcal{A} \right|}C_{rx}\left( L_f+L_fL_{\pi} \right) \left\| x_1-x_2 \right\| _2 \nonumber
        \\
        =&\ \sqrt{\left| \mathcal{S} \right|\left| \mathcal{A} \right|}\left( L_rC_{f\pi}+C_{rx}C_{f\pi}L_{\pi}+C_{rx}\left( L_f+L_fL_{\pi} \right) \right) \left\| x_1-x_2 \right\| _2   \label{eq:term_2_Lip_hyper}.
    \end{align}
    Similarly, the third term in \eqref{eq:M_decom_hypergrad} satisfy
    \begin{align}
        &\left\| \nabla V^*\left( x_1 \right) ^{\top}\vartheta (x_1)-\nabla V^*\left( x_2 \right) ^{\top}\vartheta (x_2) \right\| _2 \nonumber
        \\
        \le&\left\| \nabla V^*\left( x_1 \right) ^{\top}\vartheta (x_1)-\nabla V^*\left( x_1 \right) ^{\top}\vartheta (x_2) \right\| _2+\left\| \nabla V^*\left( x_1 \right) ^{\top}\vartheta (x_2)-\nabla V^*\left( x_2 \right) ^{\top}\vartheta (x_2) \right\| _2 \nonumber
        \\
        \le&\ \frac{\sqrt{\left| \mathcal{S} \right|}C_{rx}}{1-\gamma}L_{\vartheta}\left\| x_1-x_2 \right\| _2+\sqrt{\left| \mathcal{S} \right|\left| \mathcal{A} \right|\left( 1+\gamma \right)}C_{f\pi}L_{V1}^{M}\left\| x_1-x_2 \right\| _2    \nonumber
        \\
        =& \left( \frac{\sqrt{\left| \mathcal{S} \right|}C_{rx}}{1-\gamma}L_{\vartheta}+\sqrt{\left| \mathcal{S} \right|\left| \mathcal{A} \right|\left( 1+\gamma \right)}C_{f\pi}L_{V1}^{M} \right) \left\| x_1-x_2 \right\| _2.  \label{eq:term_3_Lip_hyper}
    \end{align}
    Collecting \eqref{eq:term_1_Lip_hyper}, \eqref{eq:term_2_Lip_hyper} and \eqref{eq:term_3_Lip_hyper} concludes that $\nabla \phi(x)$ is $L_\phi^M$-Lipschitz continuous with
    \begin{align}
        L_{\phi}^{M}=&\ L_f+L_fL_{\pi}+\sqrt{\left| \mathcal{S} \right|\left| \mathcal{A} \right|}\left( L_rC_{f\pi}+C_{rx}C_{f\pi}L_{\pi}+C_{rx}\left( L_f+L_fL_{\pi} \right) \right)    \nonumber
        \\
        &+\frac{\sqrt{\left| \mathcal{S} \right|}C_{rx}}{1-\gamma}L_{\vartheta}+\sqrt{\left| \mathcal{S} \right|\left| \mathcal{A} \right|\left( 1+\gamma \right)}C_{f\pi}L_{V1}^{M}    \nonumber
        \\
        =&\ \mathcal{O}\kh{\frac{\left| \mathcal{S} \right|^{\frac{3}{2}}\left| \mathcal{A} \right|^{\frac{3}{2}}}{\left( 1-\gamma \right) ^3}},  \label{eq:L_phi_M}
    \end{align}
    where \eqref{eq:L_phi_M} comes from the expressions of $L_\pi$, $L_\vartheta$ and $L^M_{V1}$.
\end{proof}

\subsection{Convergence Property of $w_k$}\label{sec:converge_w}
Define
\begin{equation*}
    A_{k}^{*}:=\left( I- \gamma P^{\pi^*(x_k)} \right) ^{\top},\quad \text{and}\quad b_{k}^{*} := \vartheta(x_k).
\end{equation*}
In \cref{alg:M-SoBiRL}, the goal of $w_k$ is to track $w_{k}^{*}=\left( A_{k}^{*} \right) ^{-1}b_{k}^{*}$ through outer iterations. This section delves into the descent property of $\norm{w_k-w_k^*}_2$. Initially, we prove that the exact quantities $\left\| b_{k}^{*} \right\| _2$ and $\left\| w_{k}^{*} \right\| _2$ are uniformly bounded.

\begin{lemma}
    Under Assumptions~\ref{assu:f} and \ref{assu:r}, it holds that
    \begin{align*}
        \left\| b_{k}^{*} \right\| _2 &\le \sqrt{\left| \mathcal{S} \right|\left| \mathcal{A} \right|\left( 1+\gamma \right)}C_{f\pi},
        \\
        \left\| w_{k}^{*} \right\| _2 &\le \frac{\sqrt{\left| \mathcal{S} \right|\left| \mathcal{A} \right|\left( 1+\gamma \right)}}{1-\gamma}C_{f\pi}.
    \end{align*}
\end{lemma}
\begin{proof}
    Taking into account $\norm{U}_2\le \sqrt{\left| \mathcal{S} \right|\left| \mathcal{A} \right|\left( 1+\gamma \right)} $, $\pi^*\in\Delta^{\abs{\mdps}}$, and $\nabla_\pi f(x,\pi^*(x))\le C_{f\pi}$, we can obtain
    \begin{align*}
        \left\| b_{k}^{*} \right\| _2 = &\norm{U^{\top}\mathrm{diag}\left( \pi _{k}^{*} \right) \nabla _{\pi}f(x_k,\pi _{k}^{*})}
        \\
        \le & \norm{U}_2 \norm{\mathrm{diag}\left( \pi _{k}^{*} \right)}_2 \norm{\nabla _{\pi}f(x_k,\pi _{k}^{*})}_2
        \\
        \le &\ \sqrt{\left| \mathcal{S} \right|\left| \mathcal{A} \right|\left( 1+\gamma \right)}C_{f\pi},
    \end{align*}
    where $\pi^*_k=\pi^*(x_k)$. Additionally, denote the eigenvalue of $A_k^*$ with the smallest moduli by $\lambda_{min}\kh{A_k^*}$. \cref{pro:inverse_nonnegative} reveals $\lambda_{min}\kh{A_k^*}\ge 1-\gamma$. In this way, $\norm{A_k^*}_2\ge \lambda_{min}\kh{A_k^*}$ and
    \begin{equation*}
        \left\| \left( A_{k}^{*} \right) ^{-1} \right\| _2\le \frac{1}{\left\| A_{k}^{*} \right\| _2}\le \frac{1}{1-\gamma}.
    \end{equation*}
    So, it completes the proof by
    \begin{equation*}
        \left\| w_{k}^{*} \right\| _2=\left\| \left( A_{k}^{*} \right) ^{-1}b_{k}^{*} \right\| _2\le \left\| \left( A_{k}^{*} \right) ^{-1} \right\| _2\left\| b_{k}^{*} \right\| _2\le \frac{\sqrt{\left| \mathcal{S} \right|\left| \mathcal{A} \right|\left( 1+\gamma \right)}}{1-\gamma}C_{f\pi}.
    \end{equation*}
\end{proof}

\begin{lemma}\label{lem:pre_for_w}
    Under Assumptions~\ref{assu:f} and \ref{assu:r}, the following inequalities hold in \cref{alg:M-SoBiRL}.
    \begin{align*}
        \left\| A_{k}^{*}-A_{k-1}^{*} \right\| _2 \le&\ \sqrt{\left| \mathcal{A} \right|}L_{\pi}\left\| x_k-x_{k-1} \right\| _2,
        \\
        \left\| A_k-A_{k}^{*} \right\| _2 \le&\ \sqrt{\left| \mathcal{A} \right|}\left\| \pi _k-\pi _{k}^{*} \right\| _2,
        \\
        \left\| A_k \right\| _2\le &\ \sqrt{\left| \mathcal{S} \right|\left( 1+\gamma \right)},   
        \\
       \left\| A_{k}^{*} \right\| _2\le &\ \sqrt{\left| \mathcal{S} \right|\left( 1+\gamma \right)},
        \\
        \left\| A_{k}^{\top}A_k-\left( A_{k}^{*} \right) ^{\top}A_{k}^{*} \right\| _2\le &\ 2\sqrt{\left| \mathcal{S} \right|\left| \mathcal{A} \right|\left( 1+\gamma \right)}\left\| \pi _k-\pi _{k}^{*} \right\| _2,
        \\
        \left\| b_k-b_{k}^{*} \right\| _2\le&\ \sqrt{\left| \mathcal{S} \right|\left| \mathcal{A} \right|\left( 1+\gamma \right)}\left( L_f+C_{f\pi} \right) \left\| \pi _k-\pi _{k}^{*} \right\| _{\infty}.
    \end{align*}
\end{lemma}
\begin{proof}
    Drawing on \cref{lem:Ppi_Lip} and the boundedness of $A_k,A_k^*$,
    \begin{equation*}
        \left\| A_k \right\| _2\le \sqrt{\left\| A_k \right\| _1\left\| A_k \right\| _{\infty}}\le \sqrt{\left| \mathcal{S} \right|\left( 1+\gamma \right)},\quad\text{and}\quad \left\| A_{k}^{*} \right\| _2\le \sqrt{\left| \mathcal{S} \right|\left( 1+\gamma \right)},
    \end{equation*}
    we obtain
    \begin{equation*}
        \left\| A_k-A_{k}^{*} \right\| _2\le \sqrt{\left| \mathcal{A} \right|}\left\| \pi _k-\pi _{k}^{*} \right\| _2,
    \end{equation*}
    \begin{equation*}
        \left\| A_{k}^{*}-A_{k-1}^{*} \right\| _2\le \left\| P^{\pi _{k}^{*}}-P^{\pi _{k-1}^{*}} \right\| _2\le \sqrt{\left| \mathcal{A} \right|}\left\| \pi _{k}^{*}-\pi _{k-1}^{*} \right\| _2\le \sqrt{\left| \mathcal{A} \right|}L_{\pi}\left\| x_k-x_{k-1} \right\| _2,
    \end{equation*}
    and
    \begin{align*}
        \left\| A_{k}^{\top}A_k-\left( A_{k}^{*} \right) ^{\top}A_{k}^{*} \right\| _2\le& \left\| A_{k}^{\top}\left( A_k-A_{k}^{*} \right) \right\| _2+\left\| \left( A_k-A_{k}^{*} \right) ^{\top}A_{k}^{*} \right\| _2
        \\
        \le&\ 2\sqrt{\left| \mathcal{S} \right|\left| \mathcal{A} \right|\left( 1+\gamma \right)}\left\| \pi _k-\pi _{k}^{*} \right\| _2.
    \end{align*}
    Moreover, the Lipschitz and boundedness properties related to $\nabla f$ in \cref{assu:f} reveal that
    \begin{align*}
        \left\| b_k-b_{k}^{*} \right\| _2\le &\ \sqrt{\left| \mathcal{S} \right|\left| \mathcal{A} \right|\left( 1+\gamma \right)}\left\| \mathrm{diag}\left( \pi _k \right) \nabla _{\pi}f(x_k,\pi _k)-\mathrm{diag}\left( \pi _{k}^{*} \right) \nabla _{\pi}f(x_k,\pi _{k}^{*}) \right\| _2
        \\
        \le&\ \sqrt{\left| \mathcal{S} \right|\left| \mathcal{A} \right|\left( 1+\gamma \right)}\left\| \mathrm{diag}\left( \pi _k \right) \left( \nabla _{\pi}f(x_k,\pi _k)-\nabla _{\pi}f(x_k,\pi _{k}^{*}) \right) \right\| _2
        \\
        &+ \sqrt{\left| \mathcal{S} \right|\left| \mathcal{A} \right|\left( 1+\gamma \right)}\left\| \left( \mathrm{diag}\left( \pi _k \right) -\mathrm{diag}\left( \pi _{k}^{*} \right) \right) \nabla _{\pi}f(x_k,\pi _{k}^{*}) \right\| _2 
        \\
        \le&\ \sqrt{\left| \mathcal{S} \right|\left| \mathcal{A} \right|\left( 1+\gamma \right)}\left( L_f+C_{f\pi} \right) \left\| \pi _k-\pi _{k}^{*} \right\| _{\infty}.
    \end{align*}
\end{proof}

The following lemma illustrates the descent property of $\norm{w_k-w_k^*}^2_2$
\begin{lemma}\label{lem:descent_w}
Under Assumption~\ref{assu:f} and \ref{assu:r}, the iterates $\{(x_k,\pi_k,w_k)\}$ generated by \cref{alg:M-SoBiRL} satisfy
\begin{align*}
    \left\| w_k-w_{k}^{*} \right\| _{2}^{2}\le&\left( 1-\frac{\eta \left( 1-\gamma \right) ^2}{2} \right) \left\| w_{k-1}-w_{k-1}^{*} \right\| _{2}^{2}+\eta ^2C_{\sigma \pi}^{2}\left( \frac{1}{\left( 1-\gamma \right) ^2}+4 \right) \left\| \pi _k-\pi _{k}^{*} \right\| _{\infty}^{2}
    \\
    &+2L_{w}^{2}\left( 1+\frac{1}{\eta \left( 1-\gamma \right) ^2} \right) \left\| x_k-x_{k-1} \right\| _{2}^{2}+4\eta ^2\left| \mathcal{S} \right|^2\left( 1+\gamma \right) ^2\left\| w_{k-1}-w_{k-1}^{*} \right\| _{2}^{2},
\end{align*}
with
\begin{align*}
    C_{\sigma \pi}&=\frac{2\left| \mathcal{S} \right|^{\frac{3}{2}}\left| \mathcal{A} \right|^{\frac{3}{2}}\left( 1+\gamma \right) C_{f\pi}}{1-\gamma}+\left| \mathcal{S} \right|\left| \mathcal{A} \right|^{\frac{3}{2}}\sqrt{\left( 1+\gamma \right)}C_{f\pi}+\left| \mathcal{S} \right|\sqrt{\left| \mathcal{A} \right|}\left( 1+\gamma \right) \left( L_f+C_{f\pi} \right) ,
    \\
    L_w&=\frac{1}{1-\gamma}\left( L_{\vartheta}+\frac{\left| \mathcal{A} \right|}{1-\gamma}\sqrt{\left| \mathcal{S} \right|\left( 1+\gamma \right)}C_{f\pi}L_{\pi} \right) .
\end{align*}
\end{lemma}
\begin{proof}
    By the identity $\norm{a-b}_2^2=\norm{a}^2+\norm{b}^2-2\innerp{a,b}$, we establish
    \begin{equation}\label{eq:w_wstart_decom}
        \left\| w_k-w_{k}^{*} \right\| _{2}^{2}=\left\| w_k-w_{k-1}^{*} \right\| _{2}^{2}+\left\| w_{k-1}^{*}-w_{k}^{*} \right\| _{2}^{2}-2\left< w_k-w_{k-1}^{*},w_{k}^{*}-w_{k-1}^{*} \right>,
    \end{equation}
    and estimate each member. Denote 
    \begin{equation*}
    \sigma _k:=A_{k}^{\top}A_kw_{k-1}-A_{k}^{\top}b_k,
    \end{equation*}
    which is the update direction of $w_{k-1}$, and correspondingly, use the following notation for reference,
    \begin{equation*}
    \sigma _{k}^{*}:=\left( A_{k}^{*} \right) ^{\top}A_{k}^{*}w_{k}^{*}-\left( A_{k}^{*} \right) ^{\top}b_{k}^{*}.        
    \end{equation*}
    By the update rule of $w$, $w_{k} = w_{k-1} - \eta\sigma_k$, we decompose the first term in \eqref{eq:w_wstart_decom}:
    \begin{equation}
       \left\| w_k-w_{k-1}^{*} \right\| _{2}^{2}=\left\| w_{k-1}-w_{k-1}^{*} \right\| _{2}^{2}+\eta ^2\left\| \sigma _k \right\| _{2}^{2}-2\eta \left< \sigma _k,w_{k-1}-w_{k-1}^{*} \right>.
    \end{equation}
    The identity
    \begin{align*}
        \sigma _k-\sigma _{k}^{*}=&A_{k}^{\top}A_k\left( w_{k-1}-w_{k-1}^{*} \right) +\left( A_{k}^{\top}A_k-\left( A_{k}^{*} \right) ^{\top}A_{k}^{*} \right) w_{k-1}^{*}
        \\
        &+\left( A_{k}^{*}-A_k \right) ^{\top}b_{k}^{*}+A_k\left( b_{k}^{*}-b_k \right) 
    \end{align*}
    and the observation that $\sigma_k^*=0$ yield
    \begin{align}
        \left\| \sigma _k \right\| _2=&\left\| \sigma _k-\sigma _{k}^{*} \right\| _2    \nonumber
        \\
        \le& \norm{A_{k}^{\top}A_k\left( w_{k-1}-w_{k-1}^{*} \right)}_2 + \norm{\left( A_{k}^{\top}A_k-\left( A_{k}^{*} \right) ^{\top}A_{k}^{*} \right) w_{k-1}^{*}}_2     \nonumber
        \\
        &+\norm{\left( A_{k}^{*}-A_k \right) ^{\top}b_{k}^{*}}_2+\norm{A_k\left( b_{k}^{*}-b_k \right) }_2  \nonumber
        \\
        \le& \left| \mathcal{S} \right|\left( 1+\gamma \right) \left\| w_{k-1}-w_{k-1}^{*} \right\| _2+\frac{2\left| \mathcal{S} \right|\left| \mathcal{A} \right|\left( 1+\gamma \right) C_{f\pi}}{1-\gamma}\left\| \pi _k-\pi _{k}^{*} \right\| _2    \nonumber
        \\
        &+\ \left| \mathcal{A} \right|\sqrt{\left| \mathcal{S} \right|\left( 1+\gamma \right)}C_{f\pi}\left\| \pi _k-\pi _{k}^{*} \right\| _2+\left| \mathcal{S} \right|\sqrt{\left| \mathcal{A} \right|}\left( 1+\gamma \right) \left( L_f+C_{f\pi} \right) \left\| \pi _k-\pi _{k}^{*} \right\| _{\infty},   \nonumber
        \\
        \le&\ \left| \mathcal{S} \right|\left( 1+\gamma \right) \left\| w_{k-1}-w_{k-1}^{*} \right\| _2+C_{\sigma \pi}\left\| \pi _k-\pi _{k}^{*} \right\| _{\infty}, \label{eq:ineq_sigma1}
    \end{align}
    with $C_{\sigma \pi}=\frac{2\left| \mathcal{S} \right|^{\frac{3}{2}}\left| \mathcal{A} \right|^{\frac{3}{2}}\left( 1+\gamma \right) C_{f\pi}}{1-\gamma}+\left| \mathcal{S} \right|\left| \mathcal{A} \right|^{\frac{3}{2}}\sqrt{\left( 1+\gamma \right)}C_{f\pi}+\left| \mathcal{S} \right|\sqrt{\left| \mathcal{A} \right|}\left( 1+\gamma \right) \left( L_f+C_{f\pi} \right)$, where the second inequality comes from the results of \cref{lem:pre_for_w}, and the last inequality is obtained by $\left\| \pi _k-\pi _{k}^{*} \right\| _2\le \sqrt{\left| \mathcal{S} \right|\left| \mathcal{A} \right|}\left\| \pi _k-\pi _{k}^{*} \right\| _{\infty}$. Subsequently, we bound the third term in \eqref{eq:w_wstart_decom} in a similar way,
    \begin{align}
        &\left< \sigma _k,w_{k-1}-w_{k-1}^{*} \right> \nonumber
        \\
        =&\left< \sigma _k-\sigma _{k}^{*},w_{k-1}-w_{k-1}^{*} \right>  \nonumber
        \\
        =&\left< A_{k}^{\top}A_k\left( w_{k-1}-w_{k-1}^{*} \right) ,w_{k-1}-w_{k-1}^{*} \right> +\left< \left( A_{k}^{\top}A_k-\left( A_{k}^{*} \right) ^{\top}A_{k}^{*} \right) w_{k-1}^{*},w_{k-1}-w_{k-1}^{*} \right>    \nonumber
        \\
        &+\left< \left( A_{k}^{*}-A_k \right) ^{\top}b_{k}^{*},w_{k-1}-w_{k-1}^{*} \right> +\left< A_k\left( b_{k}^{*}-b_k \right) ,w_{k-1}-w_{k-1}^{*} \right>     \nonumber
        \\
        \ge&\ \left( 1-\gamma \right) ^2\left\| w_{k-1}-w_{k-1}^{*} \right\| _{2}^{2}-C_{\sigma \pi}\left\| \pi _k-\pi _{k}^{*} \right\| _{\infty}\left\| w_{k-1}-w_{k-1}^{*} \right\| _2. \nonumber
    \end{align}
    By Young’s inequality,
    \begin{align*}
        \eta C_{\sigma \pi}\left\| \pi _k-\pi _{k}^{*} \right\| _{\infty}\left\| w_{k-1}-w_{k-1}^{*} \right\| _2\le& \frac{\eta ^2}{2\rho _1}\left\| \pi _k-\pi _{k}^{*} \right\| _{\infty}^{2}+\frac{C_{\sigma \pi}^{2}\rho _1}{2}\left\| w_{k-1}-w_{k-1}^{*} \right\| _{2}^{2}
        \\
        =&\frac{\eta ^2C_{\sigma \pi}^{2}}{2\left( 1-\gamma \right) ^2}\left\| \pi _k-\pi _{k}^{*} \right\| _{\infty}^{2}+\frac{\left( 1-\gamma \right) ^2}{2}\left\| w_{k-1}-w_{k-1}^{*} \right\| _{2}^{2},
    \end{align*}
    where we choose $\rho _1=\frac{\left( 1-\gamma \right) ^2}{C_{\sigma \pi}^{2}}$. It follows that 
    \begin{align}
        \eta \left< \sigma _k,w_{k-1}-w_{k-1}^{*} \right> \ge \frac{\left( 1-\gamma \right) ^2}{2}\left\| w_{k-1}-w_{k-1}^{*} \right\| _{2}^{2}-\frac{\eta ^2C_{\sigma \pi}^{2}}{2\left( 1-\gamma \right) ^2}\left\| \pi _k-\pi _{k}^{*} \right\| _{\infty}^{2}.
        \label{eq:ineq_sigma2}
    \end{align}
    Substituting the inequality \eqref{eq:ineq_sigma2} into \eqref{eq:w_wstart_decom} yields
    \begin{equation}\label{eq:w_descent_1}
        \left\| w_k-w_{k-1}^{*} \right\| _{2}^{2}\le \left( 1-\eta \left( 1-\gamma \right) ^2 \right) \left\| w_{k-1}-w_{k-1}^{*} \right\| _{2}^{2}+\eta ^2\left\| \sigma _k \right\| _{2}^{2}+\frac{\eta ^2C_{\sigma \pi}^{2}}{\left( 1-\gamma \right) ^2}\left\| \pi _k-\pi _{k}^{*} \right\| _{\infty}^{2}.
    \end{equation}
    Then, we bound the term $\left\| w_{k}^{*}-w_{k-1}^{*} \right\| _2$,
    \begin{align}
        &\left\| w_{k}^{*}-w_{k-1}^{*} \right\| _2   \nonumber
        \\
        =&\left\| \left( A_{k}^{*} \right) ^{-1}b_{k}^{*}-\left( A_{k-1}^{*} \right) ^{-1}b_{k-1}^{*} \right\| _2    \nonumber
        \\
        \le& \left\| \left( A_{k}^{*} \right) ^{-1}\left( b_{k}^{*}-b_{k-1}^{*} \right) \right\| _2+\left\| \left( A_{k}^{*} \right) ^{-1}\left( A_{k}^{*}-A_{k-1}^{*} \right) \left( A_{k-1}^{*} \right) ^{-1}b_{k-1}^{*} \right\| _2   \nonumber
        \\
        \le&\ \frac{L_{\vartheta}}{1-\gamma}\left\| x_k-x_{k-1} \right\| _2+\frac{\left| \mathcal{A} \right|}{\left( 1-\gamma \right) ^2}\sqrt{\left| \mathcal{S} \right|\left( 1+\gamma \right)}C_{f\pi}L_{\pi}\left\| x_k-x_{k-1} \right\| _2  \nonumber
        \\
        =&\ L_{w} \left\| x_k-x_{k-1} \right\| _2, \label{eq:w_descent_2}
    \end{align}
    with $L_w = \frac{1}{1-\gamma}\left( L_{\vartheta}+\frac{\left| \mathcal{A} \right|}{1-\gamma}\sqrt{\left| \mathcal{S} \right|\left( 1+\gamma \right)}C_{f\pi}L_{\pi} \right)$. Apply Young's inequality to the inner product term in \eqref{eq:w_wstart_decom},
    \begin{align}
        &-2\left< w_k-w_{k-1}^{*},w_{k}^{*}-w_{k-1}^{*} \right>     \nonumber
        \\
        =&-2\left< w_{k-1}-w_{k-1}^{*},w_{k}^{*}-w_{k-1}^{*} \right> +2\eta \left< \sigma _k,w_{k}^{*}-w_{k-1}^{*} \right>   \nonumber
        \\
        \le&\ \rho _2\left\| w_{k-1}-w_{k-1}^{*} \right\| _{2}^{2}+\frac{1}{\rho _2}\left\| w_{k}^{*}-w_{k-1}^{*} \right\| _{2}^{2}+\eta ^2\left\| \sigma _k \right\| _{2}^{2}+\left\| w_{k}^{*}-w_{k-1}^{*} \right\| _{2}^{2}    \nonumber
        \\
        =&\ \frac{\eta \left( 1-\gamma \right) ^2}{2}\left\| w_{k-1}-w_{k-1}^{*} \right\| _{2}^{2}+\frac{2}{\eta \left( 1-\gamma \right) ^2}\left\| w_{k}^{*}-w_{k-1}^{*} \right\| _{2}^{2} \nonumber
        \\
        &+\eta ^2\left\| \sigma _k \right\| _{2}^{2}+\left\| w_{k}^{*}-w_{k-1}^{*} \right\| _{2}^{2}, \label{eq:w_descent_3}
    \end{align}
    where we take $\,\,\rho _2=\frac{\eta \left( 1-\gamma \right) ^2}{2}$. Assembling \eqref{eq:w_descent_1}, \eqref{eq:w_descent_2} and \eqref{eq:w_descent_3}, we can estimate $\left\| w_k-w_{k}^{*} \right\| _{2}^{2}$ in \eqref{eq:w_wstart_decom},
    \begin{align*}
        \left\| w_k-w_{k}^{*} \right\| _{2}^{2}=&\left\| w_k-w_{k-1}^{*} \right\| _{2}^{2}+\left\| w_{k-1}^{*}-w_{k}^{*} \right\| _{2}^{2}-2\left< w_k-w_{k-1}^{*},w_{k}^{*}-w_{k-1}^{*} \right> 
        \\
        \le& \left( 1-\eta \left( 1-\gamma \right) ^2 \right) \left\| w_{k-1}-w_{k-1}^{*} \right\| _{2}^{2}+\eta ^2\left\| \sigma _k \right\| _{2}^{2}+\frac{\eta ^2C_{\sigma \pi}^{2}}{\left( 1-\gamma \right) ^2}\left\| \pi _k-\pi _{k}^{*} \right\| _{\infty}^{2}
        \\
        &+\left\| w_{k-1}^{*}-w_{k}^{*} \right\| _{2}^{2}+\frac{\eta \left( 1-\gamma \right) ^2}{2}\left\| w_{k-1}-w_{k-1}^{*} \right\| _{2}^{2}
        \\
        &+\frac{2}{\eta \left( 1-\gamma \right) ^2}\left\| w_{k}^{*}-w_{k-1}^{*} \right\| _{2}^{2}+\eta ^2\left\| \sigma _k \right\| _{2}^{2}+\left\| w_{k}^{*}-w_{k-1}^{*} \right\| _{2}^{2}\,\,
        \\
        =&\left( 1-\frac{\eta \left( 1-\gamma \right) ^2}{2} \right) \left\| w_{k-1}-w_{k-1}^{*} \right\| _{2}^{2}+\frac{\eta ^2C_{\sigma \pi}^{2}}{\left( 1-\gamma \right) ^2}\left\| \pi _k-\pi _{k}^{*} \right\| _{\infty}^{2}
        \\
        &+2\left( 1+\frac{1}{\eta \left( 1-\gamma \right) ^2} \right) \left\| w_{k}^{*}-w_{k-1}^{*} \right\| _{2}^{2}+2\eta ^2\left\| \sigma _k \right\| _{2}^{2}
        \\
        \le& \left( 1-\frac{\eta \left( 1-\gamma \right) ^2}{2} \right) \left\| w_{k-1}-w_{k-1}^{*} \right\| _{2}^{2}+\eta ^2C_{\sigma \pi}^{2}\left( \frac{1}{\left( 1-\gamma \right) ^2}+4 \right) \left\| \pi _k-\pi _{k}^{*} \right\| _{\infty}^{2}
        \\
        &+2L_{w}^{2}\left( 1+\frac{1}{\eta \left( 1-\gamma \right) ^2} \right) \left\| x_k-x_{k-1} \right\| _{2}^{2}+4\eta ^2\left| \mathcal{S} \right|^2\left( 1+\gamma \right) ^2\left\| w_{k-1}-w_{k-1}^{*} \right\| _{2}^{2},
    \end{align*}
    where the last inequality is earned by incorporating
    \begin{equation*}
        \left\| \sigma _k \right\| _{2}^{2}\le 2\left| \mathcal{S} \right|^2\left( 1+\gamma \right) ^2\left\| w_{k-1}-w_{k-1}^{*} \right\| _{2}^{2}+2C_{\sigma \pi \,\,}^{2}\left\| \pi _k-\pi _{k}^{*} \right\| _{\infty}^{2}
    \end{equation*}
    and 
    \begin{equation*}
        \left\| w_{k}^{*}-w_{k-1}^{*} \right\| _{2}^{2}\le L_{w}^{2}\left\| x_k-x_{k-1} \right\| _{2}^{2},
    \end{equation*}
    implied in \eqref{eq:ineq_sigma1} and \eqref{eq:w_descent_2}, respectively.
\end{proof}

\subsection{Convergence Properties of $\pi_k$ and $Q_k$}\label{sec:Bellman}
The softmax mapping, plays a significant role in the update rule of $\pi_k$ and $Q_k$, for which we begin this section by introducing some properties of it, following from \citep{cen2022fastNPG}. Given a $\theta\in\mathbb{R}^{\abs{\mdps}\abs{\mdpa}}$, for any $s\in\mdps$, use $\theta_s\in\mathbb{R}^{\abs{\mdpa}}$ to denote a vector with $\theta_{s}(a)=\theta_{sa}=\theta(s,a)$. Recall the softmax mapping:
\begin{align*}
    \texttt{softmax}:\ \ \mathbb{R}^{\abs{\mdps}\abs{\mdpa}} &\longrightarrow\mathbb{R}^{\abs{\mdps}\abs{\mdpa}}
    \\
    \theta &\longmapsto \kh{\frac{\exp \left( \theta \left( s,a \right) \right)}{\sum_{a^{\prime}}{\exp \left( \theta \left( s,a^{\prime} \right) \right)}}}_{sa}.
\end{align*}
Consider the typical component, where $\pi_q\in\mathbb{R}^{\abs{\mdpa}}$ is parameterized by $q\in\mathbb{R}^{\abs{\mdpa}}$,
\begin{equation*}
    \pi _q\left( a \right) =\frac{\exp \left( q\left( a \right) \right)}{\sum_{a^{\prime}}{\exp \left( q\left( a^{\prime} \right) \right)}}.
\end{equation*}
In this way,
\begin{align}
	&\left| \log \left( \left\| \exp \left( q_1 \right) \right\| _1 \right) -\log \left( \left\| \exp \left( q_2 \right) \right\| _1 \right) \right|   \nonumber
    \\
	=&\left| \left. \langle q_1-q_2,\left. \nabla _{\theta}\log \left\| \exp\mathrm{(}q)\parallel _1 \right) \right|_{q=q_c} \right. \rangle \right|   \nonumber
    \\
	\le& \left\| q_1-q_2 \right\| _{\infty}\left\| \left. \nabla _q\log \left\| \exp\mathrm{(}q)\parallel _1 \right) \right|_{q=q_c} \right\| _1   \nonumber
    \\
	=&\left\| q_1-q_2 \right\| _{\infty},  \label{eq:soft_ine_1}
\end{align}
where $q_c$ is a certain convex combination of $q_1$ and $q_2$, and 
\begin{equation*}
    \nabla _q\log \parallel \exp\mathrm{(}q)\parallel _1=\frac{1}{\parallel \exp\mathrm{(}q)\parallel _1}\exp\mathrm{(}q)=\pi _q.
\end{equation*}
It follows that
\begin{align}
    \left\| \log \pi _{q_1}-\log \pi _{q_2} \right\| _{\infty}\le& \left\| q_1-q_2 \right\| _{\infty}+\left| \log \left( \left\| \exp \left( q_1 \right) \right\| _1 \right) -\log \left( \left\| \exp \left( q_2 \right) \right\| _1 \right) \right|   \nonumber
    \\
    \le&\ 2\left\| q_1-q_2 \right\| _{\infty}.  \label{eq:soft_ine_2}
\end{align}

In \cref{alg:M-SoBiRL}, it adopts $Q_k$ to approximate the optimal soft Q-value function of $\pi_k$ dynamically, i.e., the environment $\mdp_\tau(x_k)$ based on which the value function is evaluated varies through the outer iterations. To this end, we will analyze the difference of value functions, incurred by the change of the environment parameterized by $x$. Recall that $\pi _{k}^{*}=\pi ^*\left( x_k \right)$, $Q_{k}^{*}=Q^*\left( x_k \right) $, $V_{k}^{*}=V^*\left( x_k \right) $. 
\begin{lemma}\label{lem:diff_environ}
    Under~\cref{assu:l}, given a policy $\pi$, for any $x_1,x_2\in\mathbb{R}^n$,
    \begin{align*}
        \left\|V^\pi(x_1)-V^\pi(x_2)\right\|_{\infty} \le \left\|Q^\pi(x_1)-Q^\pi(x_2)\right\|_{\infty} \le \frac{C_{rx}}{1-\gamma}\norm{x_1-x_2}_2,
    \end{align*}
\end{lemma}
\begin{proof}
By definitions of the soft value functions, for any $s\in\mdps$ and $a\in\mdpa$,
\begin{align}
\left|V^{\pi}_s(x_1)-V^{\pi}_s(x_2)\right| & = {\mathbb{E}_{a\sim \pi\kh{\cdot | s}}}\left[\left(-\tau \log \pi(a | s)+Q^{\pi}_{sa}(x_1)\right)-\left(-\tau \log \pi(a | s)+Q^{\pi}_{sa}(x_2)\right)\right] \nonumber
\\
&\le \left\|Q^\pi(x_1)-Q^\pi(x_2)\right\|_{\infty}, \label{eq:bound_V_diff}
\end{align}
and 
\begin{align}
\abs{Q^\pi_{sa}(x_1)-Q^\pi_{sa}(x_2)} &\le \abs{r_{sa}(x_1)-r_{sa}(x_2)} + \gamma\mathbb{E}_{s^{\prime} \sim P(\cdot | s, a)}\zkh{\left|V^{\pi}_{s^\prime}(x_1)-V^{\pi}_{s^\prime}(x_2)\right|} \label{eq:bound_Q_diff}.
\end{align}
Substituting \eqref{eq:bound_V_diff} into \eqref{eq:bound_Q_diff} yields
\begin{align*}
\left\|Q^\pi(x_1)-Q^\pi(x_2)\right\|_{\infty}\le C_{rx}\norm{x_1-x_2}_2 + \gamma \left\|Q^\pi(x_1)-Q^\pi(x_2)\right\|_{\infty},   
\end{align*}
which leads to the conclusion.
\end{proof}

We present the properties of the operator $\mathcal{T} _{\mathcal{M} _{\tau}\left( x \right)}:\mathbb{R}^{\abs{\mdps}\abs{\mdpa}} \mapsto \mathbb{R}^{\abs{\mdps}\abs{\mdpa}}$, following from \citep{haarnoja2018sac,cen2022fastNPG}.
\begin{proposition}
     The soft Bellman optimality operator $\mathcal{T} _{\mathcal{M} _{\tau}\left( x \right)}:\mathbb{R}^{\abs{\mdps}\abs{\mdpa}} \mapsto \mathbb{R}^{\abs{\mdps}\abs{\mdpa}}$ associated with $x\in\mathbb{R}^n$ satisfies the properties below.
     \begin{itemize}[label=$\bullet$,leftmargin=*]
      \item The optimal soft Q-function $Q^*(x)$ is a fixed point of $\mathcal{T} _{\mathcal{M} _{\tau}\left( x \right)}$, i.e.,
      \begin{equation}\label{eq:soft_opt_Q}
          \mathcal{T} _{\mathcal{M} _{\tau}\left( x \right)}\kh{Q^*(x)} = Q^*(x).
      \end{equation}
      \item $\mathcal{T} _{\mathcal{M} _{\tau}\left( x \right)}$ is a $\gamma$-contraction in the $\ell_{\infty}$ norm, i.e., for any $Q_1, Q_2 \in \mathbb{R}^{|\mathcal{S}||\mathcal{A}|}$, it holds that
      \begin{equation}\label{eq:soft_contract_Q}
          \left\| \mathcal{T} _{\mathcal{M} _{\tau}\left( x \right)}(Q^1)-\mathcal{T} _{\mathcal{M} _{\tau}\left( x \right)}(Q^2) \right\| _{\infty}\le \gamma \left\| Q^1-Q^2 \right\| _{\infty}.
      \end{equation}
      \item With any initial $Q_0$, applying $\mathcal{T} _{\mathcal{M} _{\tau}\left( x \right)}$ repeatedly converges to $Q^*(x)$ linearly, i.e., for any $N\in\mathbb{N}$,
      \begin{equation*}
          \left\| \mathcal{T} _{\mathcal{M} _{\tau}\left( x \right)}^{N}(Q^0)-Q^*\left( x \right) \right\| _{\infty}\le \gamma ^N\left\| Q^0-Q^*\left( x \right) \right\| _{\infty}.
      \end{equation*}
    \end{itemize}
\end{proposition}
\begin{proof}
    Substitute one of the consistency conditions from~\citep{nachum2017bridging},
    \begin{equation*}
        V_{s}^{*}\left( x \right) =\tau \log \left( \left\| \exp \left( Q_{s}^{*}\left( x \right) /\tau \right) \right\| _1 \right), 
    \end{equation*}
    into the definition of $\mathcal{T} _{\mathcal{M} _{\tau}(x)}$,
    \begin{align*}
        \mathcal{T} _{\mathcal{M} _{\tau}\left( x \right)}(Q^*\left( x \right) )(s,a)=&\ r_{sa}(x)+\gamma \underset{s^{\prime}\sim P(\cdot \mid s,a)}{\mathbb{E}}\left[ \tau \log \left( \left\| \exp \left( Q_{s}^{*}\left( x \right) /\tau \right) \right\| _1 \right) \right] 
        \\
        =&\ r_{sa}\left( x \right) +\gamma \underset{s^{\prime}\sim P(\cdot \mid s,a)}{\mathbb{E}}\left[ V_{s}^{*}\left( x \right) \right] 
        \\
        =&\ Q^*\left( x \right) \left( s,a \right). 
    \end{align*}
    Additionally, for any $Q^1, Q^2 \in \mathbb{R}^{|\mathcal{S}||\mathcal{A}|}$,
    \begin{align*}
        &\left| \mathcal{T} _{\mathcal{M} _{\tau}\left( x \right)}(Q^1)\left( s,a \right) -\mathcal{T} _{\mathcal{M} _{\tau}\left( x \right)}(Q^2)\left( s,a \right) \right|
        \\
        =&\ \gamma \tau \left| \underset{s^{\prime}\sim P(\cdot \mid s,a)}{\mathbb{E}}\left[ \log \left( \left\| \exp \left( Q_{s}^{1}\left( x \right) /\tau \right) \right\| _1 \right) \right] -\underset{s^{\prime}\sim P(\cdot \mid s,a)}{\mathbb{E}}\left[ \log \left( \left\| \exp \left( Q_{s}^{2}\left( x \right) /\tau \right) \right\| _1 \right) \right] \right|
        \\
        \le&\ \gamma \tau \left\| Q^1/\tau -Q^2/\tau \right\| _{\infty}
        \\
        =&\ \gamma \left\| Q^1-Q^2 \right\| _{\infty},
    \end{align*}
    where the inequality follows from \eqref{eq:soft_ine_1}. Combine \eqref{eq:soft_opt_Q} and \eqref{eq:soft_contract_Q},
    \begin{align*}
        \left\| \mathcal{T} _{\mathcal{M} _{\tau}\left( x \right)}^{N}(Q^0)-Q^*\left( x \right) \right\| _{\infty}=&\left\| \mathcal{T} _{\mathcal{M} _{\tau}\left( x \right)}^{N}(Q^0)-\mathcal{T} _{\mathcal{M} _{\tau}\left( x \right)}^{N}Q^*\left( x \right) \right\| _{\infty}
        \\
        \le&\ \gamma \left\| \mathcal{T} _{\mathcal{M} _{\tau}\left( x \right)}^{N-1}(Q^0)-\mathcal{T} _{\mathcal{M} _{\tau}\left( x \right)}^{N-1}Q^*\left( x \right) \right\| _{\infty}
        \\
        \le&\ \gamma ^2\left\| \mathcal{T} _{\mathcal{M} _{\tau}\left( x \right)}^{N-2}(Q^0)-\mathcal{T} _{\mathcal{M} _{\tau}\left( x \right)}^{N-2}Q^*\left( x \right) \right\| _{\infty}
        \\
        \le& \cdots 
        \\
        \le&\ \gamma ^N\left\| Q^0-Q^*\left( x \right) \right\| _{\infty}.
    \end{align*}
\end{proof}

Combining the contraction property of $\mathcal{T}_{\mdp_\tau\kh{x}}$ and \cref{lem:diff_environ} yield
  \begin{align}\label{eq:Q_descent}
      \left\| Q_{k+1}-Q_{k+1}^{*} \right\| _{\infty}\le \gamma ^N\left\| Q_k-Q_{k+1}^{*} \right\| _{\infty}\le \gamma ^N\left( \left\| Q_k-Q_{k}^{*} \right\| _{\infty}+\frac{C_{rx}}{1-\gamma}\left\| x_{k+1}-x_k \right\| _2 \right).
  \end{align}
Revisiting the update rule for $V_k$,
\begin{align*}
    \forall s \in \mathcal{S}:\quad V_k\left( s \right)=&\ \tau \log \left( \sum_a{\exp \left( Q_k\left( s,a \right) /\tau \right)} \right)
    \\
    =&\ \tau \log \left( \norm{Q_k(s)/\tau}_1\right),
\end{align*}
we apply the inequality \eqref{eq:soft_ine_1} and the  consistency condition \eqref{eq:comply_V},
\begin{align*}
    \left| V_k\left( s \right) -V_{s}^{*}\left( x_k \right) \right|=&\ \tau \left| \log \left( \left\| Q_k(s)/\tau \right\| _1 \right) -\log \left( \left\| Q_{k}^{*}(s)/\tau \right\| _1 \right) \right|
    \\
    \le&\ \tau \left\| Q_k(s)/\tau -Q_{k}^{*}(s)/\tau \right\| _{\infty}
    \\
    =&\left\| Q_k(s)-Q_{k}^{*}(s) \right\| _{\infty}.
\end{align*}
In this way, we bound the error term,
\begin{equation}\label{eq:bound_V_error}
    \left\| V_k-V_{k}^{*} \right\| _{\infty}\le \left\| Q_k-Q_{k}^{*} \right\| _{\infty}.
\end{equation}

\subsection{Convergence Analysis of M-SoBiRL}\label{sec:converge_MSoBiRL}
To begin with, some lemmas are provided for measuring the quality of the hyper-gradient estimator~$\widehat{\nabla}\phi$. Following this, we proceed to the convergence analysis of \cref{alg:M-SoBiRL}.

\begin{lemma}\label{lem:pre_for_hatphi}
    Under Assumptions~\ref{assu:f} and \ref{assu:r}, the following inequalities hold in \cref{alg:M-SoBiRL}.
    \begin{align*}
        \left\| \mathrm{diag}\left( \pi _k \right) \nabla _{\pi}f(x_k,\pi _k)-\mathrm{diag}\left( \pi _{k}^{*} \right) \nabla _{\pi}f(x_k,\pi _{k}^{*}) \right\| _2\le& \left( L_f+C_{f\pi} \right) \left\| \pi _k-\pi _{k}^{*} \right\| _{\infty}
        \\
        \left\| \nabla _x\varphi (x_k,V_k)^{\top}w_k-\nabla _x\varphi (x_k,V_{k}^{*})^{\top}w_{k}^{*} \right\| _2\le&\ \sqrt{\left| \mathcal{S} \right|}C_{rx}\left\| w_k-w_{k}^{*} \right\| _2
        \\
        &+\frac{2\left| \mathcal{S} \right|\left| \mathcal{A} \right|^{3/2}\sqrt{\left( 1+\gamma \right)}C_{f\pi}C_{rx}}{\tau \left( 1-\gamma \right)}\left\| V_{k}^{*}-V_k \right\| _{\infty}.
    \end{align*}
\end{lemma}
\begin{proof}
    Drawing on the Lipschitz and boundedness properties of $f$ in \cref{assu:f}, one can establish
    \begin{align*}
        &\left\| \mathrm{diag}\left( \pi _k \right) \nabla _{\pi}f(x_k,\pi _k)-\mathrm{diag}\left( \pi _{k}^{*} \right) \nabla _{\pi}f(x_k,\pi _{k}^{*}) \right\| _2
        \\
        \le& \left\| \mathrm{diag}\left( \pi _k \right) \left( \nabla _{\pi}f(x_k,\pi _k)-\nabla _{\pi}f(x_k,\pi _{k}^{*}) \right) \right\| _2+\left\| \left( \mathrm{diag}\left( \pi _k \right) -\mathrm{diag}\left( \pi _{k}^{*} \right) \right) \nabla _{\pi}f(x_k,\pi _{k}^{*}) \right\| _2
        \\
        \le&\ L_f\left\| \pi _k-\pi _{k}^{*} \right\| _{\infty}+C_{f\pi}\left\| \pi _k-\pi _{k}^{*} \right\| _{\infty}
        \\
        =&\left( L_f+C_{f\pi} \right) \left\| \pi _k-\pi _{k}^{*} \right\| _{\infty}.
    \end{align*}
    Compute the partial derivative
    \begin{equation*}
        \frac{\partial \varphi _s\left( x,v \right)}{\partial x_i}=\frac{\sum_a{\partial _{x_i}r_{sa}\left( x \right) \exp \left( \tau ^{-1}\left( r_{sa}\left( x \right) +\gamma \sum_{s^{\prime\prime}}{P_{sas^{\prime\prime}}v_{s^{\prime\prime}}} \right) \right)}}{\sum_a{\exp \left( \tau ^{-1}\left( r_{sa}\left( x \right) +\gamma \sum_{s^{\prime\prime}}{P_{sas^{\prime\prime}}v_{s^{\prime\prime}}} \right) \right)}},
    \end{equation*}
    and substitute $v=V^*(x)$ into it,
    \begin{equation*}
        \nabla _x\varphi _s\left( x,V^*\left( x \right) \right) = \sum_a{\nabla r_{sa}\left( x \right) \pi _{sa}^{*}\left( x \right)}.
    \end{equation*}
    Denote the auxiliary policy generated by softmax parameterization associated with $V$, i.e., for any $(s,a)\in\mdps\times\mdpa$,
    \begin{equation*}
        \pi _{sa}^{V}=\frac{\exp \left( \tau ^{-1}\left( r_{sa}\left( x \right) +\gamma \sum_{s^{\prime\prime}}{P_{sas^{\prime\prime}}V_{s^{\prime\prime}}} \right) \right)}{\sum_a{\exp \left( \tau ^{-1}\left( r_{sa}\left( x \right) +\gamma \sum_{s^{\prime\prime}}{P_{sas^{\prime\prime}}V_{s^{\prime\prime}}} \right) \right)}}.
    \end{equation*}
    In a similar fashion,
    \begin{equation*}
        \nabla _x\varphi _s\left( x,V \right) =\sum_a{\nabla r_{sa}\left( x \right) \pi _{sa}^{V}\left( x \right)}.
    \end{equation*}
    Applying \cref{lem:byproduct} with $H=1,I=1$, we obtain
    \begin{equation*}
        \left\| \nabla _x\varphi _s\left( x,V^*\left( x \right) \right) -\nabla _x\varphi _s\left( x,V \right) \right\| _2\le C_{rx}\sqrta\left\| \pi _{s}^{*}\left( x \right) -\pi _{s}^{V} \right\| _2.
    \end{equation*}
    Collecting all rows of $\nabla_x\varphi(x,v)$ leads to
    \begin{align*}
        \left\| \nabla _x\varphi \left( x,V \right) \right\| _2\le \left\| \nabla _x\varphi \left( x,V \right) \right\| _F\le \sqrt{\left| \mathcal{S} \right|}C_{rx},
    \end{align*}
    and
    \begin{align*}
        &\left\| \nabla _x\varphi \left( x,V^*\left( x \right) \right) -\nabla _x\varphi \left( x,V \right) \right\| _2
        \\
        \le& \left\| \nabla _x\varphi \left( x,V^*\left( x \right) \right) -\nabla _x\varphi \left( x,V \right) \right\| _F
        \\
        \le&\ C_{rx}\sqrta\left\| \pi ^*\left( x \right) -\pi ^V \right\| _2
        \\
        \le&\ \sqrt{\left| \mathcal{S} \right|}\left| \mathcal{A} \right|C_{rx}\left\| \pi ^*\left( x \right) -\pi ^V \right\| _{\infty}
        \\
        \le&\ 2\sqrt{\left| \mathcal{S} \right|}\left| \mathcal{A} \right|\tau ^{-1}C_{rx}\left\| V^*\left( x \right) -V \right\| _{\infty},
    \end{align*}
    where the last inequality results from \eqref{eq:soft_ine_2}. It follows that
    \begin{align*}
        &\left\| \nabla _x\varphi (x_k,V_k)^{\top}w_k-\nabla _x\varphi (x_k,V_{k}^{*})^{\top}w_{k}^{*} \right\| _2
        \\
        \le& \left\| \nabla _x\varphi (x_k,V_k)^{\top}\left( w_k-w_{k}^{*} \right) \right\| _2+\left\| \left( \nabla _x\varphi (x_k,V_k)-\nabla _x\varphi (x_k,V_{k}^{*}) \right) ^{\top}w_{k}^{*} \right\| _2
        \\
        \le& \left\| \nabla _x\varphi (x_k,V_k) \right\| _2\left\| w_k-w_{k}^{*} \right\| _2+\left\| \nabla _x\varphi (x_k,V_k)-\nabla _x\varphi (x_k,V_{k}^{*}) \right\| _2\left\| w_{k}^{*} \right\| _2
        \\
        \le&\ \sqrt{\left| \mathcal{S} \right|}C_{rx}\left\| w_k-w_{k}^{*} \right\| _2+\frac{2\left| \mathcal{S} \right|\left| \mathcal{A} \right|^{3/2}\sqrt{\left( 1+\gamma \right)}C_{f\pi}C_{rx}}{\tau \left( 1-\gamma \right)}\left\| V_{k}^{*}-V_k \right\| _{\infty}.
    \end{align*}    
\end{proof}

In the sequel, we analyze the error introduced by the estimator $\widehat{\nabla}\phi\kh{x,\pi,V,w}$ in approximating the true hyper-gradient $\nabla \phi(x)$, as stated in the following lemma.
\begin{lemma}\label{lem:esti_diff}
    Under Assumptions~\ref{assu:f} and \ref{assu:r}, the hyper-gradient estimator constructed in \cref{alg:M-SoBiRL},
    \begin{equation*}
    \begin{aligned}
         \widehat{\nabla} \phi \left(x_k,\pi_k,V_k,w_{k}\right)=\nabla_x f\left(x_k, \pi_k\right) + \frac{1}{\tau} {\nabla r(x_k)}^\top \operatorname{diag}\kh{\pi_k}\nabla_\pi f(x_k,\pi_k) -\frac{1}{\tau}\nabla_{x}\varphi(x_k,V_k)^\top w_k,
    \end{aligned}
    \end{equation*}
    satisfy
    \begin{align*}
        &\left\| \widehat{\nabla }\phi\left( x_k,\pi _k,V_k,w_k \right) -\nabla \phi \left( x_k \right) \right\| _2
        \\
        \le&\ L_{\widehat{\phi}\pi}\left\| \pi _k-\pi _{k}^{*} \right\| _{\infty} +\sqrt{\left| \mathcal{S} \right|}C_{rx}\left\| w_k-w_{k}^{*} \right\| _2+\frac{2\left| \mathcal{S} \right|\left| \mathcal{A} \right|^{3/2}\sqrt{\left( 1+\gamma \right)}C_{f\pi}C_{rx}}{\tau \left( 1-\gamma \right)}\left\| V_{k}^{*}-V_k \right\| _{\infty}.
    \end{align*}
    with
    \begin{equation*}
        L_{\widehat{\phi}\pi} = L_f+\tau ^{-1}\sqrt{\left| \mathcal{S} \right|\left| \mathcal{A} \right|}C_{rx}\left( L_f+C_{f\pi} \right).
    \end{equation*}
\end{lemma}
\begin{proof}
    Recall the counterpart true hyper-gradient,
    \begin{align*}
        \nabla \phi \left( x_k \right) =\nabla _xf\left( x_k,\pi _{k}^{*} \right) +\tau ^{-1}\left( \nabla r_k \right) ^{\top}\mathrm{diag}\left( \pi _{k}^{*} \right) \nabla _{\pi}f(x_k,\pi _{k}^{*})-\tau ^{-1}\nabla _x\varphi (x_k,V_{k}^{*})^{\top}w_{k}^{*}.
    \end{align*}
    Take into account the difference and substitute the results of \cref{lem:pre_for_hatphi},
    \begin{align*}
        &\left\| \widehat{\nabla }\phi _k\left( x_k,\pi _k,V_k,w_k \right) -\nabla \phi \left( x_k \right) \right\| _2
        \\
        \le& \left\| \nabla _xf\left( x_k,\pi _{k}^{*} \right) -\nabla _xf\left( x_k,\pi _k \right) \right\| _2+\tau ^{-1}\left\| \nabla r_k \right\| _2\left\| \mathrm{diag}\left( \pi _k \right) \nabla _{\pi}f(x_k,\pi _k)-\mathrm{diag}\left( \pi _{k}^{*} \right) \nabla _{\pi}f(x_k,\pi _{k}^{*}) \right\| _2
        \\
        &+\tau ^{-1}\left\| \nabla _x\varphi (x_k,v_k)^{\top}w_k-\nabla _x\varphi (x_k,v_{k}^{*})^{\top}w_{k}^{*} \right\| _2
        \\
        \le&\ L_f\left\| \pi _k-\pi _{k}^{*} \right\| _{\infty}+\tau ^{-1}\sqrt{\left| \mathcal{S} \right|\left| \mathcal{A} \right|}C_{rx}\left( L_f+C_{f\pi} \right) \left\| \pi _k-\pi _{k}^{*} \right\| _{\infty}
        \\
        &+\sqrt{\left| \mathcal{S} \right|}C_{rx}\left\| w_k-w_{k}^{*} \right\| _2+\frac{2\left| \mathcal{S} \right|\left| \mathcal{A} \right|^{3/2}\sqrt{\left( 1+\gamma \right)}C_{f\pi}C_{rx}}{\tau \left( 1-\gamma \right)}\left\| V_{k}^{*}-V_k \right\| _{\infty}
        \\
        =&\left( L_f+\tau ^{-1}\sqrt{\left| \mathcal{S} \right|\left| \mathcal{A} \right|}C_{rx}\left( L_f+C_{f\pi} \right) \right) \left\| \pi _k-\pi _{k}^{*} \right\| _{\infty}
        \\
        &+\sqrt{\left| \mathcal{S} \right|}C_{rx}\left\| w_k-w_{k}^{*} \right\| _2+\frac{2\left| \mathcal{S} \right|\left| \mathcal{A} \right|^{3/2}\sqrt{\left( 1+\gamma \right)}C_{f\pi}C_{rx}}{\tau \left( 1-\gamma \right)}\left\| V_{k}^{*}-V_k \right\| _{\infty}.
    \end{align*}
\end{proof}

Consequently, we arrive at the convergence analysis. Denote $\delta _{Q}^{k}:=\left\| Q_k-Q_{k}^{*} \right\| _{\infty}^{2}$, $\delta _{\pi}^{k}=\left\| \log \pi _k-\log \pi _{k}^{*} \right\| _{\infty}^{2}$ and $\delta _{w}^{k}:=\left\| w_k-w_{k}^{*} \right\| _{2}^{2}$. In this fashion, revisit the established results in previous subsections.

 \begin{itemize}[label=$\bullet$,leftmargin=*]
  \item The definition of \texttt{softmax} and the estimation \eqref{eq:soft_ine_2} reveal that
  \begin{equation}\label{eq:delta_pi}
      \delta _{\pi}^{k}\le \frac{4}{\tau ^2}\delta _{Q}^{k}.
  \end{equation}
    \item Apply~\eqref{eq:delta_pi} on the result of \cref{lem:descent_w},
  \begin{align}
      \delta _{w}^{k+1}\le& \left( 1-\frac{\eta \left( 1-\gamma \right) ^2}{2}+4\eta ^2\left| \mathcal{S} \right|^2\left( 1+\gamma \right) ^2 \right) \delta _{w}^{k}+\frac{4\eta ^2C_{\sigma \pi}^{2}}{\tau ^2}\left( \frac{1}{\left( 1-\gamma \right) ^2}+4 \right) \delta _{Q}^{k+1}    \nonumber
    \\
    &+2L_{w}^{2}\left( 1+\frac{1}{\eta \left( 1-\gamma \right) ^2} \right) \beta ^2\left\| \widehat{\nabla }\phi _k \right\| _{2}^{2}    \label{eq:delta_w}
  \end{align}
  \item Substitute the update rule of $x$ into \eqref{eq:Q_descent},
  \begin{align}
      \delta _{Q}^{k+1}\le 2\gamma ^{2N}\left( \delta _{Q}^{k}+\left( \frac{C_{rx}}{1-\gamma} \right) ^2\beta ^2\left\| \widehat{\nabla }\phi _k \right\| _{2}^{2} \right). \label{eq:delta_Q}
  \end{align}
  \item Incorporate \eqref{eq:bound_V_error},
  \begin{align*}
      \left\| V_k-V_{k}^{*} \right\| _{\infty}\le \left\| Q_k-Q_{k}^{*} \right\| _{\infty}
  \end{align*}
  into the estimation of \cref{lem:esti_diff},
  \begin{align}
      \left\| \widehat{\nabla }\phi _k-\nabla \phi \left( x_k \right) \right\| _{2}^{2}\le&\ 3L_{\widehat{\phi }\pi}^{2}\delta _{\pi}^{k}+3\left| \mathcal{S} \right|C_{rx}^{2}\delta _{w}^{k}+\frac{12\left| \mathcal{S} \right|^2\left| \mathcal{A} \right|^3\left( 1+\gamma \right) C_{f\pi}^{2}C_{rx}^{2}}{\tau ^2\left( 1-\gamma \right) ^2}\delta _{Q}^{k}   \nonumber
    \\
    \le&\ \frac{12}{\tau ^2}\left( L_{\widehat{\phi }\pi}^{2}+\frac{\left| \mathcal{S} \right|^2\left| \mathcal{A} \right|^3\left( 1+\gamma \right) C_{f\pi}^{2}C_{rx}^{2}}{\left( 1-\gamma \right) ^2} \right) \delta _{Q}^{k}+3\left| \mathcal{S} \right|C_{rx}^{2}\delta _{w}^{k}.\label{eq:diff_estimator}
  \end{align}
\end{itemize}

\begin{theorem}\label{the:proof_MSoBiRL}
    Under Assumptions~\ref{assu:f} and~\ref{assu:r}, in \cref{alg:M-SoBiRL}, we can choose constant step sizes~$\beta,\eta$, and the inner iteration number $N\sim\mathcal{O}(1)$. Then the iterates $\{x_k\}$~satisfy
    \begin{equation*}
        \frac{1}{K}\sum_{k=1}^K \norm{\nabla \phi(x_k)}_2^2 = \mathcal{O}\kh{\frac{1}{K}}.
    \end{equation*}
    Detailed parameter setting is listed as \eqref{eq:parameters}.
\end{theorem}
\begin{proof}
    We consider the merit function 
    \begin{equation*}
        \mathcal{L} _k:=\phi \left( x_k \right) +\zeta _Q\delta _{Q}^{k}+\zeta _w\delta _{w}^{k},
    \end{equation*}
    where the coefficients $\zeta_Q$ and $\zeta_w$ are to be determined later. By the Lipschitz property of $\phi(x)$ and the update rule of $x$ in \cref{alg:M-SoBiRL},
    \begin{align}
        \phi \left( x_{k+1} \right) \le&\ \phi \left( x_k \right) -\beta \left< \widehat{\nabla }\phi _k,\nabla \phi \left( x_k \right) \right> +\frac{L_{\phi}^{M}}{2}\beta ^2\left\| \widehat{\nabla }\phi _k \right\| _{2}^{2}     \nonumber
        \\
        =&\ \phi \left( x_k \right) -\frac{\beta}{2}\left( \left\| \widehat{\nabla }\phi _k \right\| _{2}^{2}+\left\| \nabla \phi \left( x_k \right) \right\| _{2}^{2}-\left\| \widehat{\nabla }\phi _k-\nabla \phi \left( x_k \right) \right\| _{2}^{2} \right) +\frac{L_{\phi}^{M}}{2}\beta ^2\left\| \widehat{\nabla }\phi _k \right\| _{2}^{2} \nonumber
        \\
        \le&\ \phi \left( x_k \right) -\frac{\beta}{2}\left\| \nabla \phi \left( x_k \right) \right\| _{2}^{2}+\frac{\beta}{2}\left\| \widehat{\nabla }\phi _k-\nabla \phi \left( x_k \right) \right\| _{2}^{2}+\left( \frac{L_{\phi}^{M}}{2}\beta ^2-\frac{\beta}{2} \right) \left\| \widehat{\nabla }\phi _k \right\| _{2}^{2}.  \label{eq:phi_descent}
    \end{align}
    Subsequently, it follows that
    \begin{align}
        \mathcal{L} _{k+1}-\mathcal{L} _k=&\ \phi \left( x_{k+1} \right) +\zeta _Q\delta _{Q}^{k+1}+\zeta _w\delta _{w}^{k+1}-\phi \left( x_k \right) -\zeta _Q\delta _{Q}^{k}-\zeta _w\delta _{w}^{k}      \nonumber
        \\
        \le& -\frac{\beta}{2}\left\| \nabla \phi \left( x_k \right) \right\| _{2}^{2}+\frac{\beta}{2}\left\| \widehat{\nabla }\phi _k-\nabla \phi \left( x_k \right) \right\| _{2}^{2}+\left( \frac{L_{\phi}^{M}}{2}\beta ^2-\frac{\beta}{2} \right) \left\| \widehat{\nabla }\phi _k \right\| _{2}^{2}   \nonumber
        \\
        &+\zeta _Q\delta _{Q}^{k+1}+\zeta _w\delta _{w}^{k+1}-\zeta _Q\delta _{Q}^{k}-\zeta _w\delta _{w}^{k}   \label{eq:L_ine_1}
        \\
        \le& -\frac{\beta}{2}\left\| \nabla \phi \left( x_k \right) \right\| _{2}^{2}+\frac{\beta}{2}\left( \frac{12}{\tau ^2}\left( L_{\widehat{\phi }\pi}^{2}+\frac{\left| \mathcal{S} \right|^2\left| \mathcal{A} \right|^3\left( 1+\gamma \right) C_{f\pi}^{2}C_{rx}^{2}}{\left( 1-\gamma \right) ^2} \right) \delta _{Q}^{k}+3\left| \mathcal{S} \right|C_{rx}^{2}\delta _{w}^{k} \right)    \nonumber
        \\
        &+\left( \frac{L_{\phi}^{M}}{2}\beta ^2-\frac{\beta}{2} \right) \left\| \widehat{\nabla }\phi _k \right\| _{2}^{2}   \nonumber
        \\
        &+2\gamma ^{2N}\left( \zeta _Q+\frac{4\eta ^2C_{\sigma \pi}^{2}}{\tau ^2}\zeta _w\left( \frac{1}{\left( 1-\gamma \right) ^2}+4 \right) \right) \left( \delta _{Q}^{k}+\left( \frac{C_{rx}}{1-\gamma} \right) ^2\beta ^2\left\| \widehat{\nabla }\phi _k \right\| _{2}^{2} \right)    \nonumber
        \\
        &+\zeta _w\left( 1-\frac{\eta \left( 1-\gamma \right) ^2}{2}+4\eta ^2\left| \mathcal{S} \right|^2\left( 1+\gamma \right) ^2 \right) \delta _{w}^{k}   \nonumber
        \\
        &+2L_{w}^{2}\zeta _w\left( 1+\frac{1}{\eta \left( 1-\gamma \right) ^2} \right) \beta ^2\left\| \widehat{\nabla }\phi _k \right\| _{2}^{2}-\zeta _Q\delta _{Q}^{k}-\zeta _w\delta _{w}^{k}   \label{eq:L_ine_2}
        \\
        =&-\frac{\beta}{2}\left\| \nabla \phi \left( x_k \right) \right\| _{2}^{2}+I_Q\delta _{Q}^{k}+I_w\delta _{w}^{k}+I_{\widehat{\phi }}\left\| \widehat{\nabla }\phi _k \right\| _{2}^{2},  \label{eq:L_descent}
    \end{align}
    with
    \begin{align*}
        I_w=&\ \frac{3\beta \left| \mathcal{S} \right|C_{rx}^{2}}{2}+\zeta _w\left( -\frac{\eta \left( 1-\gamma \right) ^2}{2}+4\eta ^2\left| \mathcal{S} \right|^2\left( 1+\gamma \right) ^2 \right) 
        \\
        I_Q=&\ 2\gamma ^{2N}\left( \zeta _Q+\frac{4\eta ^2C_{\sigma \pi}^{2}}{\tau ^2}\zeta _w\left( \frac{1}{\left( 1-\gamma \right) ^2}+4 \right) \right) -\zeta _Q
        \\
        &+\frac{6\beta}{\tau ^2}\left( L_{\widehat{\phi }\pi}^{2}+\frac{\left| \mathcal{S} \right|^2\left| \mathcal{A} \right|^3\left( 1+\gamma \right) C_{f\pi}^{2}C_{rx}^{2}}{\left( 1-\gamma \right) ^2} \right) 
        \\
        I_{\widehat{\phi }}=&\ \frac{L_{\phi}^{M}}{2}\beta ^2+2\beta ^2L_{w}^{2}\zeta _w\left( 1+\frac{1}{\eta \left( 1-\gamma \right) ^2} \right) -\frac{\beta}{2}
        \\
        &+2\gamma ^{2N}\left( \zeta _Q+\frac{4\eta ^2C_{\sigma \pi}^{2}}{\tau ^2}\zeta _w\left( \frac{1}{\left( 1-\gamma \right) ^2}+4 \right) \right) \left( \frac{C_{rx}}{1-\gamma} \right) ^2\beta ^2,
    \end{align*}
    where \eqref{eq:L_ine_1} comes from \eqref{eq:phi_descent} and \eqref{eq:L_ine_2} is the derivation of substituting \eqref{eq:delta_Q}, \eqref{eq:delta_w} and \eqref{eq:diff_estimator}.
    One sufficient condition for $I_Q,I_w,I_{\widehat{\phi}}<0$ is that
    \begin{equation}\label{eq:sufficient}
        \begin{aligned}
                \begin{cases}
        	\zeta _Q=\zeta _w=1,\\
        	\frac{3\beta \left| \mathcal{S} \right|C_{rx}^{2}}{2\eta \left( 1-\gamma \right) ^2}<\frac{1}{2}-4\eta \left| \mathcal{S} \right|^2\frac{\left( 1+\gamma \right) ^2}{\left( 1-\gamma \right) ^2},\\
        	\beta <\frac{\tau ^2}{8}\left( L_{\widehat{\phi }\pi}^{2}+\frac{\left| \mathcal{S} \right|^2\left| \mathcal{A} \right|^3\left( 1+\gamma \right) C_{f\pi}^{2}C_{rx}^{2}}{\left( 1-\gamma \right) ^2} \right) ^{-1},\\
        	\gamma ^{2N}<\frac{1}{8}\left( 1+\frac{4\eta ^2C_{\sigma \pi}^{2}}{\tau ^2}\left( \frac{1}{\left( 1-\gamma \right) ^2}+4 \right) \right) ^{-1},\\
        	\frac{L_{\phi}^{M}}{2}\beta +2\beta L_{w}^{2}\left( 1+\frac{1}{\eta \left( 1-\gamma \right) ^2} \right) +\frac{1}{4}\left( \frac{C_{rx}}{1-\gamma} \right) ^2\beta <\frac{1}{2}.
            \end{cases}      
        \end{aligned}
    \end{equation}
    Set the step size $\beta =\rho \eta $ with $\rho>0$. More precisely, we present the parameter configuration to guarantee \eqref{eq:sufficient}.
    \begin{equation}\label{eq:parameters}
        \begin{aligned}
        \begin{cases}
        	\zeta _Q=\zeta _w=1,\\
        	\eta <\min \left\{ 1,\frac{\left( 1-\gamma \right) ^2}{16\left| \mathcal{S} \right|^2\left( 1+\gamma \right) ^2} \right\} ,\\
        	\rho <\min \left\{ \frac{\left( 1-\gamma \right) ^2}{6\left| \mathcal{S} \right|C_{rx}^{2}},\frac{\left( 1-\gamma \right) ^2}{8L_{w}^{2}} \right\} ,\\
        	\beta <\min \left\{ \frac{\tau ^2}{8}\left( L_{\widehat{\phi }\pi}^{2}+\frac{\left| \mathcal{S} \right|^2\left| \mathcal{A} \right|^3\left( 1+\gamma \right) C_{f\pi}^{2}C_{rx}^{2}}{\left( 1-\gamma \right) ^2} \right) ^{-1},\frac{1}{4}\left( \frac{L_{\phi}^{M}}{2}+2L_{w}^{2}+\frac{1}{4}\left( \frac{C_{rx}}{1-\gamma} \right) ^2 \right) ^{-1} \right\} ,\\
        	\gamma ^{2N}<\frac{1}{8}\left( 1+\frac{4C_{\sigma \pi}^{2}}{\tau ^2}\left( \frac{1}{\left( 1-\gamma \right) ^2}+4 \right) \right) ^{-1},
        \end{cases}     
        \end{aligned}
    \end{equation}
    which means $N\sim\mathcal{O}(1)$ and the step sizes can be chosen as constants. In this way, \eqref{eq:L_descent} implies
    \begin{equation*}
        \mathcal{L} _{k+1}-\mathcal{L} _k\le -\frac{\beta}{2}\left\| \nabla \phi \left( x_k \right) \right\| _{2}^{2}.
    \end{equation*}
    Summing and telescoping it, we have
    \begin{equation*}
        \frac{1}{K}\sum_{k=0}^K{\left\| \nabla \phi \left( x_k \right) \right\| _{2}^{2}}\le \frac{2\mathcal{L} _0}{K\beta}=\frac{2}{K\beta}\left( \phi \left( x_0 \right) +\delta _{Q}^{0}+\delta _{w}^{0} \right) =\mathcal{O} \left( \frac{1}{K} \right) .
    \end{equation*}
\end{proof}

\section{ANALYSIS OF SOBIRL}\label{sec:ana_SoBiRL}
In this section, we prove the convergence of \cref{alg:SoBiRL}, SoBiRL. We outline the proof sketch of \cref{the:sobirl} here. Firstly, we characterize the distributional drift induced by two different policies~$\pi^1,\pi^2$ in \cref{lem:divergence_bound}. Then, the Lipschitz property of the hyper-gradient is clarified in~\cref{pro:hypergrad_Lipz}, and the quality of the hyper-gradient estimator is measured in~\cref{pro:lipz_hp_estimator}. Based on these two results, we arrive at~\cref{the:convergent_SoBiRL}, the convergence analysis of SoBiRL.

The next proposition, generalized from Lemma~1 in \citep{chakraborty2024parl}, considers the distribution $\rho \left( \left( d_1,d_2,\ldots,d_I \right) ;\pi\right)$,  with each trajectory $d_i=\{\kh{s_h^i,a_h^i}\}_{h=0}^{H-1}$ ($i=1,2,\ldots,I$) sampled from the trajectory distribution~$\rho\kh{d;\pi}$, i.e.,
\begin{equation*}
    P(d;\pi )=\bar{\rho}\kh{s_0}\zkh{\Pi_{h=0}^{H-2}\pi\kh{a_h|s_h}\bar{P}\kh{s_{h+1}|s_h,a_h}}\pi\kh{a_{H-1}|s_{H-1}},
\end{equation*}

\begin{lemma}\label{lem:divergence_bound}
    Denote the total variation between distributions by $D_F\kh{\cdot,\cdot}$. For any trajectory tuple $\left( d_1,d_2,\ldots,d_I \right)$, with each trajectory holding a finite horizon $H$, we have
    \begin{align}
        D_F\left( \rho \left( \left( d_1,d_2,\ldots,d_I \right) ;\pi^1 \right) ,\rho \left( \left( d_1,d_2,\ldots,d_I \right) ;\pi ^2 \right) \right) \le \frac{HI}{2}\sqrt{\left| \mathcal{A} \right|}\left\| \pi^1-\pi^2 \right\| _2.   \label{eq:bound_TV}
    \end{align}
    Specifically, for any $x_1,x_2\in\mathbb{R}^n$, take $\pi^1=\pi^*(x_1),\pi^2=\pi^*(x_2)$,
    \begin{align*}
        D_F\left( \rho \left( \left( d_1,d_2,\ldots,d_I \right) ;\pi^*(x_1) \right) ,\rho \left( \left( d_1,d_2,\ldots,d_I \right) ;\pi^*(x_2) \right) \right) \le \frac{HIL_\pi}{2}\sqrt{\left| \mathcal{A} \right|}\left\| x_1-x_2 \right\| _2.
    \end{align*}
\end{lemma}
\begin{proof}
    Firstly, we focus on the situation $I=1$, i.e., the one-trajectory distribution $\rho \left( d;\pi \right)$. Note that the total variation is the $F$-divergence induced by $F(x)=\frac{1}{2}\abs{x-1}$. The triangle inequality of the $F$-divergence leads to a decomposition:
    \begin{align}
        D_F\left( \rho \left( d;\pi ^1 \right) ,\rho \left( d;\pi ^2 \right) \right) \le D_F\left( \rho \left( d;\pi ^1 \right) ,\rho \left( d;p \right) \right) +D_F\left( \rho \left( d;p \right) ,\rho \left( d;\pi ^2 \right) \right),  \label{eq:trian_decomp}
    \end{align}
    with $p$ as a mixed policy executing $\pi^2$ at the first time-step of any trajectory and obeying $\pi^1$ for subsequent timesteps. In this way, by definition of $D_F$ and $p$,
    \begin{align}
        &D_F\left( \rho \left( d;\pi ^1 \right) ,\rho \left( d;p \right) \right) =\sum_d{P \left( d;p \right) F\left( \frac{P\left( d;\pi ^1 \right)}{P \left( d;p \right)} \right)} \nonumber
        \\
        =&\sum_d{P \left( d;p \right) F\left( \frac{\bar{\rho}\left( s_0 \right) \pi _{s_0a_0}^{1}\bar{P}\left( s_1|s_0,a_0 \right) \left[ \Pi _{h=1}^{H-2}\bar{P}\left( s_{h+1}|s_h,a_h \right) \pi _{s_ha_h}^{1} \right] \pi _{s_{H-1}a_{H-1}}^{1}}{\bar{\rho}\left( s_0 \right) \pi _{s_0a_0}^{2}\bar{P}\left( s_1|s_0,a_0 \right) \left[ \Pi _{h=1}^{H-2}\bar{P}\left( s_{h+1}|s_h,a_h \right) \pi _{s_ha_h}^{1} \right] \pi _{s_{H-1}a_{H-1}}^{1}} \right)}  \nonumber
        \\
        =&\sum_d{P \left( d;p \right) F\left( \frac{\pi _{s_0a_0}^{1}}{\pi _{s_0a_0}^{2}} \right)}.\label{eq:DF_pi1_p_1}
    \end{align}
    From \eqref{eq:DF_pi1_p_1}, it reveals that for different trajectory tuples, they share the same $F(\cdot)$ term if their initial states-actions pairs $(s_0,a_0)$ are identical, which implies
    \begin{align}
        &D_F\left( \rho \left( d;\pi ^1 \right) ,\rho \left( d;p \right) \right)    \nonumber
        \\
        =&\sum_{s_0,a_0}{\bar{\rho}\left( s_0 \right) \pi _{s_0a_0}^{2}F\left( \frac{\pi _{s_0a_0}^{1}}{\pi _{s_0a_0}^{2}} \right)}   \nonumber
        \\
        =&\sum_{s_0}{\bar{\rho}\left( s_0 \right) \sum_{a_0}{\pi _{s_0a_0}^{2}}}F\left( \frac{\pi _{s_0a_0}^{1}}{\pi _{s_0a_0}^{2}} \right)   \nonumber
        \\
        =&\sum_{s_0}{\bar{\rho}\left( s_0 \right) D_F\left( \pi _{\cdot |s_0}^{1},\pi _{\cdot |s_0}^{2} \right)}.    \label{eq:DF_pi1_p_2}
    \end{align} 
    Then we consider the second term in the triangle inequality decomposition \eqref{eq:trian_decomp}.
    \begin{align}
        &D_F\left( \rho \left( d;p \right) ,\rho \left( d;\pi ^2 \right) \right)    \nonumber
        \\
        =&\sum_d{P\left( d;\pi ^2 \right) F\left( \frac{P\left( d;p \right)}{P\left( d;\pi ^2 \right)} \right)} \nonumber
        \\
        =&\sum_d{P\left( d;\pi ^2 \right) F\left( \frac{\bar{\rho}\left( s_0 \right) \pi _{s_0a_0}^{2}\bar{P}\left( s_1|s_0,a_0 \right) \left[ \Pi _{h=1}^{H-2}\bar{P}\left( s_{h+1}|s_h,a_h \right) \pi _{s_ha_h}^{1} \right] \pi _{s_{H-1}a_{H-1}}^{1}}{\bar{\rho}\left( s_0 \right) \pi _{s_0a_0}^{2}\bar{P}\left( s_1|s_0,a_0 \right) \left[ \Pi _{h=1}^{H-2}\bar{P}\left( s_{h+1}|s_h,a_h \right) \pi _{s_ha_h}^{2} \right] \pi _{s_{H-1}a_{H-1}}^{2}} \right)}    \label{eq:explain_p}
        \\
        =&\sum_{s_0}{\bar{\rho}\left( s_0 \right) \sum_{a_0}{\pi _{s_0a_0}^{2}}}D_F\left( \rho \left( d^{H-1};\pi ^1 \right) ,\rho \left( d^{H-1};\pi ^2 \right) \right)    \nonumber
        \\
        \le& \sum_{s_0}{\bar{\rho}\left( s_0 \right) \sum_{a_0}{\pi _{s_0a_0}^{2}}}D_F\left( \rho \left( d^{H-1};\pi ^1 \right) ,\rho \left( d^{H-1};p^{\prime} \right) \right)     \nonumber
        \\
        &+\sum_{s_0}{\bar{\rho}\left( s_0 \right) \sum_{a_0}{\pi _{s_0a_0}^{2}}}D_F\left( \rho \left( d^{H-1};p^{\prime} \right) ,\rho \left( d^{H-1};\pi ^2 \right) \right) , \label{eq:explain_p_prime}
    \end{align}
    where $d^{H-1}$ denotes the trajectory with $H-1$ horizon, obtained by cutting $s_0$ and $a_0$. One can derive \eqref{eq:explain_p} since $p$ aligns with $\pi^2$ at the first timestep, and with $\pi^1$ at the following steps, and \eqref{eq:explain_p_prime} follows from constructing a new mixed policy $p^\prime$ in a similar way and applying the triangle inequality. A subsequent result is that we reduce the original divergence between $\pi^1$ and $\pi^2$ measured on $H$-length trajectories to divergence measured on one-length trajectories in \eqref{eq:DF_pi1_p_2} and $(H-1)$-length trajectories in \eqref{eq:explain_p_prime}. Repeating the process on the $(H-1)$-length trajectories in \eqref{eq:explain_p_prime} yields
    \begin{align}
        &D_F\left( \rho \left( d;\pi ^1 \right) ,\rho \left( d;\pi ^2 \right) \right) \nonumber
        \\
        \le& \sum_{h=0}^{H-1}{\mathbb{E} _{s^{\prime}\sim \rho _h}\left[ D_F\left( \pi _{\cdot |s^{\prime}}^{1},\pi _{\cdot |s^{\prime}}^{2} \right) \right]}   \nonumber
        \\
        \le&\ HD_F\left( \pi _{\cdot |s^*}^{1},\pi _{\cdot |s^*}^{2} \right) \nonumber
        \\
        =&\frac{H}{2}\left\| \pi _{\cdot |s^*}^{1}-\pi _{\cdot |s^*}^{2} \right\| _1    \nonumber
        \\
        \le&\ \frac{H}{2}\sqrt{\left| \mathcal{A} \right|}\left\| \pi ^1-\pi ^2 \right\| _2, \label{eq:explain_lip_pistar}
    \end{align}
    where $\rho_h$ is a state distribution related to the timestep $h$ and 
    \begin{align*}
        s^*=\argmax_{s^\prime\in\mdps}D_F\left( \pi _{\cdot |s^{\prime}}^{1},\pi _{\cdot |s^{\prime}}^{2} \right) .
    \end{align*}
    Note that $\rho \left( \left( d_1,d_2,\ldots,d_I \right) ;\pi \right) $ is the product measure of $\rho \left( d_i;\pi \right), \,\left( i=1,2,\ldots,I \right) $, i.e., $\rho \left( \left( d_1,d_2,\ldots,d_I \right) ;\pi \right) =\rho \left( d_1;\pi \right) \times \rho \left( d_2;\pi \right) \times \cdots \times \rho \left( d_I;\pi \right) $. By the total variation inequality for the product measure,
    \begin{align*}
        D_F\left( \rho \left( \left( d_1,d_2,\ldots,d_I \right) ;\pi ^1 \right) ,\rho \left( \left( d_1,d_2,\ldots,d_I \right) ;\pi ^2 \right) \right) \le& \sum_{i=1}^I{D_F\left( \rho \left( d_i;\pi ^1 \right) ,\rho \left( d_i;\pi ^2 \right) \right)}
        \\
        \le& \frac{HI}{2}\sqrt{\left| \mathcal{A} \right|}\left\| \pi ^1-\pi ^2 \right\| _2.
    \end{align*}
\end{proof}

A byproduct of \cref{lem:divergence_bound} is the lemma presented below, which measures the difference between the expectations of the same function evaluated on different distributions.
\begin{lemma}\label{lem:byproduct}
    If the vector-valued function $z\kh{d_1,d_2,\ldots,d_I;x}\in\mathbb{R}^n$ is bounded by $C$, i.e., $\norm{z\kh{d_1,d_2,\ldots,d_I;x}}_2\le C$, then for any trajectory tuple $\left( d_1,d_2,\ldots,d_I \right)$, with each trajectory holding a finite horizon $H$, and policies $\pi_1,\pi_2$, it holds
    \begin{align*}
        \norm{\mathbb{E} _{d_i\sim \rho \left( d;\pi_1 \right)}\left[ z\kh{d_1,d_2,\ldots,d_I;x} \right] -\mathbb{E} _{d_i\sim \rho \left( d;\pi_2 \right)}\left[ z\kh{d_1,d_2,\ldots,d_I;x} \right]}_2\le {CHI}\sqrt{\left| \mathcal{A} \right|}\left\| \pi_1-\pi_2 \right\| _2.
    \end{align*}
\end{lemma}
\begin{proof}
    We resort to the property of the Integral Probability Metric (IPM) \citep{sriperumbudur2009IPM},
    \begin{align*}
        \mathrm{sup}_{\left| \xi \right|\le B}\mathbb{E} _{d_i\sim \rho \left( d;\pi_1 \right)}\left[ \xi \right] -\mathbb{E} _{d_i\sim \rho \left( d;\pi_2  \right)}\left[ \xi \right] = 2B \times D_{F}\left( \rho \left( d;\pi ^*\left( x_1 \right) \right) ,\rho \left( d;\pi ^*\left( x_2 \right) \right) \right), 
    \end{align*}
    where $D_{F}$ is the total variation between distributions. Subsequently, it yields
    \begin{align}
        &\left| b^{\top}\left( \mathbb{E} _{d_i\sim \rho \left( d;\pi_1 \right)}\left[ z\left( d_1,d_2,\ldots,d_I;x_2 \right) \right] -\mathbb{E} _{d_i\sim \rho \left( d;\pi_2 \right)}\left[ z\left( d_1,d_2,\ldots,d_I;x_2 \right) \right] \right) \right| \nonumber
        \\
        \le&\ 2C\left\| b \right\| _2 D_{F}\left( \rho \left( \left( d_1,d_2,\ldots,d_I \right) ;\pi_1 \right) ,\rho \left( \left( d_1,d_2,\ldots,d_I \right) ;\pi_2 \right) \right) \nonumber
    \end{align}
    for any $b\in\mathbb{R}^n$, which means
    \begin{align}
        &\norm{\mathbb{E} _{d_i\sim \rho \left( d;\pi_1 \right)}\left[ z\left( d_1,d_2,\ldots,d_I;x_2 \right) \right] -\mathbb{E} _{d_i\sim \rho \left( d;\pi_2 \right)}\left[ z\left( d_1,d_2,\ldots,d_I;x_2 \right) \right] }_2 \nonumber
        \\
        \le&\ 2C\times D_{F}\left( \rho \left( \left( d_1,d_2,\ldots,d_I \right) ;\pi_1 \right) ,\rho \left( \left( d_1,d_2,\ldots,d_I \right) ;\pi_2 \right) \right).  \label{eq:combine_1}
    \end{align}
    Combining \eqref{eq:bound_TV} and \eqref{eq:combine_1} completes the proof.
\end{proof}

Drawing from the distributional drift investigated by \cref{lem:byproduct}, we can bound the difference between expectations, where both functions and distributions are associated with the upper-level variable $x$.
\begin{lemma}\label{lem:Zx_lipz}
    Under \cref{assu:r}, if $z\kh{d_1,d_2,\ldots,d_I;x}\in\mathbb{R}^n$ is bounded by $C$, i.e., $\norm{z\kh{d_1,d_2,\ldots,d_I;x}}_2\le C$, and is $L_z$-Lipschitz continuous with respect to $x$, then the function
    \begin{align*}
        Z(x) =& \mathbb{E} _{d_i\sim \rho \left( d;\pi ^*\left( x \right) \right)}\left[ z\kh{d_1,d_2,\ldots,d_I;x} \right]   \nonumber
    \end{align*}
    is $L_Z$-Lipschitz continuous with $L_Z=L_z+CHIL_\pi\sqrt{\left| \mathcal{A} \right|}$, i.e., for any $x_1,x_2\in\mathbb{R}^n$,
    \begin{equation*}
        \norm{Z(x_1) - Z(x_2)}_2\le L_Z\norm{x_1-x_2}_2.
    \end{equation*}
\end{lemma}
\begin{proof}
    To begin with, we decompose $Z(x_1)-Z(x_2)$ into two terms,
    \begin{align*}
        Z\left( x_1 \right) -Z\left( x_2 \right) =&\left( \mathbb{E} _{d_i\sim \rho \left( d;\pi ^*\left( x_1 \right) \right)}\left[ z\left( d_1,d_2,\ldots,d_I;x_1 \right) \right] -\mathbb{E} _{d_i\sim \rho \left( d;\pi ^*\left( x_1 \right) \right)}\left[ z\left( d_1,d_2,\ldots,d_I;x_2 \right) \right] \right) 
        \\
        &+\left( \mathbb{E} _{d_i\sim \rho \left( d;\pi ^*\left( x_1 \right) \right)}\left[ z\left( d_1,d_2,\ldots,d_I;x_2 \right) \right] -\mathbb{E} _{d_i\sim \rho \left( d;\pi ^*\left( x_2 \right) \right)}\left[ z\left( d_1,d_2,\ldots,d_I;x_2 \right) \right] \right). 
    \end{align*}
    The first term can be bounded by the Lipschitz continuity of $z$ since the two expectations share the same distribution $\rho(d;\pi^*(x_1)$. To control the second term, applying \cref{lem:byproduct} and considering the Lipschitz continuity of $\pi^*(x)$ lead to
    \begin{align*}
        \left\| \mathbb{E} _{d_i\sim \rho \left( d;\pi ^*\left( x_1 \right) \right)}\left[ z\left( d_1,d_2,\ldots,d_I;x_2 \right) \right] -\mathbb{E} _{d_i\sim \rho \left( d;\pi ^*\left( x_2 \right) \right)}\left[ z\left( d_1,d_2,\ldots,d_I;x_2 \right) \right] \right\| \le CHIL_\pi\sqrt{\abs{\mathcal{A}}} \left\| x_1-x_2 \right\| _2.
    \end{align*}
    It completes the proof by
    \begin{equation*}
        \begin{aligned}
            \left\| Z\left( x_1 \right) -Z\left( x_2 \right) \right\| _2\le L_z\left\| x_1-x_2 \right\| _2+C{HIL_{\pi}}\sqrt{\left| \mathcal{A} \right|}\left\| x_1-x_2 \right\| _2.
        \end{aligned}
    \end{equation*}
\end{proof}

The following proposition measures the quality of the implicit differentiation estimators. A similar result is introduced as an assumption in \citep{shen2024bilevelRL} (Lemma 18, condition (c)). We claim that it is satisfied in our setting.
\begin{proposition}\label{pro:property_VQ_estimator}
    Under \cref{assu:r}, for any upper-level variable $x\in\mathbb{R}^n$ and policies $\pi^1,\pi^2$, we have
    \begin{equation*}
        \begin{aligned}
            \max _{s,a}\left\| \widetilde{\nabla} Q_{sa}\kh{x,\pi^1} -\widetilde{\nabla }Q_{sa}\left( x,\pi^2 \right) \right\| _2&\le \frac{(1+\gamma)C_{rx}\sqrta}{\left( 1-\gamma \right) ^2}\left\| \pi^1 -\pi^2 \right\| _2,
            \\
            \max _s\left\| \widetilde{\nabla} V_{s}\left( x,\pi^1 \right) -\widetilde{\nabla }V_s\left( x,\pi^2 \right) \right\| _2 &\le \frac{(1+\gamma)C_{rx}\sqrta}{\left( 1-\gamma \right) ^2}\left\| \pi^1 -\pi^2 \right\| _2.
        \end{aligned}
    \end{equation*}
     Specifically, taking $\pi^1=\pi^*(x)$ and $\pi^2=\pi$ yields
    \begin{equation*}
        \begin{aligned}
            \max _{s,a}\left\| \nabla Q_{sa}^{*}\left( x \right) -\widetilde{\nabla }Q_{sa}\left( x,\pi \right) \right\| _2&\le \frac{(1+\gamma)C_{rx}\sqrta}{\left( 1-\gamma \right) ^2}\left\| \pi ^*\left( x \right) -\pi \right\| _2
            \\
            \max _s\left\| \nabla V_{s}^{*}\left( x \right) -\widetilde{\nabla }V_s\left( x,\pi \right) \right\| _2 &\le \frac{(1+\gamma)C_{rx}\sqrta}{\left( 1-\gamma \right) ^2}\left\| \pi ^*\left( x \right) -\pi \right\| _2.
        \end{aligned}
    \end{equation*}
\end{proposition}
\begin{proof}
    Recalling the expression of $\widetilde{\nabla }V_s\left( x,\pi \right) $
    \begin{equation*}
        \begin{aligned}
            \widetilde{\nabla }V_s(x,\pi )&=\mathbb{E} \left[ \sum_{t=0}^{\infty}{\gamma ^t}\nabla r_{s_t,a_t}(x) | s_0=s,\pi ,\mathcal{M} _{\tau}\left( x \right) \right],
        \end{aligned}
    \end{equation*}
    we obstain the boundedness
    \begin{equation*}
        \begin{aligned}
            \left\| \widetilde{\nabla }V_s\left( x,\pi \right) \right\| _2\le \frac{C_{rx}}{1-\gamma},
        \end{aligned}
    \end{equation*}
    and the recursive rule
    \begin{align}
        \widetilde{\nabla }V_s\left( x,\pi^1 \right) &=\sum_a{\pi^1 _{sa}\left( x \right) \nabla r_{sa}\left( x \right) +\gamma \mathbb{E} _{s^{\prime}\sim \rho \left( s_1;\pi^1 \right)}\left[ \widetilde{\nabla }V_{s^{\prime}}\left( x,\pi^1 \right) \right]}.  \label{eq:nabla_V1}
        \\
        \widetilde{\nabla }V_s\left( x,\pi^2 \right) &=\sum_a{\pi^2 _{sa}\left( x \right) \nabla r_{sa}\left( x \right) +\gamma \mathbb{E} _{s^{\prime}\sim \rho \left( s_1;\pi^2 \right)}\left[ \widetilde{\nabla }V_{s^{\prime}}\left( x,\pi^2 \right) \right]}.  \label{eq:nabla_V2}
    \end{align}
    Subtracting \eqref{eq:nabla_V1} by \eqref{eq:nabla_V2} yields
        \begin{align}
            \widetilde{\nabla }V_s\left( x,\pi^1 \right)  - \widetilde{\nabla }V_{s}\left( x,\pi^2 \right)=& \sum_a{\left( \pi _{sa}^{1}\left( x \right) -\pi _{sa}^2\left( x \right) \right) \nabla r_{sa}\left( x \right)}  \label{eq:diff_tilde_V}
            \\
            &+\gamma \left( \mathbb{E} _{s^{\prime}\sim \rho \left( s_1;\pi ^1 \right)}\left[\widetilde{\nabla }V_{s^{\prime}}\left( x,\pi^1 \right) \right] -\mathbb{E} _{s^{\prime}\sim \rho \left( s_1;\pi^2 \right)}\left[ \widetilde{\nabla }V_{s^{\prime}}\left( x,\pi^2 \right) \right] \right)    \nonumber
        \end{align}
    The first term can be bounded by applying \cref{lem:byproduct} with $I=1,H=1,\norm{\nabla r_{sa}(x)}_2\le C_{rx}$, and the second term can be decomposed into
    \begin{align}
        &\ \mathbb{E} _{s^{\prime}\sim \rho \left( s_1;\pi^1 \right)}\left[ \widetilde{\nabla }V_{s^{\prime}}\left( x,\pi^1 \right) \right] -\mathbb{E} _{s^{\prime}\sim \rho \left( s_1;\pi^2 \right)}\left[ \widetilde{\nabla }V_{s^{\prime}}\left( x,\pi^2 \right) \right] \nonumber
        \\
        =&\ \mathbb{E} _{s^{\prime}\sim \rho \left( s_1;\pi^1 \right)}\left[ \widetilde{\nabla }V_{s^{\prime}}\left( x,\pi^1 \right) -\widetilde{\nabla }V_{s^{\prime}}\left( x,\pi^2 \right) \right]  \nonumber
        \\
        &+\mathbb{E} _{s^{\prime}\sim \rho \left( s_1;\pi^1 \right)}\left[ \widetilde{\nabla }V_{s^{\prime}}\left( x,\pi^2 \right) \right] -\mathbb{E} _{s^{\prime}\sim \rho \left( s_1;\pi^2 \right)}\left[ \widetilde{\nabla }V_{s^{\prime}}\left( x,\pi^2 \right) \right],  \label{eq:bound_1}
    \end{align}
    where the term \eqref{eq:bound_1} can be bounded by applying \cref{lem:byproduct} with $I=1,\ H=2$ and $\norm{\widetilde{\nabla}V_{s^{\prime}} (x,\pi)}_2\le \frac{C_{rx}}{1-\gamma}$. Defining
    \begin{equation*}
        \Delta V:=\max _s\left\| \widetilde{\nabla }V_s\left( x,\pi^1 \right) -\widetilde{\nabla }V_s\left( x,\pi^2 \right) \right\| _2,
    \end{equation*}
    and collecting the observations above into \eqref{eq:diff_tilde_V} lead to
    \begin{equation*}
        \begin{aligned}
            \Delta V\le {C_{rx}}\sqrta\left\| \pi^1-\pi^2 \right\| _2+\gamma \left( \Delta V+\frac{2C_{rx}\sqrta}{1-\gamma}\left\| \pi ^1-\pi^2 \right\| _2 \right),
        \end{aligned}        
    \end{equation*}
    which means
    \begin{align*}
        \Delta V\le \frac{(1+\gamma)C_{rx}\sqrta}{\left( 1-\gamma \right) ^2}\left\| \pi ^1-\pi^2 \right\| _2.
    \end{align*}
    By the recursive rule
    \begin{equation*}
        \begin{aligned}
            \widetilde{\nabla }Q_{sa}\left( x,\pi \right) &=\nabla r_{sa}\left( x \right) +\gamma \mathbb{E} _{s^{\prime}\sim P\left( \cdot |sa \right)}\left[ \widetilde{\nabla }V_{s^{\prime}}\left( x,\pi \right) \right],
        \end{aligned}
    \end{equation*}
    we achieve the similar property of $\widetilde{\nabla}Q\kh{x,\pi}$.
\end{proof}

A byproduct of \cref{pro:property_VQ_estimator} is the Lipschitz smoothness of the implicit differentiations, $V^*(x)$ and $Q^*(x)$.
\begin{proposition}\label{pro:QV_Lipz}
    Under~\cref{assu:r}, given $x_1,x_2\in\mathbb{R}^n$, we have for any $s\in\mdps$, $a\in\mdpa$,
    \begin{align*}
        \left\| \nabla Q_{sa}^{*}(x_1)-\nabla Q_{sa}^{*}(x_2) \right\| _2&\le L_{V1}\left\| x_1-x_2 \right\| _2,
        \\
        \left\| \nabla V_{s}^{*}(x_1)-\nabla V_{s}^{*}(x_2) \right\| _2&\le L_{V1}\left\| x_1-x_2 \right\| _2,
    \end{align*}
    where $L_{V1}>0$ is the smoothness constant for the value functions,
    \begin{equation*}
        \begin{aligned}
        L_{V1}=\frac{(1+\gamma)C_{rx}L_\pi\sqrta}{\left( 1-\gamma \right) ^2}+\frac{L_r}{1-\gamma}.
        \end{aligned}
    \end{equation*}
\end{proposition}
\begin{proof}
    \cref{pro:first_order_nablaV} and the definition of of the implicit differentiation reveal that 
    \begin{align*}
        \nabla V_{s}^{*}(x)=\widetilde{\nabla }V_s(x,\pi ^*\left( x \right) )=\mathbb{E} \left[ \sum_{t=0}^{\infty}{\gamma ^t}\nabla r_{s_t,a_t}(x) | s_0=s,\pi ^*(x),\mathcal{M} _{\tau}\left( x \right) \right]. 
    \end{align*}
    Through this equality, the following decomposition is derived,
    \begin{align*}
        \nabla V_{s}^{*}(x_1)-\nabla V_{s}^{*}(x_2)=&\left( \widetilde{\nabla }V_s(x_1,\pi ^*\left( x_1 \right) )-\widetilde{\nabla }V_s(x_1,\pi ^*\left( x_2 \right) ) \right) 
        \\
        &+\left( \widetilde{\nabla }V_s(x_1,\pi ^*\left( x_2 \right) )-\widetilde{\nabla }V_s(x_2,\pi ^*\left( x_2 \right) ) \right).
    \end{align*}
    Apply \cref{pro:property_VQ_estimator} and Lipschitz continuity of $\pi^*$,
    \begin{align*}
        \left\| \widetilde{\nabla }V_s(x_1,\pi ^*\left( x_1 \right) )-\widetilde{\nabla }V_s(x_1,\pi ^*\left( x_2 \right) ) \right\| _2\le&\ \frac{(1+\gamma)C_{rx}\sqrta}{\left( 1-\gamma \right) ^2}\left\| \pi ^*\left( x_1 \right) -\pi ^*\left( x_2 \right) \right\| _2
        \\
        \le&\ \frac{(1+\gamma)C_{rx}L_\pi\sqrta}{\left( 1-\gamma \right) ^2}\left\| x_1-x_2 \right\| _2.     
    \end{align*}
    Additionally, Lipschitz smoothness of of $r$ yields
    \begin{align*}
        &\left\| \widetilde{\nabla }V_s(x_1,\pi ^*\left( x_2 \right) )-\widetilde{\nabla }V_s(x_2,\pi ^*\left( x_2 \right) ) \right\| _2
        \\
        \le&\ \mathbb{E} \left[ \sum_{t=0}^{\infty}{\gamma ^t}\left\| \nabla r_{s_t,a_t}(x_1)-\nabla r_{s_t,a_t}(x_2) \right\| _2\,\,| s_0=s,\pi ^*\left( x_2 \right) \right] 
        \\
        \le&\ \frac{L_r}{1-\gamma}\left\| x_1-x_2 \right\| _2.
    \end{align*}
    Consequently,
    \begin{align*}
        \left\| \nabla V_{s}^{*}(x_1)-\nabla V_{s}^{*}(x_2) \right\| _2\le \frac{(1+\gamma)C_{rx}L_\pi\sqrta}{\left( 1-\gamma \right) ^2}\left\| x_1-x_2 \right\| _2+\frac{L_r}{1-\gamma}\left\| x_1-x_2 \right\| _2=L_{V1}\left\| x_1-x_2 \right\| _2.
    \end{align*}
    In a similar fashion, we can also conclude
    \begin{equation*}
        \left\| \nabla Q_{sa}^{*}(x_1)-\nabla Q_{sa}^{*}(x_2) \right\| _2\le L_{V1}\left\| x_1-x_2 \right\| _2.
    \end{equation*}
\end{proof}

In this way, it establishes the Lipschitz smoothness of the hyper-objective $\phi(x)$, which plays a key role in the convergence analysis.
\begin{proposition}\label{pro:hypergrad_Lipz}
    Under the Assumptions~\ref{assu:l} and~\ref{assu:r}, $\nabla \phi(x)$ is $L_{\phi}$-Lipschitz continuous, i.e., for any $x_1,x_2\in\mathbb{R}^n$,
    \begin{equation*}
       \norm{\nabla \phi(x_1) - \nabla \phi(x_2)}_2\le L_{\phi}\norm{x_1-x_2}_2, 
    \end{equation*}
    where $L_{\phi}>0$ is defined by
    \begin{equation*}
        \begin{aligned}
            L_\phi = L_{l1}+L_lHIL_\pi\sqrta+2\tau ^{-1}HI\left( C_lL_{V_1}+\frac{C_{rx}}{1-\gamma}L_l \right) +\frac{2H^2I^2C_lC_{rx}L_{\pi}}{\tau \left( 1-\gamma \right)}\sqrta.
        \end{aligned}
    \end{equation*}
\end{proposition}
\begin{proof}
    Given
    \begin{align}
        \nabla \phi(x) =&\ \mathbb{E} _{d_i\sim \rho \left( d;\pi ^*\left( x \right) \right)}\left[ \nabla l\left( d_1,d_2,\ldots,d_I;x \right) \right]   \label{eq:decom_hypergrad}
        \\
        &+\tau ^{-1}\mathbb{E} _{d_i\sim \rho \left( d;\pi ^*\left( x \right) \right)}\left[ l\left( d_1,d_2,\ldots,d_I;x \right) \left( \sum_i{\sum_h{\nabla \left( Q_{s_{h}^{i}a_{h}^{i}}^{*}\left( x \right) -V_{s_{h}^{i}}^{*}\left( x \right) \right)}} \right) \right],     \nonumber
    \end{align}
    it is deduced from \cref{lem:Zx_lipz} that the first term in \eqref{eq:decom_hypergrad} is $\kh{L_{l1}+L_lHIL_\pi\sqrta}$-Lipschitz continuous since $\nabla l\kh{d_1,d_2,\ldots,d_I;x}$ is $L_l$-boundedness and $L_{l1}$-Lipschitz continuous, revealed by \cref{assu:l}. Subsequently, we consider the second term in \eqref{eq:decom_hypergrad}. Specifically, assembling the $L_l$-Lipschitz continuity of $l$, the $L_{V1}$-Lipschitz continuity of $\nabla V^*_{s}$, and the $L_{V1}$-Lipschitz continuity of $\nabla Q^*_{sa}$, along with the boundedness of these functions, we conclude the function defined by
    \begin{equation*}
        \begin{aligned}
            z\kh{d_1,d_2,\ldots,d_I;x}:=l\left( d_1,d_2,\ldots,d_I;x \right) \left( \sum_i{\sum_h{\nabla \left( Q_{s_{h}^{i}a_{h}^{i}}^{*}\left( x \right) -V_{s_{h}^{i}}^{*}\left( x \right) \right)}} \right)
        \end{aligned}
    \end{equation*}
    is bounded by 
    \begin{equation*}
        \left\| z\left( d_1,d_2,\ldots,d_I;x \right) \right\| _2\le 2HIC_l\frac{C_{rx}}{1-\gamma},
    \end{equation*}
    and $L_z$-Lipschitz continuous with
    \begin{equation*}
        L_z = 2HI\kh{C_lL_{V_1}+\frac{C_{rx}}{1-\gamma}L_l}.
    \end{equation*}
    Applying \cref{lem:Zx_lipz} on the second term in \eqref{eq:decom_hypergrad} implies it is $\kh{L_z+\frac{2H^2I^2C_lC_{rx}L_{\pi}}{1-\gamma}\sqrt{\left| \mathcal{A} \right|}}$-Lipschitz continuous. Combining the Lipschitz constants of the two terms in \eqref{eq:decom_hypergrad} provides the Lipschitz constant of $\nabla \phi(x)$, i.e.,
    \begin{equation*}
        \begin{aligned}
            L_{\phi}=&\kh{L_{l1}+L_lHIL_\pi\sqrta}  +\tau ^{-1}\left( L_z+\frac{2H^2I^2C_lC_{rx}L_{\pi}}{1-\gamma}\sqrt{\left| \mathcal{A} \right|} \right) 
            \\
            =&\ L_{l1}+L_lHIL_\pi\sqrta+2\tau ^{-1}HI\left( C_lL_{V_1}+\frac{C_{rx}}{1-\gamma}L_l \right) +\frac{2H^2I^2C_lC_{rx}L_{\pi}}{\tau \left( 1-\gamma \right)}\sqrta.
        \end{aligned}
    \end{equation*}
    
\end{proof}

Next, we analyze the error brought by the estimator $\widetilde{\nabla}\phi\kh{x,\pi}$ to approximate the true hyper-gradient $\nabla \phi(x)$ in the following proposition. $\widetilde{\nabla}\phi\kh{x,\pi}$.
\begin{proposition}\label{pro:lipz_hp_estimator}   
    Under Assumptions~\ref{assu:l} and \ref{assu:r}, for any upper-level variable $x\in\mathbb{R}^n$ and policies $\pi^1,\pi^2$, we have
    \begin{equation*}
        \begin{aligned}
            \norm{\widetilde{\nabla} \phi(x,\pi^1)-\widetilde{\nabla} \phi(x,\pi^2)}_2\le L_{\widetilde{\phi }}\norm{\pi^1-\pi^2}_2,
        \end{aligned}
    \end{equation*}
    with $L_{\widetilde{\phi }}={HIL_l\sqrta}+\frac{2\tau^{-1}C_lC_{rx}HI\sqrt{\left| \mathcal{A} \right|}}{1-\gamma}\left( \frac{1+\gamma}{1-\gamma}+HI \right)$.
    Specifically, taking $\pi^1=\pi^*(x)$ and $\pi^2=\pi$ yields
        \begin{equation*}
        \begin{aligned}
            \norm{{\nabla} \phi(x)-\widetilde{\nabla} \phi(x,\pi)}_2\le L_{\widetilde{\phi }}\norm{\pi^*(x)-\pi}_2.
        \end{aligned}
    \end{equation*}
\end{proposition}
\begin{proof}
    Considering $\norm{\nabla l}_2\le L_l$ and applying \cref{lem:byproduct}, we obtain
    \begin{align*}
        \left\| \mathbb{E} _{d_i\sim \rho \left( d;\pi ^1 \right)}\left[ \nabla l\left( d_1,d_2,\ldots,d_I;x \right) \right] -\mathbb{E} _{d_i\sim \rho \left( d;\pi ^2 \right)}\left[ \nabla l\left( d_1,d_2,\ldots,d_I;x \right) \right] \right\| _2\le {HIL_l\sqrta}\left\| \pi ^1-\pi ^2 \right\| _2.
    \end{align*}
    Denote $z\left( d_1,d_2,\ldots,d_I; x,\pi \right) :=l\left( d_1,d_2,\ldots,d_I;x \right) \left( \sum_i{\sum_h{\widetilde{\nabla }\left( Q_{s_{h}^{i}a_{h}^{i}}-V_{s_{h}^{i}} \right) \left( x,\pi \right)}} \right)$. Taking into account the boundedness and Lipschitz continuity of $l$ in \cref{assu:l}, and of $\widetilde{\nabla}Q,\ \widetilde{\nabla}V$ in~\cref{pro:property_VQ_estimator}, we can establish the boundedness,
    \begin{align}
        \left\| z\left( d_1,d_2,\ldots,d_I;x,\pi \right) \right\| _2\le 2C_lHI\frac{C_{rx}}{1-\gamma}, \label{eq:boundz}
    \end{align}
    and the Lipschitz continuity of $z\left( d_1,d_2,\ldots,d_I;x,\pi \right) $, i.e.,
    \begin{align}
        \left\| z\left( d_1,d_2,\ldots,d_I;x,\pi _1 \right) -z\left( d_1,d_2,\ldots,d_I;x,\pi _2 \right) \right\| _2\le 2C_lHI\frac{(1+\gamma )C_{rx}\sqrt{\left| \mathcal{A} \right|}}{\left( 1-\gamma \right) ^2}\left\| \pi _1-\pi _2 \right\| _2. \label{eq:Lipz}
    \end{align}
    In this way,
    \begin{align}
        &\left\| \mathbb{E} _{d_i\sim \rho \left( d;\pi ^1 \right)}\left[ z\left( x,\pi _1 \right) \right] -\mathbb{E} _{d_i\sim \rho \left( d;\pi ^2 \right)}\left[ z\left( x,\pi _2 \right) \right] \right\| _2 \nonumber
        \\
        \le& \left\| \mathbb{E} _{d_i\sim \rho \left( d;\pi ^1 \right)}\left[ z\left( x,\pi _1 \right) \right] -\mathbb{E} _{d_i\sim \rho \left( d;\pi ^1 \right)}\left[ z\left( x,\pi _2 \right) \right] \right\| _2 \label{eq:boundz1}
        \\
        &+\left\| \mathbb{E} _{d_i\sim \rho \left( d;\pi ^1 \right)}\left[ z\left( x,\pi _2 \right) \right] -\mathbb{E} _{d_i\sim \rho \left( d;\pi ^2 \right)}\left[ z\left( x,\pi _2 \right) \right] \right\| _2 \label{eq:boundz2}
        \\
        \le&\ 2C_lHI\frac{(1+\gamma )C_{rx}\sqrt{\left| \mathcal{A} \right|}}{\left( 1-\gamma \right) ^2}\left\| \pi _1-\pi _2 \right\| _2+ 2C_lH^2I^2\frac{C_{rx}}{1-\gamma}\sqrta\left\| \pi _1-\pi _2 \right\| _2, \nonumber
        \\
        =&\ \frac{2C_lC_{rx}HI\sqrt{\left| \mathcal{A} \right|}}{1-\gamma}\left( \frac{1+\gamma}{1-\gamma}+HI \right)  \left\| \pi _1-\pi _2 \right\| _2, \nonumber
    \end{align}
    where we bound \eqref{eq:boundz1} by the Lipschitz continuity \eqref{eq:Lipz}, and bound \eqref{eq:boundz2} by applying \cref{lem:byproduct} with~\eqref{eq:boundz}.
    Given
    \begin{align}
    \widetilde{\nabla} \phi(x,\pi) =&\ \mathbb{E} _{d_i\sim \rho \left( d;\pi \right)}\left[ \nabla l\left( d_1,d_2,\ldots,d_I;x \right) \right]   \nonumber
    \\
    &+\tau ^{-1}\mathbb{E} _{d_i\sim \rho \left( d;\pi \right)}\left[ l\left( d_1,d_2,\ldots,d_I;x \right) \left( \sum_i{\sum_h{\widetilde{\nabla} \left( Q_{s_{h}^{i}a_{h}^{i}}-V_{s_{h}^{i}}\right)\kh{x,\pi}}} \right) \right],    \nonumber
    \end{align}
    we have
    \begin{equation*}
        \begin{aligned}
            &\left\| \widetilde{\nabla }\phi (x,\pi ^1)-\widetilde{\nabla }\phi (x,\pi ^2) \right\| _2
            \\
            \le& \left\| \mathbb{E} _{d_i\sim \rho \left( d;\pi ^1 \right)}\left[ \nabla l\left( d_1,d_2,\ldots,d_I;x \right) \right] -\mathbb{E} _{d_i\sim \rho \left( d;\pi ^2 \right)}\left[ \nabla l\left( d_1,d_2,\ldots,d_I;x \right) \right] \right\| _2
            \\
            &+\tau ^{-1}\left\| \mathbb{E} _{d_i\sim \rho \left( d;\pi ^1 \right)}\left[ z\left( d_1,d_2,\ldots,d_I;x,\pi ^1 \right) \right] -\mathbb{E} _{d_i\sim \rho \left( d;\pi ^2 \right)}\left[ z\left( d_1,d_2,\ldots,d_I;x,\pi ^2 \right) \right] \right\| _2
            \\
            \le &\ {HIL_l\sqrta}\left\| \pi ^1-\pi ^2 \right\| _2+\frac{2\tau^{-1}C_lC_{rx}HI\sqrt{\left| \mathcal{A} \right|}}{1-\gamma}\left( \frac{1+\gamma}{1-\gamma}+HI \right)  \left\| \pi _1-\pi _2 \right\| _2
            \\
            =&\ L_{\widetilde{\phi }}\left\| \pi _1-\pi _2 \right\| _2.
        \end{aligned}
    \end{equation*}
\end{proof}

With all the lemmas and propositions now proven, we turn our attention to presenting the convergence theorem for the proposed algorithm, SoBiRL.
\begin{theorem}\label{the:convergent_SoBiRL}
    Under Assumptions~\ref{assu:l} and~\ref{assu:r} and given the accuracy in \cref{alg:SoBiRL}, we can set the constant step size~$\beta < \frac{1}{2L_\phi}$, then the iterates $\{x_k\}$~satisfy
    \begin{equation*}
        \frac{1}{K}\sum_{k=1}^K{\left\| \nabla \phi \left( x_k \right) \right\|_2 ^2}\le \frac{\phi \left( x_1 \right) -\phi ^*}{K\left( \frac{\beta}{2}-\beta ^2L_{\phi} \right)}+\frac{1+2\beta L_{\phi}}{1-2\beta L_{\phi}}L^2_{\widetilde{\phi }}\ \epsilon\ ,
    \end{equation*}
    where $\phi^*$ is the minimum of the hyper-objective $\phi(x)$, and
    \begin{equation*}
        \begin{aligned}
             L_\phi =&\  L_{l1}+L_lHIL_\pi\sqrta+2\tau ^{-1}HI\left( C_lL_{V_1}+\frac{C_{rx}}{1-\gamma}L_l \right) +\frac{2H^2I^2C_lC_{rx}L_{\pi}}{\tau \left( 1-\gamma \right)}\sqrta,
            \\
            L_{\widetilde{\phi }}=&\ {HIL_l\sqrta}+\frac{2\tau^{-1}C_lC_{rx}HI\sqrt{\left| \mathcal{A} \right|}}{1-\gamma}\left( \frac{1+\gamma}{1-\gamma}+HI \right).
        \end{aligned}
    \end{equation*}
\end{theorem}
\begin{proof}
    In this proof, all the notation $\norm{\cdot}$ without subscripts denotes the 2-norm for simplicity. Based on the $L_\phi$-Lipschitz smoothness of $\phi(x)$ revealed by \cref{pro:hypergrad_Lipz}, a gradient descent step results in the decrease in the hyper-objective $\varphi$:
    \begin{align}
        \phi \left( x_{k+1} \right) \le&\ \phi \left( x_k \right) +\left< \nabla \phi \left( x_k \right) ,x_{k+1}-x_k \right> +\frac{L_{\phi}}{2}\left\| x_{k+1}-x_k \right\| ^2 \nonumber
        \\
        =&\ \phi \left( x_k \right) -\beta \left< \nabla \phi \left( x_k \right) ,\widetilde{\nabla }\phi \left( x_k,\pi _k \right) \right> +\frac{\beta ^2L_{\phi}}{2}\left\| \widetilde{\nabla }\phi \left( x_k,\pi _k \right) \right\| ^2    \label{eq:sobirl_explain_xupdate}
        \\
        =&\ \phi \left( x_k \right) -\beta \left< \nabla \phi \left( x_k \right) ,\widetilde{\nabla }\phi \left( x_k,\pi _k \right) -\nabla \phi \left( x_k \right) +\nabla \phi \left( x_k \right) \right> +\frac{\beta ^2L_{\phi}}{2}\left\| \widetilde{\nabla }\phi \left( x_k,\pi _k \right) \right\| ^2 \nonumber
        \\
        =&\ \phi \left( x_k \right) -\beta \left\| \nabla \phi \left( x_k \right) \right\| ^2-\beta \left< \nabla \phi \left( x_k \right) ,\widetilde{\nabla }\phi \left( x_k,\pi _k \right) -\nabla \phi \left( x_k \right) \right> \nonumber
        \\
        &+\frac{\beta ^2L_{\phi}}{2}\left\| \widetilde{\nabla }\phi \left( x_k,\pi _k \right) -\nabla \phi \left( x_k \right) +\nabla \phi \left( x_k \right) \right\| ^2   \nonumber
        \\
        \le&\ \phi \left( x_k \right) -\beta \left\| \nabla \phi \left( x_k \right) \right\| ^2+\frac{\beta}{2}\left( \left\| \nabla \phi \left( x_k \right) \right\| ^2+\left\| \widetilde{\nabla }\phi \left( x_k,\pi _k \right) -\nabla \phi \left( x_k \right) \right\| ^2 \right) \nonumber
        \\
        &+\beta ^2L_{\phi}\left( \left\| \widetilde{\nabla }\phi \left( x_k,\pi _k \right) -\nabla \phi \left( x_k \right) \right\| ^2+\left\| \nabla \phi \left( x_k \right) \right\| ^2 \right) \label{eq:sobirl_explain_inequ}
        \\
        =&\ \phi \left( x_k \right) -\left( \frac{\beta}{2}-\beta ^2L_{\phi} \right) \left\| \nabla \phi \left( x_k \right) \right\| ^2+\left( \frac{\beta}{2}+\beta ^2L_{\phi} \right) \left\| \nabla \phi \left( x_k \right) -\widetilde{\nabla }\phi \left( x_k,\pi _k \right) \right\| ^2 ,    \label{eq:sobirl_explain_decrease}
    \end{align}
    where the updating rule $x_{k+1}=x_k-\beta \widetilde{\nabla }\phi \left( x_k,\pi _k \right) $ yields \eqref{eq:sobirl_explain_xupdate}, and \eqref{eq:sobirl_explain_inequ} comes from the inequalities $\abs{\innerp{a,b}}\le\frac{1}{2}\kh{\norm{a}^2+\norm{b}^2}$ and $\frac{1}{2}\norm{a+b}^2\le \norm{a}^2+\norm{b}^2$. Drawing from \cref{pro:lipz_hp_estimator} and the approximate lower-level solution $\pi_k$ satisfying $\norm{\pi^*(x_k)-\pi_k}^2 \le \epsilon$, we can incorporate
    \begin{equation*}
        \begin{aligned}
            \norm{{\nabla} \phi(x_k)-\widetilde{\nabla} \phi(x_k,\pi_k)}^2\le L^2_{\widetilde{\phi }}\norm{\pi^*(x_k)-\pi_k}^2 \le L^2_{\widetilde{\phi }}\ \epsilon,
        \end{aligned}
    \end{equation*}
    into \eqref{eq:sobirl_explain_decrease} and gain the estimation
    \begin{equation}
        \phi \left( x_{k+1} \right) \le \phi \left( x_k \right) -\left( \frac{\beta}{2}-\beta ^2L_{\phi} \right) \left\| \nabla \phi \left( x_k \right) \right\| ^2+\left( \frac{\beta}{2}+\beta ^2L_{\phi} \right) L^2_{\widetilde{\phi }}\ \epsilon .  \label{eq:telescope}
    \end{equation}
    Telescoping index from $k=1$ to $k=K$, we find that
    \begin{align*}
        \left( \frac{\beta}{2}-\beta ^2L_{\phi} \right) \sum_{k=1}^K{\left\| \nabla \phi \left( x_k \right) \right\| ^2}\le \phi \left( x_1 \right) -\phi ^*+K\left( \frac{\beta}{2}+\beta ^2L_{\phi} \right) L^2_{\widetilde{\phi }}\ \epsilon,
    \end{align*}
    where we define $\phi^*$ to be minimum of the hyper-objective. By setting the step size $\beta < \frac{1}{2L_\phi}$ and dividing both sides by $ K\left( \frac{\beta}{2}-\beta ^2L_{\phi} \right)$, we arrive at the conclusion.    
\end{proof}

\revise{The proof of \cref{the:convergent_SoBiRL} reveals that $\frac{1}{K}\sum||\nabla\phi(x_k)||^2_2=\mathcal{O}(\frac{1}{K\beta})$, and we can substitute $\beta=\mathcal{O}(1/L_{\phi})$ and $L_{{\phi}}=\mathcal{O}(\sqrt{|\mathcal{A}|})$ to obtain the outer iteration complexity as $\mathcal{O}(\sqrt{|\mathcal{A}|}\epsilon^{-1})$.}

\section{STOCHASTIC MODEL-FREE SOFT BILEVEL RL}\label{sec:stocsobirl}
In this section, the details of the stochastic model-free soft bilevel reinforcement learning algorithm, Stoc-SoBiRL are given. Moreover, we provide the statistical properties (bias and variance of the stochastic hyper-gradient) for the sampling scheme \cref{alg:samplingprocess} in \cref{sec:stoc_estimator}, and the convergence result of Stoc-SoBiRL \cref{alg:StocSoBiRL} in \cref{sec:conver_stocsobirl}.

\subsection{Details of Stoc-SoBiRL}\label{sec:detail_stocsobirl}
\begin{algorithm}[hbp]
    \caption{\texttt{HyperEstimator}($x,\pi,M,J,T$)}
    \label{alg:samplingprocess}
    \begin{algorithmic}[1]
        \REQUIRE reward parameter $x$, policy $\pi$, sampling configurations $M,J,T$
        \STATE Independently sample $M$ trajectory tuples $\bm{d}^m:=(d^m_1,d^m_2,\ldots,d^m_I)$ for $m=1,2,\ldots,M$, with each $H$-length trajectory $d^m_i=\{(s_h^{m,i},a_h^{m,i})\}_{h=0}^{H-1}$, and and denote $\bm{D}:=\{\bm{d}^1,\bm{d}^2,\ldots,\bm{d}^M\}$
        \STATE For $m=1,2,\ldots,M$, $i=1,2,\ldots,I$ and $h=0,1,\ldots,H-1$, sample $2J$ independent trajectories of length $T$, i.e., for $j=1,2,\ldots,J$,
        \begin{align*}
            \xi ^j\left( s_{h}^{m,i},a_{h}^{m,i} \right) :=&\left\{ \left( s_{q,0}^{m,i,h,j}=s_{h}^{m,i},a_{q,0}^{m,i,h,j}=a_{h}^{m,i} \right) ;\left( s_{q,1}^{m,i,h,j},a_{q,1}^{m,i,h,j} \right) ,\ldots,\left( s_{q,T-1}^{m,i,h,j},a_{q,T-1}^{m,i,h,j} \right) \right\}
            \\
            \zeta ^j\left( s_{h}^{m,i} \right) :=&\left\{ \left( s_{v,0}^{m,i,h,j}=s_{h}^{m,i};a_{v,0}^{m,i,h,j} \right) ,\left( s_{v,1}^{m,i,h,j},a_{v,1}^{m,i,h,j} \right) ,\ldots,\left( s_{v,T-1}^{m,i,h,j},a_{v,T-1}^{m,i,h,j} \right) \right\} 
        \end{align*}
        denote $\bm{\xi }:=\left\{ \xi ^j\left( s_{h}^{m,i},a_{h}^{m,i} \right) ,m=1,\ldots,M, j=1,\ldots,J,i=1,\ldots,I,h=0,\ldots,H-1 \right\}$, \\$\bm{\zeta }:=\left\{ \zeta ^j\left( s_{h}^{m,i} \right) ,m=1,\ldots,M, j=1,\ldots,J,i=1,\ldots,I,h=0,\ldots,H-1 \right\}$ and compute
        \begin{align*}
            \bar{\nabla}Q_{s_{h}^{m,i}a_{h}^{m,i}}^{\bm{\xi }}\left( x,\pi \right) :=&\frac{1}{J}\sum_{j=1}^{J}{\sum_{t=0}^{T-1}{\gamma ^t\nabla r_{s_{q,t}^{m,i,h,j}a_{q,t}^{m,i,h,j}}\left( x \right)}}
            \\
            \bar{\nabla}V_{s_{h}^{m,i}}^{\bm{\zeta }}\left( x,\pi \right) :=&\frac{1}{J}\sum_{j=1}^{J}{\sum_{t=0}^{T-1}{\gamma ^t\nabla r_{s_{v,t}^{m,i,h,j}a_{v,t}^{m,i,h,j}}\left( x \right)}}\,\,
        \end{align*}
        \STATE Compute the stochastic hyper-gradient
        \begin{equation*}
            \bar{\nabla}\phi(\bm{D},\bm{\xi},\bm{\zeta};x,\pi) =\frac{1}{M}\sum_{m=1}^M{\nabla l\left( \boldsymbol{d}^m;x \right)}+\tau ^{-1}\frac{1}{M}\sum_{m=1}^M{l\left( \boldsymbol{d}^m;x \right)}\sum_i\sum_h{\bar{\nabla}\left( Q_{s_{h}^{m,i}a_{h}^{m,i}}^{\boldsymbol{\xi }}-V_{s_{h}^{m,i}}^{\boldsymbol{\zeta }} \right)}\left( x,\pi \right) 
        \end{equation*}
        \ENSURE $\bar{\nabla}\phi(\bm{D},\bm{\xi},\bm{\zeta};x,\pi)$, $(\bm{D},\bm{\xi},\bm{\zeta})$
    \end{algorithmic}
\end{algorithm}
The core principle for estimating the hyper-gradient is to sample independent trajectories, evaluate the relevant quantities along each trajectory, and then average these values. Specifically, in the $k$-th outer iteration, to estimate $\nabla\phi$, an expectation with respect to the random variable $\bm{d}$, we can generate $M$ independent tuples $\bm{d}^m:=(d^m_1,d^m_2,\ldots,d^m_I)$ for $m=1,2,\ldots,M$ and denote $\bm{D}:=\{\bm{d}^1,\bm{d}^2,\ldots,\bm{d}^M\}$. The next focus is to tackle the terms $\nabla Q^*$ and $\nabla V^*$ in \eqref{eq:mf_hypergrad2}. Assuming access to a
generative model \citep{kearns1998finite,li2024generativemodel}, from any initial state-action pair $(s,a)$, we sample $J$ independent trajectories of length $T$ by implementing $\pi_k$, i.e., for $j=1,2,\ldots,J$,
\begin{equation}
\xi _{k}^{j}( s,a) :=\{( s_{0}^{j}=s,a_{0}^{j}=a);(s_{1}^{j},a_{1}^{j} ) ,\ldots,( s_{T-1}^{j},a_{T-1}^{j})\}.    
\end{equation}
Collecting all these random variables yields $\bm{\xi}_k:=\{\xi _{k}^{j}\left( s,a \right) , j=1,\ldots,J,s\in \mathcal{S} ,a\in \mathcal{A} \}$, and the estimator for $\nabla Q^*$ can be constructed as 
\begin{equation}\label{eq:sampleQ}
\bar{\nabla}Q_{sa}^{\bm{\xi }_k}\left( x_k,\pi _k \right) :=\frac{1}{J}\sum_{j=1}^{J}{\sum_{t=0}^{T-1}{\gamma ^t\nabla r_{s_{t}^{j}a_{t}^{j}}\left( x_k \right)}}.
\end{equation}
Similarly, the constructions of the random variable $\bm{\zeta}_k$ and the associated estimator $\bar{\nabla}V^{\bm{\zeta }_k}$ for $\nabla V^*$ are detailed in \cref{alg:samplingprocess}. Consequently, we obtain the stochastic hyper-gradient:
\begin{align}\label{eq:stoc_hyper}
    \bar{\nabla}\phi \left( \bm{D}_k,\bm{\xi }_k,\bm{\zeta }_k;x_k,\pi _k \right) =\frac{1}{M}\sum_{m=1}^M{\nabla l\left( \bm{d}_k^m;x_k \right)} +\frac{1}{\tau M}\sum_{m=1}^M{l\left( \bm{d}_k^m;x_k \right)}\sum_i\sum_h{\bar{\nabla}\left( Q_{s_{h}^{m,i}a_{h}^{m,i}}^{\bm{\xi }_k}-V_{s_{h}^{m,i}}^{\bm{\zeta }_k} \right)}\left( x_k,\pi _k \right), 
\end{align}
which is abbreviated as $\bar{\nabla}\phi _k$ in the following discussion. 

Momentum techniques are known to be beneficial for reducing variance and accelerating algorithms \citep{cutkosky2019momentum}. To this end, we maintain a momentum-instructed $h_k$ in the $k$-th outer iteration,
\begin{equation}\label{eq:compute_hk}
    h_k =  \bar{\nabla}\phi_k+ (1-\mu_k)(h_{k-1}-\bar{\nabla}\phi(\bm{D}_k,\bm{\xi}_k,\bm{\zeta}_k;x_{k-1},\pi_{k-1})),
\end{equation}
which tracks $\bar{\nabla}\phi_k$ via current and historical hyper-gradient estimates. Using $h_k$ to update the upper-level $x_k$, we obtain the stochastic \cref{alg:StocSoBiRL}, Stoc-SoBiRL.

\begin{algorithm}[tbp]
    \caption{Stochastic model-free soft bilevel reinforcement learning algorithm (Stoch-SoBiRL)}
    \label{alg:StocSoBiRL}
    \begin{algorithmic}[1]
        \REQUIRE iteration number $K$, step sizes $\{\beta_k,\mu_k\}_{k=1}^K$, initialization $x_1, \pi_0, h_0$, sampling configurations $M,J,T$, lower-level accuracy $\{\epsilon_k\}_{k=1}^K$
        \FOR{$k = 1, \dots, K$}
            \STATE Solve the lower-level problem with the initialization $\pi_{k-1}$ to get ${\pi}_k$ satisfying $\norm{{\pi}_k-\pi^*(x_k)}_2^2\le \epsilon_k$
            \STATE Obtain the stochastic hyper-gradient and samples: $\bar{\nabla}\phi_k,\,(\bm{D}_k,\bm{\xi}_k,\bm{\zeta}_k)=\texttt{HyperEstimator}(x_k,\pi_k,M,J,T)$\!\!\!\!
            \STATE Compute the momentum-instructed gradient estimator $h_k$ via \eqref{eq:compute_hk}
            \STATE Implement a stochastic hyper-gradient step $x_{k+1} = x_k - \beta_k h_k$
        \ENDFOR
        \STATE Sample $\bar{K}$ from the uniform distribution $\mathcal{U}\{1,2,\ldots,K\}$
        \ENSURE $(x_{\bar{K}},\pi_{\bar{K}})$
    \end{algorithmic}
\end{algorithm}

\subsection{Stochasticity of the Estimator $\bar{\nabla}\phi$}\label{sec:stoc_estimator}
This part analyzes the bias and variance introduced by the stochastic hyper-gradient $\bar{\nabla}\phi$, which serve as a cornerstone for the convergence analysis of the developed Stoc-SoBiRL. To begin with, we show the bias of the estimators $\bar{\nabla}Q^{\bm{\xi}}$ and $\bar{\nabla}V^{\bm{\zeta}}$ to $\widetilde{\nabla}Q$ and $\widetilde{\nabla}V$, respectively.
\begin{lemma}
    Under \cref{assu:r}, we have that given any $s\in\mdps$ and $a\in\mdpa$,
    \begin{equation}\label{eq:bias_Q}
    \begin{aligned}
        \left\| \mathbb{E} _{\boldsymbol{\xi }_k}\left[ \bar{\nabla}Q_{sa}^{\boldsymbol{\xi }_k}\left( x_k,\pi _k \right) \right] -\widetilde{\nabla }Q_{sa}\left( x_k,\pi _k \right) \right\| _{2}\le&\ \frac{\gamma ^{T}C_{rx}}{1-\gamma},
        \\
        \left\| \mathbb{E} _{\boldsymbol{\zeta }_k}\left[ \bar{\nabla}V_{s}^{\boldsymbol{\zeta }_k}\left( x_k,\pi _k \right) \right] -\widetilde{\nabla }V_s\left( x_k,\pi _k \right) \right\| _{2}\le&\ \frac{\gamma ^{T}C_{rx}}{1-\gamma}.
    \end{aligned}         
    \end{equation}
\end{lemma}
\begin{proof}
    Recall the expression of $\widetilde{\nabla}Q$ in \cref{sec:SoBiRL},
    \begin{align*}
          \widetilde{\nabla} Q_{sa}(x_k,\pi_k) = \mathbb{E}\left[ \sum_{t=0}^\infty \gamma^t\nabla r_{s_t,a_t}(x_k)\,\Big |\,s_0=s,a_0=a,\pi_k\right].
    \end{align*}  
    By definition of $\bar{\nabla}Q_{sa}^{\boldsymbol{\xi }_k}\left( x_k,\pi _k \right) $ in \eqref{eq:sampleQ}, which is indeed a $T$-depth truncation of $\widetilde{\nabla}Q_{sa}(x_k,\pi_k)$, we obtain
    \begin{align*}
        \mathbb{E} _{\boldsymbol{\xi }_k}\left[ \bar{\nabla}Q_{sa}^{\boldsymbol{\xi }_k}\left( x_k,\pi _k \right) \right] -\widetilde{\nabla }Q_{sa}\left( x_k,\pi _k \right) ={\mathbb{E} \left[ \sum_{t=T}^{\infty}{\gamma ^t\nabla r_{s_t,a_t}\left( x_k \right)}\,\Big |\, s_0=s,a_0=a,\pi _k \right]},
    \end{align*}
    which together with the boundedness of $\nabla r$ implies
    \begin{equation*}
        \left\| \mathbb{E} _{\boldsymbol{\xi }_k}\left[ \bar{\nabla}Q_{sa}^{\boldsymbol{\xi }_k}\left( x_k,\pi _k \right) \right] -\widetilde{\nabla }Q_{sa}\left( x_k,\pi _k \right) \right\| _{2}\le \frac{\gamma ^{T}C_{rx}}{1-\gamma}.
    \end{equation*}
    In a similar fashion, we can also derive the boundedness for the bias of $\bar{\nabla}V_{s}^{\boldsymbol{\zeta }_k}\left( x_k,\pi _k \right) $.
\end{proof}

Now, we can bound the variance of $\bar{\nabla}\phi \left( \boldsymbol{D}_k,\boldsymbol{\xi }_k,\boldsymbol{\zeta }_k;x_k,\pi _k \right)$ and the bias between it and the targeted $\widetilde{\nabla}\phi$. 
\begin{proposition}\label{pro:bias_var}
    Under Assumptions~\ref{assu:l} and~\ref{assu:r}, we have
    \begin{align}
        \left\| \mathbb{E} _{\boldsymbol{D}_k,\boldsymbol{\xi }_k,\boldsymbol{\zeta }_k}\left[ \bar{\nabla}\phi \left( \boldsymbol{D}_k,\boldsymbol{\xi }_k,\boldsymbol{\zeta }_k;x_k,\pi _k \right) \right] -\widetilde{\nabla }\phi \left( x_k,\pi _k \right) \right\| _{2} \le \frac{2\gamma ^{T}HIC_lC_{rx}}{\tau \left( 1-\gamma \right)}. \label{eq:bias_hyper}
    \end{align}
    Additionally, it holds that
    \begin{align}
        &\mathbb{E} _{\boldsymbol{D}_k,\boldsymbol{\xi }_k,\boldsymbol{\zeta }_k}\left[ \left\| \bar{\nabla}\phi \left( \boldsymbol{D}_k,\boldsymbol{\xi }_k,\boldsymbol{\zeta }_k;x_k,\pi _k \right) -\mathbb{E} _{\boldsymbol{D}_k,\boldsymbol{\xi }_k,\boldsymbol{\zeta }_k}\left[ \bar{\nabla}\phi \left( \boldsymbol{D}_k,\boldsymbol{\xi }_k,\boldsymbol{\zeta }_k;x_k,\pi _k \right) \right] \right\|^2 _{2} \right]\le\frac{4}{M}\left( L_l+\frac{2HIC_lC_{rx}\left( 1-\gamma ^{T_k} \right)}{\tau \left( 1-\gamma \right)} \right) ^2
    \end{align}
\end{proposition}
\begin{proof}
    Recaping the formulation of $\widetilde{\nabla}\phi$ in \cref{sec:SoBiRL},
    \begin{equation*}
        \widetilde{\nabla} \phi(x_k,\pi_k) =\ \mathbb{E} _{\bm{d}}\left[ \nabla l\left( \bm{d};x_k \right) \right]  
        +\frac{1}{\tau}\mathbb{E} _{\bm{d}}\left[ l\left( \bm{d};x_k \right)
     \Big( \sum_i{\sum_h{\widetilde{\nabla} \left( Q_{s_{h}^{i}a_{h}^{i}}-V_{s_{h}^{i}}\right)\kh{x_k,\pi_k}}} \Big) \right],
    \end{equation*}
    and taking into account the expression of $\bar{\nabla}\phi_k$ in \eqref{eq:stoc_hyper}, we can compute
    \begin{align*}
        &\mathbb{E} _{\boldsymbol{D}_k,\boldsymbol{\xi }_k,\boldsymbol{\zeta }_k}\left[ \bar{\nabla}\phi \left( \boldsymbol{D}_k,\boldsymbol{\xi }_k,\boldsymbol{\zeta }_k;x_k,\pi _k \right) \right] -\widetilde{\nabla }\phi \left( x_k,\pi _k \right) 
        \\
        =&\frac{1}{\tau M}\sum_{m=1}^M{\mathbb{E} _{\boldsymbol{d}^m}\left[ l\left( \boldsymbol{d}^m;x_k \right) \sum_i{\sum_h{\left\{ \mathbb{E} _{\boldsymbol{\xi }_k,\boldsymbol{\zeta }_k}\left[ \bar{\nabla}\left( Q_{s_{h}^{m,i}a_{h}^{m,i}}^{\boldsymbol{\xi }_k}-V_{s_{h}^{m,i}}^{\boldsymbol{\zeta }_k} \right) \left( x_k,\pi _k \right) \right] -\widetilde{\nabla }\left( Q_{s_{h}^{m,i}a_{h}^{m,i}}-V_{s_{h}^{m,i}} \right) \left( x_k,\pi _k \right) \right\}}} \right]}.
    \end{align*}
    Applying boundedness of $l$ and \eqref{eq:bias_Q} on the above equation yields \eqref{eq:bias_hyper}. The next step is to investigate the variance of the stochastic hyper-gradient. Note that $\bar{\nabla}\phi_k$ is the average of $M$ i.i.d. random variables 
    \begin{equation*}
        Z^m:=\nabla l\left( \boldsymbol{d}^m;x_k \right) +\frac{1}{\tau}l\left( \boldsymbol{d}^m;x_k \right) \sum_i{\sum_h{\bar{\nabla}\left( Q_{s_{h}^{m,i}a_{h}^{m,i}}^{\boldsymbol{\xi }_k}-V_{s_{h}^{m,i}}^{\boldsymbol{\zeta }_k} \right) \left( x_k,\pi _k \right)}}\,\,\text{for}\,\,m=1,2,\ldots,M.
    \end{equation*}
    Considering the boundedness of $\nabla l$, $\bar{\nabla}Q^{\bm{\xi}}$ and $\bar{\nabla}V^{\bm{\zeta}}$, we obtain
    \begin{equation*}
        \left\| Z^m-\mathbb{E} \left[ Z^m \right] \right\| _{2}^{2}\le 4\left( L_l+\frac{2HIC_lC_{rx}\left( 1-\gamma ^{T_k} \right)}{\tau \left( 1-\gamma \right)} \right) ^2.
    \end{equation*}
    This reveals
    \begin{equation*}
        \mathbb{E} \left[ \left\| \frac{1}{M}\sum_{m=1}^M{\left( Z^m-\mathbb{E} \left[ Z^m \right] \right)} \right\| _{2}^{2} \right] =\frac{1}{M^2}\sum_{m=1}^M{\mathbb{E} \left[ \left\| Z^m-\mathbb{E} \left[ Z^m \right] \right\| _{2}^{2} \right]}\le \frac{4}{M}\left( L_l+\frac{2HIC_lC_{rx}\left( 1-\gamma ^{T_k} \right)}{\tau \left( 1-\gamma \right)} \right) ^2.
    \end{equation*}
    Noticing that $\bar{\nabla}\phi \left( \boldsymbol{D}_k,\boldsymbol{\xi }_k,\boldsymbol{\zeta }_k;x_k,\pi _k \right)  = \frac{1}{M}\sum_{m=1}^M Z^m$ completes the proof.
\end{proof}
The above proposition means the bias and variance of the stochastic hyper-gradient can be sufficiently small by adjusting the sampling configurations $M,J,T$.
\subsection{Convergence Analysis of Stoc-SoBiRL}\label{sec:conver_stocsobirl}
In this part, we give a detailed convergence analysis for Stoc-SoBiRL. The Lipschitz properties of $\bar{\nabla}Q^{\bm{\xi}}$, $\bar{\nabla}V$ and $\bar{\nabla}\phi$ developed in the following lemmas serve as a foundation for convergence results of Stoc-SoBiRL.
\begin{lemma}
    Under \cref{assu:r}, for arbitrary upper-level variables $x^1,x^2\in\mathbb{R}^n$, policies $\pi^1,\pi^2$, and $s\in\mdps$, $a\in\mdpa$, we have
    \begin{equation*}
        \begin{aligned}
            \left\| \bar{\nabla} Q^{\bm{\xi}}_{sa}\kh{x^1,\pi^1} -\bar{\nabla}Q^{\bm{\xi}}_{sa}\left( x^2,\pi^2 \right) \right\| _2&\le \frac{\left( 1-\gamma ^T \right) L_r}{1-\gamma}\left\| x^1 -x^2 \right\| _2,
            \\
            \left\| \bar{\nabla} V^{\bm{\zeta}}_{s}\left(x^1,\pi^1 \right) -\bar{\nabla}V^{\bm{\zeta}}_s\left(x^2,\pi^2 \right) \right\| _2 &\le \frac{\left( 1-\gamma ^T \right) L_r}{1-\gamma}\left\| x^1 -x^2 \right\| _2.
        \end{aligned}
    \end{equation*}
\end{lemma}
\begin{proof}
    Given a fixed $\bm{\xi}$, by definition of $\bar{\nabla}Q^{\bm{\xi}}$, we can derive
    \begin{equation*}
        \bar{\nabla}Q_{sa}^{\boldsymbol{\xi }}\left( x^1,\pi ^1 \right) -\bar{\nabla}Q_{sa}^{\boldsymbol{\xi }}\left( x^2,\pi ^2 \right) =\frac{1}{J}\sum_{j=1}^J{\sum_{t=0}^{T-1}{\gamma ^t\left( \nabla r_{s_{t}^{j}a_{t}^{j}}\left( x^1 \right) -\nabla r_{s_{t}^{j}a_{t}^{j}}\left( x^2 \right) \right)}}.
    \end{equation*}
    The lipschitzness of $\nabla r$ leads to the lipschitzness of $\bar{\nabla}Q^{\bm{\xi}}$. The proof for $\bar{\nabla}V^{\bm{\zeta}}$ follows analogously.
\end{proof}

\begin{lemma}[Lipschitzness of the hyper-gradient estimator]\label{lem:lip_barnaphi}
Under Assumptions \ref{assu:l} and \ref{assu:r}, for arbitrary upper-level variables $x^1,x^2\in\mathbb{R}^n$, policies $\pi^1,\pi^2$, we have
\begin{equation}
    \left\| \bar{\nabla}\phi \left( \boldsymbol{D},\boldsymbol{\xi },\boldsymbol{\zeta };x^1,\pi ^1 \right) -\bar{\nabla}\phi \left( \boldsymbol{D},\boldsymbol{\xi },\boldsymbol{\zeta };x^2,\pi ^2 \right) \right\| _2\le L_S \left\| x^1-x^2 \right\| _2,
\end{equation}
with $L_S:=\left( L_{l1}+\frac{2HI}{\tau \left( 1-\gamma \right)}\left( C_lL_r+C_{rx}L_l \right) \right)$.
\end{lemma}
\begin{proof}
    By definition \eqref{eq:stoc_hyper}, it holds that
    \begin{align*}
        &\bar{\nabla}\phi \left( \boldsymbol{D},\boldsymbol{\xi },\boldsymbol{\zeta };x^1,\pi ^1 \right) -\bar{\nabla}\phi \left( \boldsymbol{D},\boldsymbol{\xi },\boldsymbol{\zeta };x^2,\pi ^2 \right) =\frac{1}{M}\sum_{m=1}^M{\left( \nabla l\left( \boldsymbol{d}^m;x^1 \right) -\nabla l\left( \boldsymbol{d}^m;x^2 \right) \right)}
        \\
        &+\frac{1}{\tau M}\sum_{m=1}^M{\left[ l\left( \boldsymbol{d}^m;x^1 \right) \sum_i{\sum_h{\bar{\nabla}\left( Q_{s_{h}^{m,i}a_{h}^{m,i}}^{\boldsymbol{\xi }}-V_{s_{h}^{m,i}}^{\boldsymbol{\zeta }} \right) \left( x^1,\pi ^1 \right)}}-l\left( \boldsymbol{d}^m;x^2 \right) \sum_i{\sum_h{\bar{\nabla}\left( Q_{s_{h}^{m,i}a_{h}^{m,i}}^{\boldsymbol{\xi }}-V_{s_{h}^{m,i}}^{\boldsymbol{\zeta }} \right) \left( x^2,\pi ^2 \right)}} \right]}.
    \end{align*}
    The boundedness and lipschitzness of $l,\nabla l, \bar{\nabla}Q^{\bm{\xi}}$ and $\bar{\nabla}V^{\bm{\zeta}}$ leads to the conclusion.
\end{proof}
Define $p_k:=\mathbb{E} _{\boldsymbol{D}_k,\boldsymbol{\xi }_k,\boldsymbol{\zeta }_k}\left[ \bar{\nabla}\phi \left( \boldsymbol{D}_k,\boldsymbol{\xi }_k,\boldsymbol{\zeta }_k;x_k,\pi _k \right) \right] -\widetilde{\nabla }\phi \left( x_k,\pi _k \right) $ and $u_k:=h_k-\mathbb{E} _{\boldsymbol{D}_k,\boldsymbol{\xi }_k,\boldsymbol{\zeta }_k}\left[ \bar{\nabla}\phi \left( \boldsymbol{D}_k,\boldsymbol{\xi }_k,\boldsymbol{\zeta }_k;x_k,\pi _k \right) \right]$, with $h_k$ obeying the update rule,
\begin{align*}
    h_k &= \bar{\nabla}\phi(\bm{D}_k,\bm{\xi}_k,\bm{\zeta}_k;x_k,\pi_k)  + (1-\mu_k)(h_{k-1}-\bar{\nabla}\phi(\bm{D}_k,\bm{\xi}_k,\bm{\zeta}_k;x_{k-1},\pi_{k-1}))
    \\
    &=\mu_k \bar{\nabla}\phi(\bm{D}_k,\bm{\xi}_k,\bm{\zeta}_k;x_k,\pi_k) + (1-\mu_k)\kh{h_{k-1}+\bar{\nabla}\phi(\bm{D}_k,\bm{\xi}_k,\bm{\zeta}_k;x_k,\pi_k) - \bar{\nabla}\phi(\bm{D}_k,\bm{\xi}_k,\bm{\zeta}_k;x_{k-1},\pi_{k-1})}.
\end{align*}
The subsequent analysis is based on the following results established in \cref{pro:bias_var},
\begin{align}
    \left\| \mathbb{E} _{\boldsymbol{D}_k,\boldsymbol{\xi }_k,\boldsymbol{\zeta }_k}\left[ \bar{\nabla}\phi \left( \boldsymbol{D}_k,\boldsymbol{\xi }_k,\boldsymbol{\zeta }_k;x_k,\pi _k \right) \right] -\widetilde{\nabla }\phi \left( x_k,\pi _k \right) \right\| _{2}&\le b_{\phi },   \label{eq:bias}
    \\
    \mathbb{E} _{\boldsymbol{D}_k,\boldsymbol{\xi }_k,\boldsymbol{\zeta }_k}\left[ \left\| \bar{\nabla}\phi \left( \boldsymbol{D}_k,\boldsymbol{\xi }_k,\boldsymbol{\zeta }_k;x_k,\pi _k \right) -\mathbb{E} _{\boldsymbol{D}_k,\boldsymbol{\xi }_k,\boldsymbol{\zeta }_k}\left[ \bar{\nabla}\phi \left( \boldsymbol{D}_k,\boldsymbol{\xi }_k,\boldsymbol{\zeta }_k;x_k,\pi _k \right) \right] \right\| _{2}^{2} \right] &\le \sigma _{\phi }^{2},    \label{eq:variance}
\end{align}
 where $b_\phi$ and $\sigma_\phi^2$ can be sufficiently small by adjusting the sampling configurations $M,J,T$. All the notation $\norm{\cdot}$ without subscripts denotes the 2-norm for simplicity. 

The next two lemmas characterize the descent properties of $\phi(x_k)$ and $\norm{u_k}$ recursively.
\begin{lemma}[Descent property of the hyper-objective]
Under Assumptions \ref{assu:l} and \ref{assu:r}, the iterates $\{x_k\}_{k=1}^K$ generated by \cref{alg:StocSoBiRL} satisfy
\begin{equation}\label{eq:stoc_hyperdescent}
    \mathbb{E} \left[ \phi \left( x_{k+1} \right) \right] \le \mathbb{E} \left[ \phi \left( x_k \right) -\frac{\beta _k}{2}\left\| \nabla \phi (x_k) \right\| ^2-\frac{\beta _k}{2}\left( 1-\beta _kL_{\phi} \right) \left\| h_k \right\| ^2+\beta _k\left\| u_k \right\| ^2+2\beta _kL_{\widetilde{\phi }}^{2}\epsilon_k ^2+2\beta _k\left\| p_k \right\| ^2 \right].
\end{equation}
where $L_\phi$ and $L_{\widetilde{\phi}}$ are specified in \cref{pro:hypergrad_Lipz} and \cref{pro:lipz_hp_estimator}, respectively, and the expectations are with respect to the stochasticity of the algorithm.
\end{lemma}
\begin{proof}
    Based on the $L_\phi$-Lipschitz smoothness of $\phi(x)$ revealed by \cref{pro:hypergrad_Lipz}, a stochastic gradient step results in the following estimate of the hyper-objective $\varphi$:
    \begin{align}
        \phi \left( x_{k+1} \right) \le&\ \phi \left( x_k \right) +\left< \nabla \phi \left( x_k \right) ,x_{k+1}-x_k \right> +\frac{L_{\phi}}{2}\left\| x_{k+1}-x_k \right\| ^2 \nonumber
        \\
        =&\ \phi \left( x_k \right) -\beta_k \left< \nabla \phi \left( x_k \right) , h_k \right> +\frac{\beta_k ^2L_{\phi}}{2}\left\| h_k \right\| ^2  \nonumber
        \\
        =&\ \phi \left( x_k \right) -\frac{\beta_k}{2}\norm{\nabla \phi(x_k)}^2 -\frac{\beta_k}{2}\kh{1-\beta_kL_\phi}\norm{h_k}^2 + \frac{\beta_k}{2}\norm{\nabla\phi(x_k)-h_k}^2,     \label{eq:descent_estimate1}
    \end{align}
    where the last equality comes from ${\innerp{a,b}}=\frac{1}{2}(\norm{a}^2+\norm{b}^2-\norm{a-b}^2)$. Taking into account the equation $h_k-\widetilde{\nabla }\phi \left( x_k,\pi _k \right) -p_k=u_k$ and \cref{pro:lipz_hp_estimator}, we obtain
    \begin{align*}
        \left\| h_k-\nabla \phi \left( x_k \right) \right\| ^2=&\left\| h_k-\widetilde{\nabla }\phi \left( x_k,\pi _k \right) -p_k+\widetilde{\nabla }\phi \left( x_k,\pi _k \right) -\nabla \phi \left( x_k \right) +p_k \right\| ^2
        \\
        \le&\ 2 \left\| h_k-\widetilde{\nabla }\phi \left( x_k,\pi _k \right) -p_k \right\| ^2+4\left\| \widetilde{\nabla }\phi \left( x_k,\pi _k \right) -\nabla \phi \left( x_k \right) \right\| ^2+4\left\| p_k \right\| ^2
        \\
        \le&\ 2 \left\| u_k \right\| ^2+4L_{\widetilde{\phi }}^{2}\left\| \pi _k-\pi _{k}^{*} \right\| ^2+4\left\| p_k \right\| ^2.
    \end{align*}
    Incorporating this inequality into \eqref{eq:descent_estimate1} and taking expectation with respect to the algorithm lead to the result~\eqref{eq:stoc_hyperdescent}.
\end{proof}

\begin{lemma}[Descent property of the estimation error]
Under Assumptions \ref{assu:l} and \ref{assu:r}, the iterates $\{x_k\}_{k=1}^K$ generated by \cref{alg:StocSoBiRL} satisfy
\begin{equation}\label{eq:stoc_errordescent}
\begin{aligned}
\mathbb{E} \left[ \left\| u_{k+1} \right\| ^2 \right] \le&\ 2\left( 1-\mu _{k+1} \right) ^2\mathbb{E} \left[ \left\| u_k \right\| ^2 \right] +4\mu _{k+1}^{2}\sigma _{\phi}^{2}
\\
&+16\left( 1-\mu _{k+1} \right) ^2\left( \left( L_S+L_{\phi} \right) ^2\beta _{k}^{2}\mathbb{E} \left[ \left\| h_k \right\| ^2 \right] +2L_{\widetilde{\phi }}^{2}(\epsilon _k+\epsilon _{k+1})+\mathbb{E} \left[ \left\| p_{k+1} \right\| ^2 \right] +\mathbb{E} \left[ \left\| p_k \right\| ^2 \right] \right),   
\end{aligned}
\end{equation}
where $L_\phi,L_{\widetilde{\phi}}$, and $L_S$ are specified in \cref{pro:hypergrad_Lipz}, \cref{pro:lipz_hp_estimator}, and \cref{lem:lip_barnaphi}; and the expectations are with respect to the stochasticity of the algorithm.
\end{lemma}
\begin{proof}
    By definitions of $h_k,u_k$ and $p_k$, we have
    \begin{align*}
        u_{k+1} =&\  h_{k+1}-\widetilde{\nabla }\phi \left( x_{k+1},\pi _{k+1} \right) -p_{k+1}
        \\
        =& \left( 1-\mu _{k+1} \right) \left( h_k-\mathbb{E} _{\boldsymbol{D}_k,\boldsymbol{\xi }_k,\boldsymbol{\zeta }_k}\left[ \bar{\nabla}\phi \left( \boldsymbol{D}_k,\boldsymbol{\xi }_k,\boldsymbol{\zeta }_k;x_k,\pi _k \right) \right] \right) +\left( 1-\mu _{k+1} \right) \mathbb{E} _{\boldsymbol{D}_k,\boldsymbol{\xi }_k,\boldsymbol{\zeta }_k}\left[ \bar{\nabla}\phi \left( \boldsymbol{D}_k,\boldsymbol{\xi }_k,\boldsymbol{\zeta }_k;x_k,\pi _k \right) \right] 
        \\
        &-\left( 1-\mu _{k+1} \right) \bar{\nabla}\phi \left( \boldsymbol{D}_{k+1},\boldsymbol{\xi }_{k+1},\boldsymbol{\zeta }_{k+1};x_k,\pi _k \right) +\bar{\nabla}\phi \left( \boldsymbol{D}_{k+1},\boldsymbol{\xi }_{k+1},\boldsymbol{\zeta }_{k+1};x_{k+1},\pi _{k+1} \right) -\widetilde{\nabla }\phi \left( x_{k+1},\pi _{k+1} \right) -p_{k+1}
        \\
        =&\left( 1-\mu _{k+1} \right) u_k+\mu _{k+1}\left( \bar{\nabla}\phi \left( \boldsymbol{D}_{k+1},\boldsymbol{\xi }_{k+1},\boldsymbol{\zeta }_{k+1};x_{k+1},\pi _{k+1} \right) -\widetilde{\nabla }\phi \left( x_{k+1},\pi _{k+1} \right) -p_{k+1} \right) 
        \\
        &+\left( 1-\mu _{k+1} \right) \Big[ \left( \bar{\nabla}\phi \left( \boldsymbol{D}_{k+1},\boldsymbol{\xi }_{k+1},\boldsymbol{\zeta }_{k+1};x_{k+1},\pi _{k+1} \right) -\widetilde{\nabla }\phi \left( x_{k+1},\pi _{k+1} \right) -p_{k+1} \right) 
        \\
        &-\left( \bar{\nabla}\phi \left( \boldsymbol{D}_{k+1},\boldsymbol{\xi }_{k+1},\boldsymbol{\zeta }_{k+1};x_k,\pi _k \right) -\widetilde{\nabla }\phi \left( x_k,\pi _k \right) -p_k \right) \Big].
    \end{align*}
    This equality yields the following estimation
    \begin{align}
        \mathbb{E} \left[ \left\| u_{k+1} \right\| ^2\right]\le&\ 2\left( 1-\mu _{k+1} \right) ^2\mathbb{E} \left[ \left\| u_k \right\| ^2 \right] +4\mu _{k+1}^{2}\mathbb{E} \left[ \left\| \bar{\nabla}\phi \left( \boldsymbol{D}_{k+1},\boldsymbol{\xi }_{k+1},\boldsymbol{\zeta }_{k+1};x_{k+1},\pi _{k+1} \right) -\widetilde{\nabla }\phi \left( x_{k+1},\pi _{k+1} \right) -p_{k+1} \right\| ^2 \right]   \nonumber
        \\
        &+4\left( 1-\mu _{k+1} \right) ^2\mathbb{E} \Big[ \Big\| \left( \bar{\nabla}\phi \left( \boldsymbol{D}_{k+1},\boldsymbol{\xi }_{k+1},\boldsymbol{\zeta }_{k+1};x_{k+1},\pi _{k+1} \right) -\widetilde{\nabla }\phi \left( x_{k+1},\pi _{k+1} \right) -p_{k+1} \right)  \nonumber
        \\
        &-\left( \bar{\nabla}\phi \left( \boldsymbol{D}_{k+1},\boldsymbol{\xi }_{k+1},\boldsymbol{\zeta }_{k+1};x_k,\pi _k \right) -\widetilde{\nabla }\phi \left( x_k,\pi _k \right) -p_k \right) \Big\| ^2 \Big].  \label{eq:u_esti1}
    \end{align}
    Subsequently, notice that
    \begin{align}
        &\Big\| \left( \bar{\nabla}\phi \left( \boldsymbol{D}_{k+1},\boldsymbol{\xi }_{k+1},\boldsymbol{\zeta }_{k+1};x_{k+1},\pi _{k+1} \right) -\widetilde{\nabla }\phi \left( x_{k+1},\pi _{k+1} \right) -p_{k+1} \right) \nonumber
        \\
        &\quad -\left( \bar{\nabla}\phi \left( \boldsymbol{D}_{k+1},\boldsymbol{\xi }_{k+1},\boldsymbol{\zeta }_{k+1};x_k,\pi _k \right) -\widetilde{\nabla }\phi \left( x_k,\pi _k \right) -p_k \right) \Big\|     \nonumber
        \\
        \le& \left\| \bar{\nabla}\phi \left( \boldsymbol{D}_{k+1},\boldsymbol{\xi }_{k+1},\boldsymbol{\zeta }_{k+1};x_{k+1},\pi _{k+1} \right) -\bar{\nabla}\phi \left( \boldsymbol{D}_{k+1},\boldsymbol{\xi }_{k+1},\boldsymbol{\zeta }_{k+1};x_k,\pi _k \right) \right\| +\left\| \nabla \phi \left( x_{k+1} \right) -\widetilde{\nabla }\phi \left( x_{k+1},\pi _{k+1} \right) \right\|   \nonumber
        \\
        &+\left\| \nabla \phi \left( x_k \right) -\widetilde{\nabla }\phi \left( x_k,\pi _k \right) \right\| +\left\| \nabla \phi \left( x_{k+1} \right) -\nabla \phi \left( x_k \right) \right\| +\left\| p_{k+1} \right\| +\left\| p_k \right\|    \nonumber
        \\
        \overset{(i)}{\le}&\ L_S\left\| x_{k+1}-x_k \right\| +L_{\widetilde{\phi }}\left\| \pi _{k+1}-\pi _{k+1}^{*} \right\| +L_{\widetilde{\phi }}\left\| \pi _k-\pi _{k}^{*} \right\| +L_{\phi}\left\| x_{k+1}-x_k \right\| +\left\| p_{k+1} \right\| +\left\| p_k \right\|   \nonumber
        \\
        \le& \left( L_S+L_{\phi} \right) \left\| x_{k+1}-x_k \right\| +L_{\widetilde{\phi }}(\sqrt{\epsilon _k}+\sqrt{\epsilon _{k+1}})+\left\| p_{k+1} \right\| +\left\| p_k \right\| , \label{eq:u_esti2}
    \end{align}
    where $(i)$ comes from the Lipschitz properties of $\nabla\phi$ and $\widetilde{\nabla}\phi$ revealed by \cref{pro:hypergrad_Lipz} and \cref{pro:lipz_hp_estimator}, and Lipschitz property of $\bar{\nabla}\phi$ revealed by \cref{lem:lip_barnaphi}. Incorporating \eqref{eq:u_esti2} into \eqref{eq:u_esti1} and recalling the variance $\sigma_\phi$ in \eqref{eq:variance}, we complete the proof.
\end{proof}

Now, assembling the descent properties of \eqref{eq:stoc_hyperdescent} and \eqref{eq:stoc_errordescent}, we derive the convergence results.
\begin{theorem}\label{the:convergent_StocSoBiRL}
    Under Assumptions~\ref{assu:l} and~\ref{assu:r} and given the maximum iteration number $K$, we can choose appropriate sampling configurations $M\sim\mathcal{O}(K^{4/3}),\,J\sim\mathcal{O}(1),\,T\sim\mathcal{O}(\log K)$, and set the parameters as follows,
    \begin{equation}\label{eq:beta_mu}
        \beta_k = \frac{1}{(n_\beta+k)^{1/3}},\ \ \mu_{k+1}=\max\hkh{\frac{15}{16},1-\beta _{k}^{2}},\ \ \epsilon_k=\beta_k^2\ \ \text{for}\ k=1,2,\ldots,K,
    \end{equation}
    with the integer $n_\beta\ge 3$ and
    \begin{equation}
        \frac{L_{\phi}}{\left( n_{\beta}+1 \right) ^{1/3}}+\frac{2\left( L_S+L_{\phi} \right)}{\left( n_{\beta}+1 \right) ^{2/3}}\le 1.     \label{eq:nbeta}
    \end{equation}
    Then the iterates $\{x_k\}$ generated by \cref{alg:StocSoBiRL} satisfy
    \begin{equation*}
        \frac{1}{K}\sum_{k=1}^K{\mathbb{E} \left[ \left\| \nabla \phi (x_k) \right\| ^2 \right]}=\mathcal{O} \left( \frac{\log (K)}{K^{2/3}} \right) =\widetilde{\mathcal{O} }\left( K^{-\frac{2}{3}} \right),
    \end{equation*}
    where the expectation is with respect to the stochasticity of the algorithm.
\end{theorem}
\begin{proof}
We consider the merit function $\mathcal{S} _k:=\phi \left( x_k \right) +\frac{\left\| u_k \right\| ^2}{\beta _{k-1}}$. Substituting the descent property \eqref{eq:stoc_errordescent} of $\Exp{\norm{u_{k+1}}^2}$ leads to the following inequality.
\begin{align}
    \mathbb{E} \left[ \frac{\left\| u_{k+1} \right\| ^2}{\beta _k}-\frac{\left\| u_k \right\| ^2}{\beta _{k-1}} \right] \le& \left( \frac{2\left( 1-\mu _{k+1} \right) ^2}{\beta _k}-\frac{1}{\beta _{k-1}} \right) \mathbb{E} \left[ \left\| u_k \right\| ^2 \right] +\frac{4\mu _{k+1}^{2}}{\beta _k}\sigma _{\phi}^{2}    \nonumber
    \\
    &+\frac{16\left( 1-\mu _{k+1} \right) ^2}{\beta _k}\left( \left( L_S+L_{\phi} \right) ^2\beta _{k}^{2}\mathbb{E} \left[ \left\| h_k \right\| ^2 \right] +2L_{\widetilde{\phi }}^{2}(\epsilon_k+\epsilon_{k+1}) +\mathbb{E} \left[ \left\| p_{k+1} \right\| ^2 \right] +\mathbb{E} \left[ \left\| p_k \right\| ^2 \right] \right).\label{eq:merit_u1}
\end{align}
By the choice \eqref{eq:beta_mu} of $\beta_k$ and $\mu_k$, we have for $k=1,2,\ldots,K$,
\begin{align*}
    \frac{2\left( 1-\mu _{k+1} \right) ^2}{\beta _k}&\le \frac{16\left( 1-\mu _{k+1} \right) ^2}{\beta _k}\le \frac{1-\mu _{k+1}}{\beta _k}\le \beta _k,
    \\
    \frac{2\left( 1-\mu _{k+1} \right) ^2}{\beta _k}-\frac{1}{\beta _{k-1}}&\le \beta _k-\frac{1}{\beta _{k-1}}=\left( k+n_{\beta} \right) ^{-1/3}-\left( k-1+n_{\beta} \right) ^{1/3}\le -\left( k+n_{\beta} \right) ^{-1/3}=-\beta _k.
\end{align*}
Taking the two estimations into \eqref{eq:merit_u1} shows
\begin{align}
    \mathbb{E} \left[ \frac{\left\| u_{k+1} \right\| ^2}{\beta _k}-\frac{\left\| u_k \right\| ^2}{\beta _{k-1}} \right] \le& -\beta _k\mathbb{E} \left[ \left\| u_k \right\| ^2 \right] +\frac{4\mu _{k+1}^{2}}{\beta _k}\sigma _{\phi}^{2}+\left( L_S+L_{\phi} \right) ^2\beta _{k}^{3}\mathbb{E} \left[ \left\| h_k \right\| ^2 \right]   \nonumber
    \\
    &+2\beta _kL_{\widetilde{\phi }}^{2}(\epsilon_k+\epsilon_{k+1})+\beta _k\left( \mathbb{E} \left[ \left\| p_{k+1} \right\| ^2 \right] +\mathbb{E} \left[ \left\| p_k \right\| ^2 \right] \right).    \label{eq:merit_u2}
\end{align}
Next, we combine \eqref{eq:merit_u2} with \eqref{eq:stoc_hyperdescent} to get the descent property of $\mathcal{S}_k$.
\begin{align}
    \mathbb{E} \left[ \mathcal{S} _{k+1}-\mathcal{S} _k \right] \le& -\frac{\beta _k}{2}\mathbb{E} \left[ \left\| \nabla \phi (x_k) \right\| ^2 \right] -\frac{\beta _k}{2}\left( 1-\beta _kL_{\phi}-2\beta _{k}^{2}\left( L_S+L_{\phi} \right) ^2 \right) \mathbb{E} \left[ \left\| h_k \right\| ^2 \right] \nonumber
    \\
    &+\frac{4\mu _{k+1}^{2}}{\beta _k}\sigma _{\phi}^{2}+2\beta _kL_{\widetilde{\phi }}^{2}(2\epsilon_k+\epsilon_{k+1})+\beta _k\left( \mathbb{E} \left[ \left\| p_{k+1} \right\| ^2 \right] +3\mathbb{E} \left[ \left\| p_k \right\| ^2 \right] \right) \nonumber 
    \\
    \le& -\frac{\beta _k}{2}\mathbb{E} \left[ \left\| \nabla \phi (x_k) \right\| ^2 \right] +\frac{4\mu _{k+1}^{2}}{\beta _k}\sigma _{\phi}^{2}+2\beta _kL_{\widetilde{\phi }}^{2}(2\epsilon_k+\epsilon_{k+1})+\beta _k\left( \mathbb{E} \left[ \left\| p_{k+1} \right\| ^2 \right] +3\mathbb{E} \left[ \left\| p_k \right\| ^2 \right] \right),     \nonumber 
\end{align}
where the last inequality follows from the choice \eqref{eq:nbeta} of $n_\beta$. Telescoping index from $k=1$ to $k=K$ and dividing it by $K$ yield
\begin{align}\label{eq:descent_S1}
    \frac{1}{K}\mathbb{E} \left[ \mathcal{S} _{K+1}-\mathcal{S} _1 \right] \le -\frac{1}{K}\sum_{k=1}^K{\frac{\beta _k}{2}\mathbb{E} \left[ \left\| \nabla \phi (x_k) \right\| ^2 \right]}+\frac{4\sigma _{\phi}^{2}}{K}\sum_{k=1}^K{\frac{1}{\beta _k}}+\frac{6L_{\widetilde{\phi }}^{2}}{K}\sum_{k=1}^{K+1}{\beta _k\epsilon _k}+\frac{4}{K}\sum_{k=1}^{K+1}{\beta _k\left\| p_k \right\| ^2}.
\end{align}
Rearrange \eqref{eq:descent_S1} and employ $\beta_K\le\beta_k$,
\begin{equation*}
    \frac{\beta _K}{2K}\sum_{k=1}^K{\mathbb{E} \left[ \left\| \nabla \phi (x_k) \right\| ^2 \right]}\le \frac{1}{K}\mathbb{E} \left[ \mathcal{S} _1-\mathcal{S} _{K+1} \right] +\frac{4\sigma _{\phi}^{2}}{K}\sum_{k=1}^K{\frac{1}{\beta _k}}+\frac{6L_{\widetilde{\phi }}^{2}}{K}\sum_{k=1}^{K+1}{\beta _k\epsilon _k}+\frac{4}{K}\sum_{k=1}^{K+1}{\beta _k\left\| p_k \right\| ^2}.
\end{equation*}
Dividing $\beta_K/2$ on both sides implies that
\begin{align}
    \frac{1}{K}\sum_{k=1}^K{\mathbb{E} \left[ \left\| \nabla \phi (x_k) \right\| ^2 \right]}\le&\ \frac{2}{\beta _KK}\mathbb{E} \left[ \mathcal{S} _1-\mathcal{S} _{K+1} \right] +\frac{8\sigma _{\phi}^{2}}{\beta _KK}\sum_{k=1}^K{\frac{1}{\beta _k}}+\frac{12L_{\widetilde{\phi }}^{2}}{\beta _KK}\sum_{k=1}^{K+1}{\beta _k\epsilon _k}+\frac{8}{\beta _KK}\sum_{k=1}^{K+1}{\beta _k\left\| p_k \right\| ^2} \nonumber
    \\
    \le&\ \frac{2\left( \mathcal{S} _1-\phi ^* \right)}{\beta _KK}+\frac{8\sigma _{\phi}^{2}}{\beta _KK}\sum_{k=1}^K{\frac{1}{\beta _k}}+\frac{12L_{\widetilde{\phi }}^{2}}{\beta _KK}\sum_{k=1}^{K+1}{\beta _k\epsilon _k}+\frac{8}{\beta _KK}\sum_{k=1}^{K+1}{\beta _k\left\| p_k \right\| ^2},\label{eq:descent_S2}
\end{align}
where $\phi^*$ is the minimum of $\phi(x)$. By \cref{pro:bias_var}, we can choose the appropriate sampling parameters $M\sim\mathcal{O}(K^{4/3}),\,T\sim\mathcal{O}(\log K)$ such that $\norm{p_k}^2\le b^2_\phi =  K^{-2/3}$ and $\sigma_\phi^2=K^{-4/3}$ in \eqref{eq:bias} and $\eqref{eq:variance}$. With the parameters specified in \eqref{eq:beta_mu}, we have $\sigma _{\phi}^{2}\sum_{k=1}^{K+1}{\frac{1}{\beta _k}=}\mathcal{O} \left( 1 \right)$, $\sum_{k=1}^{K+1}{\beta _k\epsilon _k=}\sum_{k=1}^{K+1}{\beta _{k}^{3}=\mathcal{O} \left( \log  K \right)}$, $\sum_{k=1}^{K+1}{\beta _k\left\| p_k \right\| ^2=}\mathcal{O} \left( 1 \right)$, and therefore the conclusion,
\begin{equation*}
    \frac{1}{K}\sum_{k=1}^K{\mathbb{E} \left[ \left\| \nabla \phi (x_k) \right\| ^2 \right]}=\mathcal{O} \left( \frac{\mathcal{S} _1-\phi ^*}{K^{2/3}} \right) +\mathcal{O} \left( \frac{1}{K^{2/3}} \right) +\mathcal{O} \left( \frac{\log K}{K^{2/3}} \right) 
    +\mathcal{O} \left( \frac{1}{K^{2/3}} \right)
    =\widetilde{\mathcal{O} }\left( K^{-\frac{2}{3}} \right).
\end{equation*}
\end{proof}

\revise{The inequality~\eqref{eq:descent_S2} in the proof of \cref{the:convergent_StocSoBiRL} shows $\frac{1}{K}\sum\mathbb{E}[||\nabla\phi(x_k)||^2_2]= \widetilde{\mathcal{O}}(L_{\widetilde{\phi}}^2{K^{-2/3}})$, which means the outer iteration complexity is $\widetilde{\mathcal{O}}(|\mathcal{A}|^{1.5}\epsilon^{-1.5})$ by substituting $L_{\widetilde{\phi}}=\mathcal{O}(\sqrta)$.}

\section{DETAILS ON EXPERIMENTS}\label{sec:details_exp}
This section presents the details of the experiment implementation. We conduct the RLHF task on three Atari games from the Arcade Learning Environment (ALE) \citep{bellemare2013arcade} and three Mujoco environments \citep{todorov2012mujoco} to empirically validate the performance of the model-free algorithm, SoBiRL. The reward provided by the original environments serves as the ground truth, and for each trajectory pair, preference is assigned to the trajectory with higher accumulated ground-truth reward. We compare SoBiRL \revise{with the bilevel algorithms DRLHF \citep{christiano2017rlhf}, PBRL \citep{shen2024bilevelRL}, HPGD \cite{thoma2024HPGD}, and a baseline algorithm SAC \citep{haarnoja2018sac,haarnoja2018sacauto}. All the bilevel solvers harness deep neural networks to predict rewards, while the baseline SAC receives ground-truth rewards for training.} Additionally, we also test on a synthetic bilevel RL experiment to verify the convergence rate of the model-based algorithm, M-SoBiRL. The experiments are produced on a workstation that consists of two Intel® Xeon® Gold 6330 CPUs (total 2$\times$28 cores), 512GB RAM, and one NVIDIA A800 (80GB memory) GPU. We have made the code available on \revise{\href{https://github.com/UCAS-YanYang/SoBiRL}{https://github.com/UCAS-YanYang/SoBiRL}}.

\subsection{Practical SoBiRL}
SoBiRL adopts deep neural networks to parameterize the policy and the reward model. In this way, we supplement \cref{alg:SoBiRL} with details, resulting in the practical version, \cref{alg:PSoBiRL}. The SAC algorithm capable of automating the temperature parameter \citep{haarnoja2018sacauto} is chosen as the lower-level solver for SoBiRL, i.e., for a fixed reward network, \texttt{SAC} is invoked with several timesteps to update the policy network, the Q-value network and the adaptive temperature parameter (line~2 of \cref{alg:PSoBiRL}). Here, to focus more on the implementation of SoBiRL, we stepsize the details of $\texttt{SAC}$, which fully follows the principles in \citep{haarnoja2018sacauto}.

\begin{algorithm}[htbp]
    \caption{Practical SoBiRL}
    \label{alg:PSoBiRL}
    \begin{algorithmic}[1]
        \REQUIRE iteration threshold $K$, inner iteration timesteps $N$, step size $\beta$, buffer $\mathcal{D}$ storing trajectory pairs, initial reward network parameterized by $\theta^r_1$, initial policy network parameterized by $\theta^\pi_0$, initial Q-value network function parameterized by $\theta^Q_0$, initial temperature parameter $\tau_0$
        \FOR{$k = 1, \dots, K$}
            \STATE $\theta^\pi_k, \theta^Q_k$,$\tau_k$ = \texttt{SAC}($\theta^\pi_{k-1}, \theta^Q_{k-1}$,$\tau_{k-1}$,$N$)  \COMMENT{lower-level policy update}
            \STATE Rollout the policy $\pi_k=\pi _{\theta^\pi_k}$ in the uppper-level MDP \COMMENT{collect  trajectories}
            \\ to obtain trajectory pairs    
            \STATE Query preference and update the buffer $\mathcal{D}$ with the labeled pairs 
            \STATE Compute the estimator $\nabla_p\phi \left( \theta^r_k,\pi_k,\tau _k \right)$ using $\mathcal{D}$  \COMMENT{upper-level reward update}
            \STATE Implement an inexact hyper-gradient descent step 
            \\ $\theta^r_{k+1} = \theta^r_k - \beta\nabla_p\phi \left( \theta^r_k,\pi_k,\tau _k \right)$
        \ENDFOR
        \ENSURE $(\theta^r_{K+1},\theta^\pi_{K+1})$
    \end{algorithmic}
\end{algorithm}

Adapting from $\widetilde{\nabla}\phi\kh{x_k,\pi_k}$ in \cref{sec:SoBiRL}, a practical version of hyper-gradient estimator $\nabla_p\phi \left( x_k,\pi _k,\tau _k \right)$ is used,
\begin{align}
    \nabla_p\phi \left( x_k,\pi _k,\tau _k \right) =&\ \mathbb{E} _{d_i\sim \rho \left( d;\pi_k \right)}\left[ \nabla l\left( d_1,d_2,\ldots,d_I;x_k \right) \right]  \label{eq:prac_free_hyper_esti}
    \\
    &+\tau _{k}^{-1}\mathbb{E} _{d_i\sim \rho \left( d;\pi _k \right)}\left[ l\left( d_1,d_2,\ldots,d_I;x_k \right) \left( \sum_i{\sum_h{\nabla \left( r_{s_{h}^{i}a_{h}^{i}}(x_k)-\mathbb{E} _{a^{\prime}\sim \pi _k\left( \cdot |s_{h}^{i} \right)}\left[ r_{s_{h}^{i}a^{\prime}}(x_k) \right] \right)}} \right) \right] .     \nonumber
\end{align}
The main difference between \eqref{eq:prac_free_hyper_esti} and \eqref{eq:free_hyper_esti} lies in two aspects: 1) $\nabla_p\phi \left( x_k,\pi _k,\tau _k \right) $ accommodates the adaptive temperature parameter $\tau_k$; 2) the term $\nabla \left(r_{s_{h}^{i}a_{h}^{i}}(x_k)-\mathbb{E} _{a^{\prime}\sim \pi _k\left( \cdot |s_{h}^{i} \right)}\left[ r_{s_{h}^{i}a^{\prime}}(x_k) \right]\right)$ is employed as an one-depth truncation of $\widetilde{\nabla} \left( Q_{s_{h}^{i}a_{h}^{i}}\kh{x_k,\pi_k}-V_{s_{h}^{i}}\kh{x_k,\pi_k}\right)$ in \eqref{eq:free_hyper_esti}, inspired by the expressions,
\begin{align*}
     \widetilde{\nabla} V_s(x_k,\pi_k) &= \mathbb{E}\left[ \sum_{t=0}^\infty \gamma^t\nabla r_{s_t,a_t}(x_k)\ \Big |\ s_0=s, \pi_k,\mdp_\tau\kh{x_k} \right],
     \\
      \widetilde{\nabla} Q_{sa}(x_k,\pi_k) &= \mathbb{E}\left[ \sum_{t=0}^\infty \gamma^t\nabla r_{s_t,a_t}(x_k)\ \Big |\ s_0=s,a_0=a,\pi_k,\mdp_\tau\kh{x_k} \right].
\end{align*}

\subsection{Experiment Settings for RLHF}
The results on the Atari games---BeamRider, Seaquest, and SpaceInvaders---are shown in Figures~\ref{fig:beamrider} and \ref{fig:rlhf2}. The results on the Mujoco environments---HalfCheetah, Walker2d, and Hopper---are presented in \cref{fig:Mujoco}. The overall performance is summarized in \cref{tab:overall_exp}.

\begin{figure*}[htbp]
\centering
\begin{minipage}{0.49\textwidth}
    \centering
    \includegraphics[width=1\linewidth]{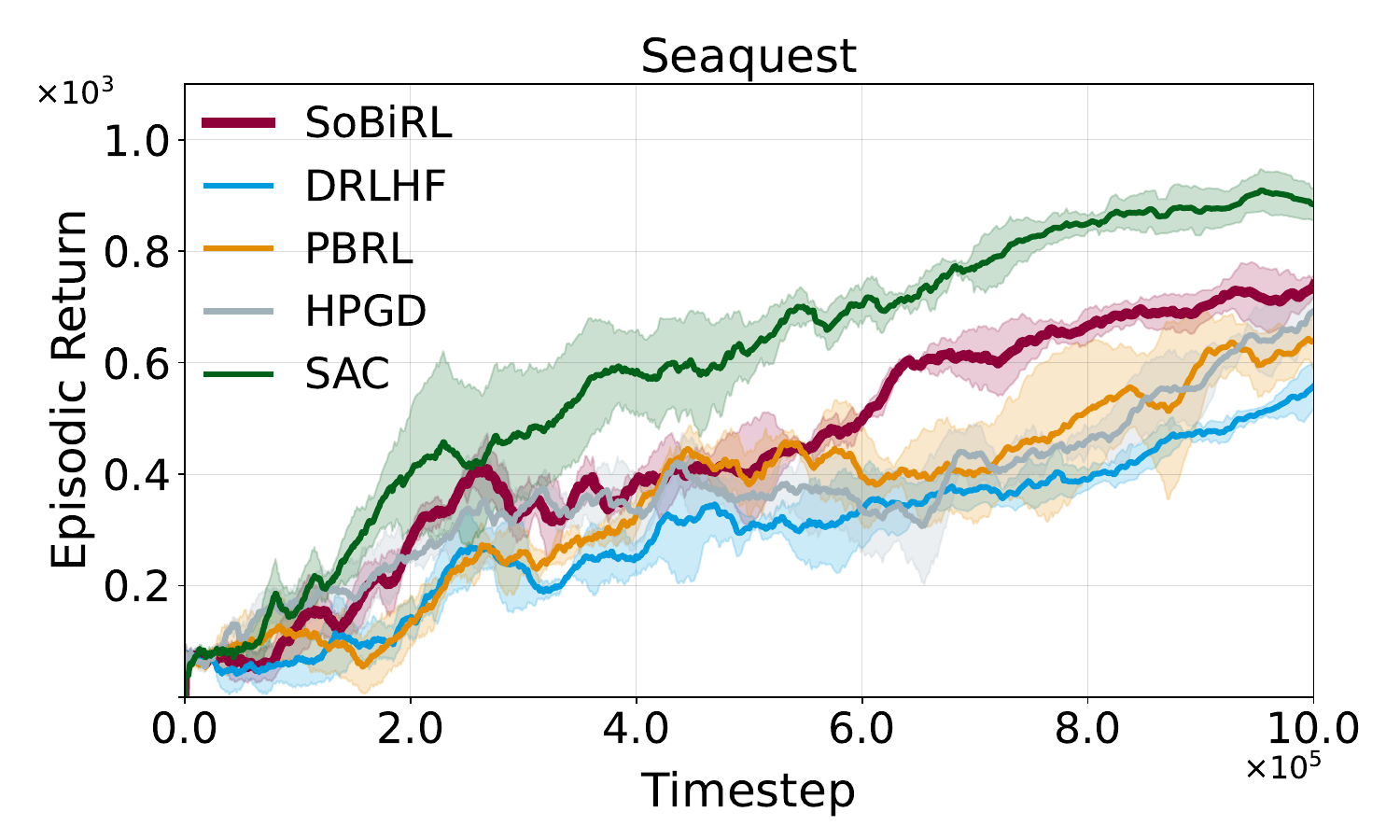}
\end{minipage}
\,
\begin{minipage}{0.49\textwidth}
    \centering
    \includegraphics[width=1\linewidth]{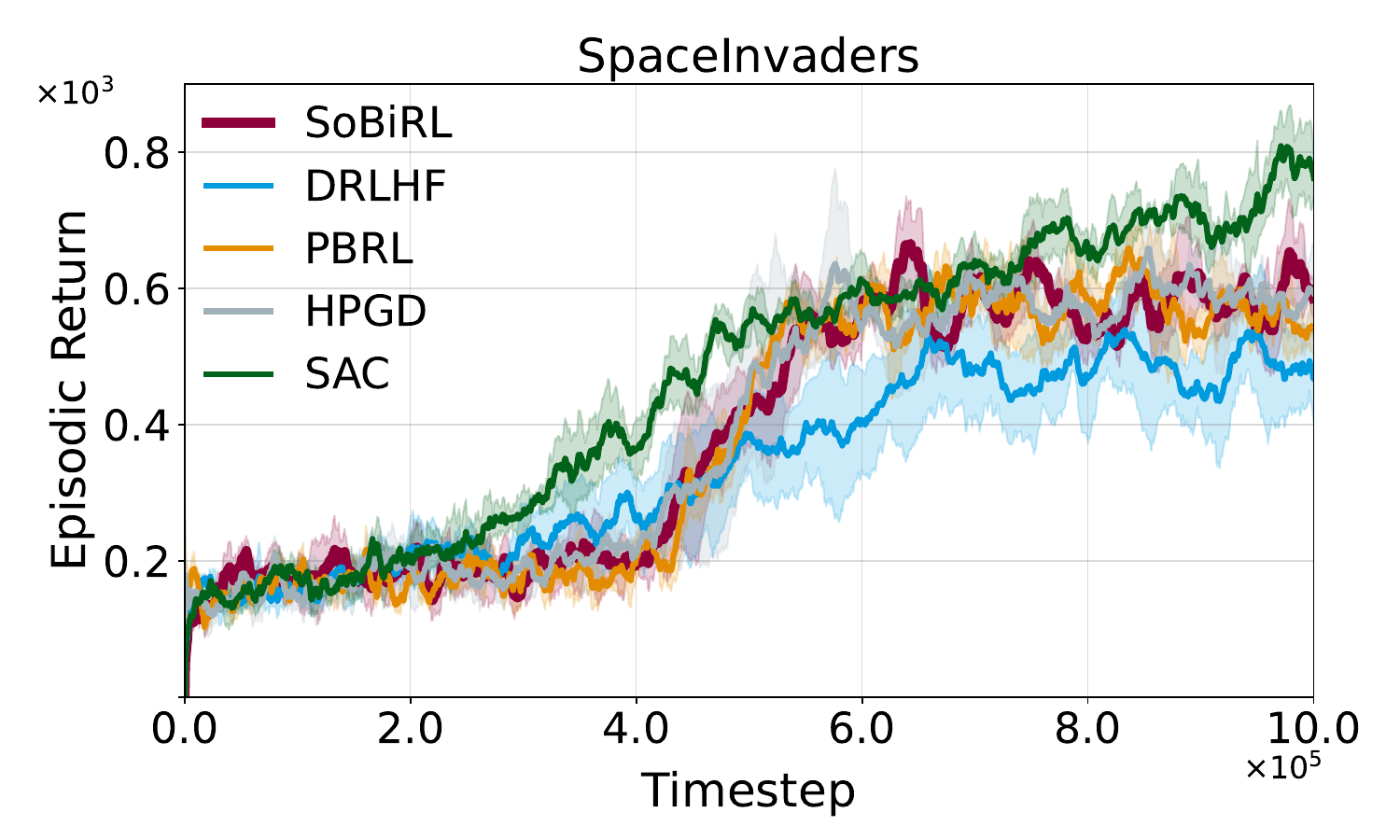}
\end{minipage}
\caption{Comparison on the Atari games---Seaquest and SpaceInvaders---evaluated by the ground-truth reward. Each bilevel algorithm collects a total of $3000$ trajectory pairs. The results are averaged over $5$ seeds.}
\label{fig:rlhf2}
\end{figure*}

\begin{figure*}[tbp]
\centering
\begin{minipage}{0.49\textwidth}
    \centering
    \includegraphics[width=1\linewidth]{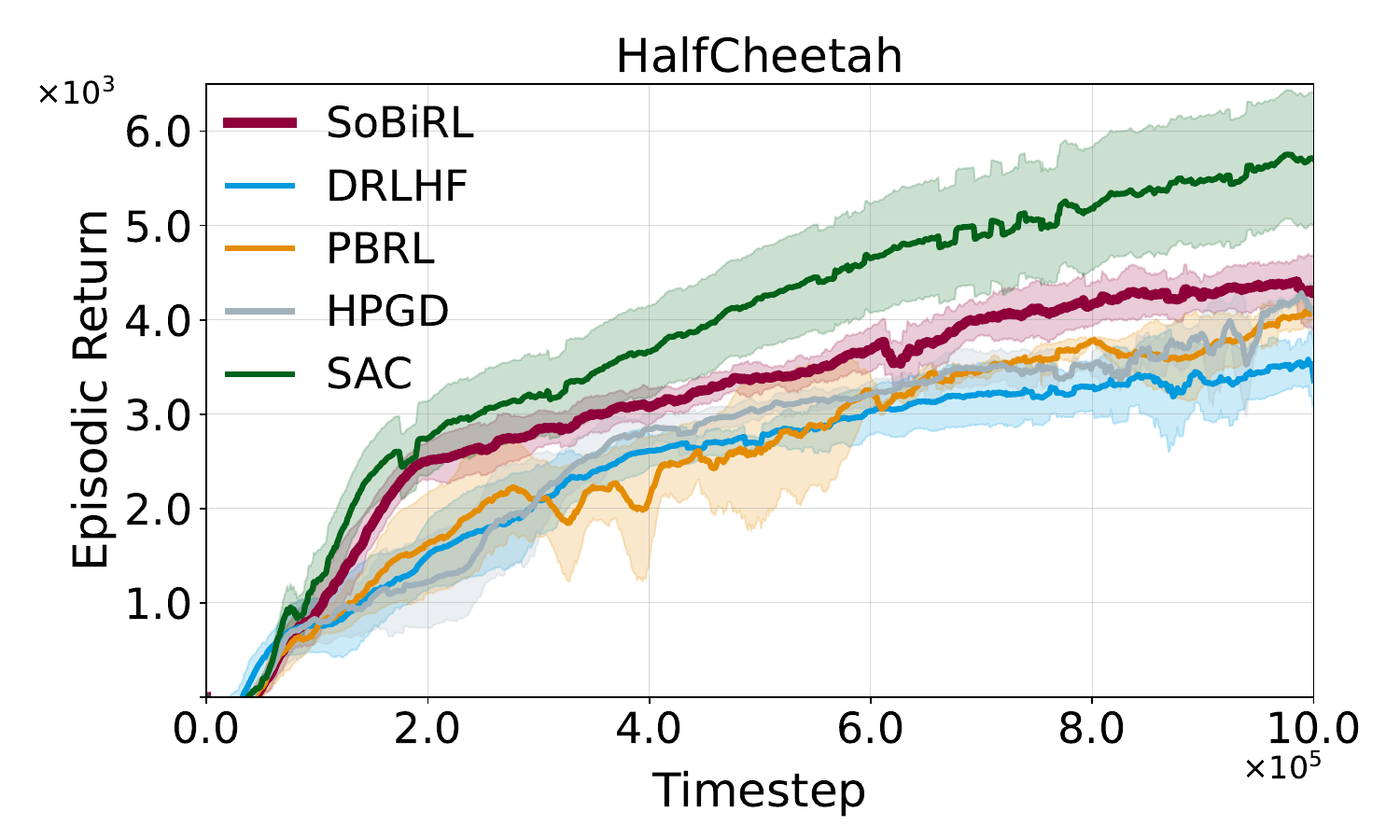}
\end{minipage}
\\
\begin{minipage}{0.49\textwidth}
    \centering
    \includegraphics[width=1\linewidth]{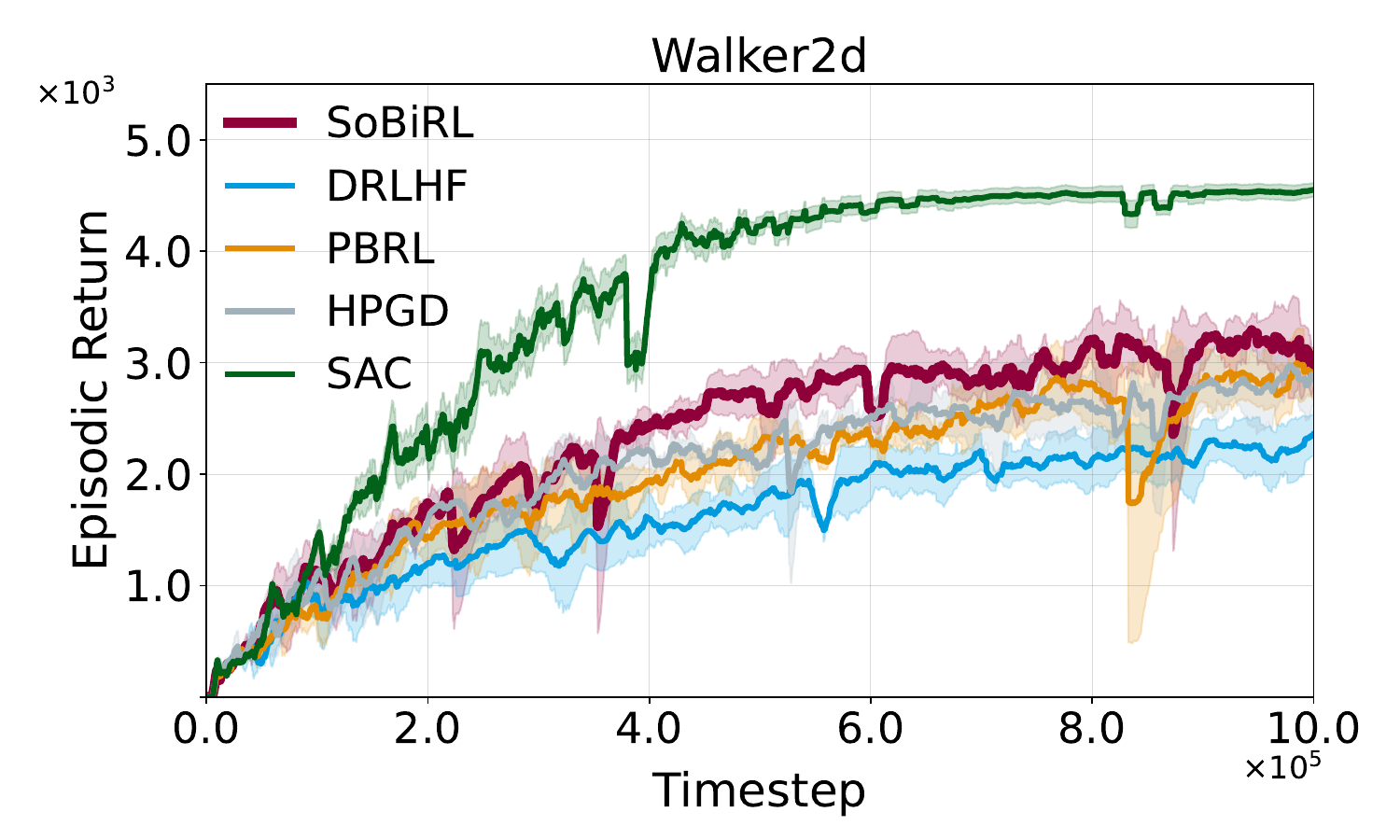}
\end{minipage}
\,
\begin{minipage}{0.49\textwidth}
    \centering
    \includegraphics[width=1\linewidth]{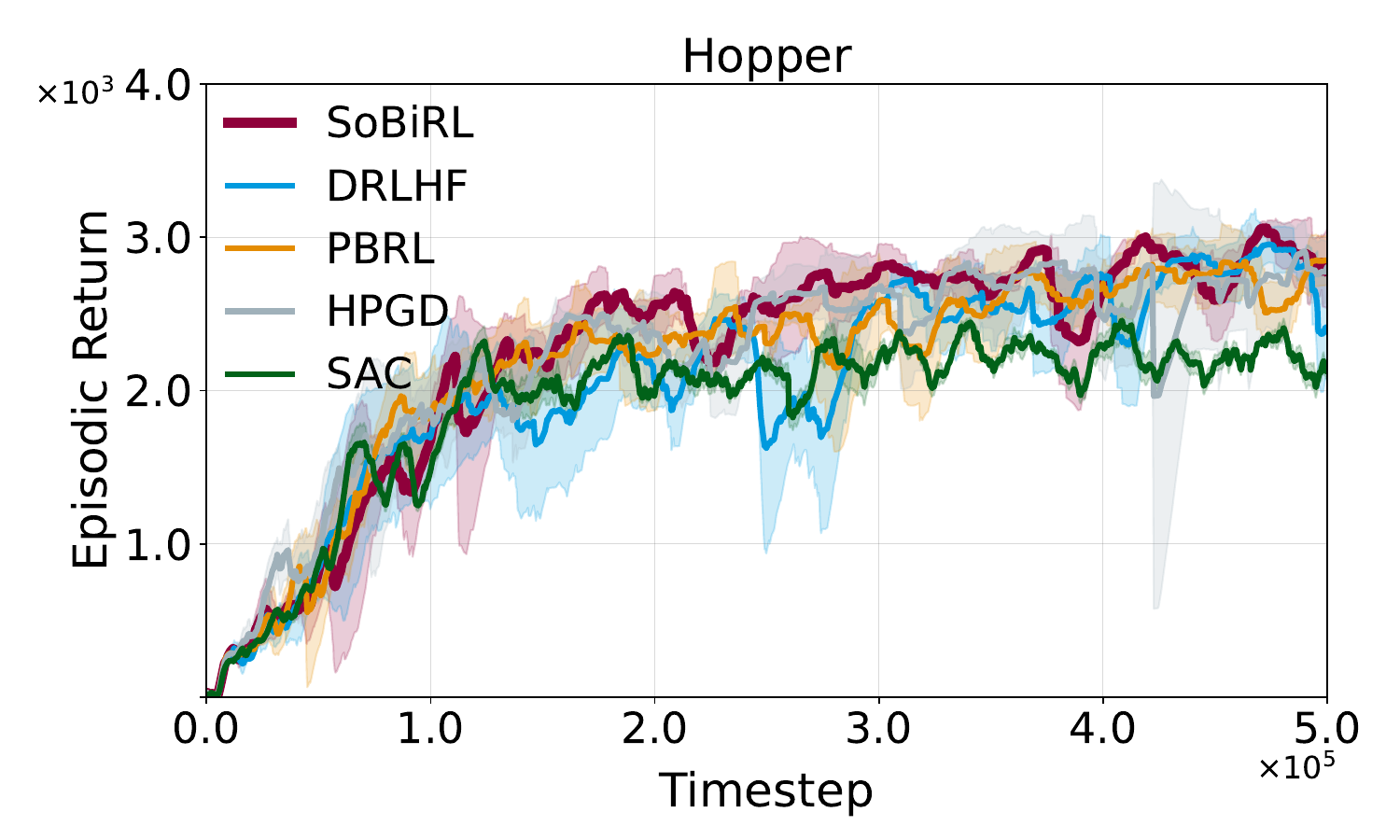}
\end{minipage}
\caption{Comparison of algorithms on the Mujoco simulations---HalfCheetah, Walker2d, and Hopper---evaluated by the ground-truth reward. Each bilevel algorithm collects a total of $3000$ trajectory pairs. The results are averaged over $5$ seeds.}
\label{fig:Mujoco}
\end{figure*}

As mentioned above, we take \texttt{SAC} as the lower-level solver for \revise{the bilevel algorithms}. We set the initial temperature parameter $\tau_0=1$, the inner iteration timesteps $N=10^4$ for Atari games and $N=5\times 10^3$ for Mujoco simulations, and the batch size to $64$. In addition, on Atari games, the actor network, the critic network, and the temperature parameter are updated by the Adam optimizer, with an initial learning rate $1\times 10^{-4}$ for BeamRider and SpaceInvaders, and $3\times 10^{-4}$ for Seaquest, linearly decaying to $0$ after $8\times10^{6}$ time steps (although the runs were actually trained for only $4\times10^{6}$ timesteps). On all Mujoco simulations, the critic network and the temperature parameter share the same initial learning rate $1\times 10^{-3}$, while that of the actor network is $3\times 10^{-4}$. They are updated by Adam optimizers, linearly decaying the learning rates to $0$ after $8\times10^{6}$ time steps (although the runs were actually trained for only $4\times10^{6}$ timesteps).

\revise{Bilevel algorithms} employ deep neural networks to parameterize the reward model. The wrapped environment returns an $84\times84\times4$ tensor as the state information. Therefore, data in the shape of $84\times84\times4$ is fed into the reward network. It undergoes four convolutional layers with kernel sizes of $7\times7,\ 5\times5,\ 3\times3,\ 3\times3$ and stride values of $3,\ 2,\ 1,\ 1$, respectively. Each convolutional layer holds $16$ filters and incorporates leaky ReLU nonlinearities ($\alpha=0.01$). Subsequently, the data passes through a fully connected layer of size $64$ and is then transformed into a scalar. Batch normalization and dropout with a dropout rate of $0.5$ are applied to all convolutional layers to mitigate overfitting. An AdamW \citep{loshchilov2018adamw} optimizer is adopted with the learning rate $3\times10^{-4}$, betas~$(0.9,0.999)$, epsilon $10^{-8}$ and weight decay $10^{-2}$.

Trajectories of $25$ timesteps are collected to construct the comparison pairs. Initially, we warm up the reward model by $500$ epochs with $600$ labeled pairs. In the following training, it collects $6$ new pairs per reward learning epoch based on the current policy, until the buffer is filled with $3000$ pairs. The batch size is set to $32$.

Regarding the Atari games, wrappers are employed, which originate from \citep{mnih2015human}: initial $0$ to $30$ no-ops to inject stochasticity, max-pooling pixel values over the last two frames, an episodic life counter, four-frame skipping to accelerate sampling, four-frame stacking to help infer game dynamics, warping the image to size $84\times84$ and clipping rewards to $[-1,1]$.

\revise{Specifically, after conducting initial tests, we observe that PBRL with the ``value penalty" and the penalty parameter as $2$ provide robust performance, for which it is adopted for comparison.}

\newpage
\subsection{Experiment Settings for the Synthetic Bilevel RL Problem}
To test M-SoBiRL, we build a synthetic bilevel RL problem in the form of \eqref{eq:standar_biRL} with the upper-level $f$ as a quadratic function of $(x,\pi)$ and the dimension of $x$ as $n=100$.  The lower-level problem is adapted from RL problems used in \citep{zhan2023mirror} and \citep{li2024preconditionmdp}. Specifically, we take $|\mdps|=10,|\mdpa|=5$ and $\gamma=0.5,\tau=1$. The transition probabilities are generated randomly, and for all $(s,a)\in\mdps\times\mdpa$, the reward $r_{sa}$ is parameterized by a quadratic form of $x$ with a random perturbation. The step sizes in \cref{alg:M-SoBiRL} are taken as $(\beta,\eta)=(0.008,0.5)$. Results of the synthetic experiment are shown in \cref{fig:test_M}. The curves exhibit the benign convergence property of the proposed algorithm.
\begin{figure*}[htbp]
\centering
\begin{minipage}{0.49\textwidth}
    \centering
    \includegraphics[width=1\linewidth]{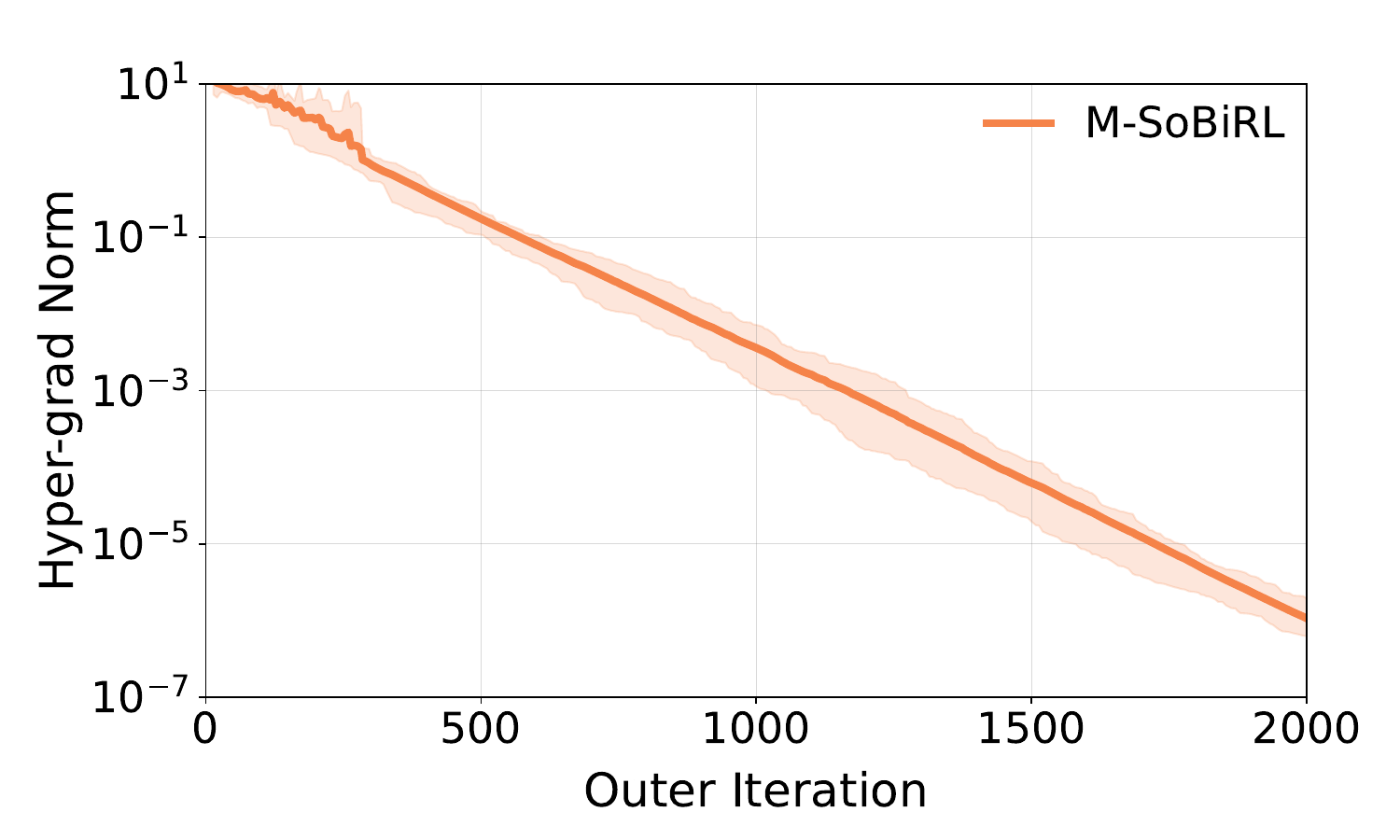}
\end{minipage}
\,
\begin{minipage}{0.49\textwidth}
    \centering
    \includegraphics[width=1\linewidth]{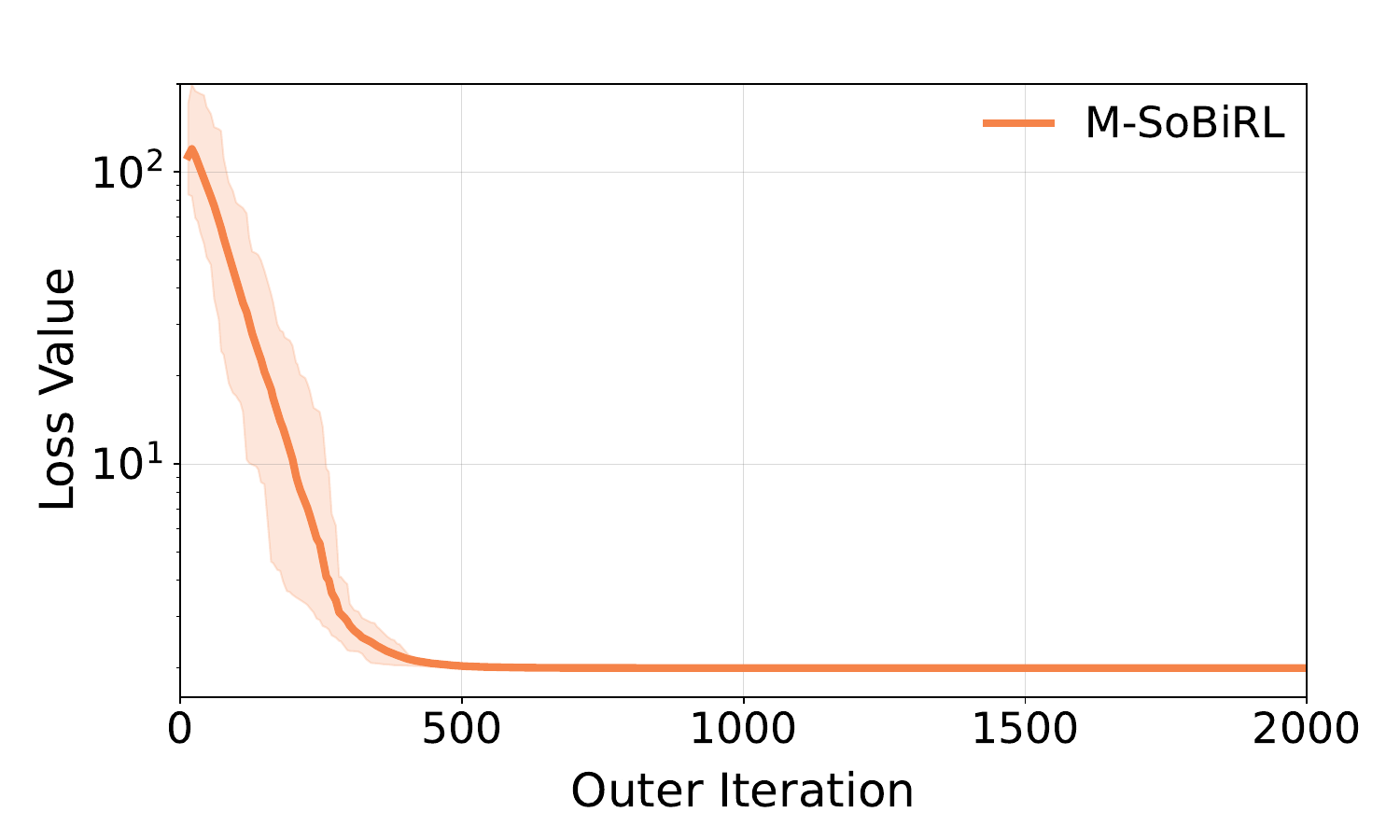}
\end{minipage}
\caption{A synthetic bilevel RL problem to verify the model-based algorithm, M-SoBiRL. Metrics are the hyper-gradient norm $\norm{\nabla \phi(x)}_2$ and the upper-level loss $f(x,\pi)$.}
\label{fig:test_M}
\end{figure*}

\section{\revise{APPLICATIONS OF BILEVEL REINFORCEMENT LEARNING}}\label{sec:append_app}
In this section, we provide a detailed explanation of how the formulation~\eqref{eq:standar_biRL}, with the upper-level function
\begin{equation}\label{eq:upper_f}
    f\left( x,\pi \right) =\mathbb{E} _{d_i\sim \rho \left( d;\pi  \right)}\left[ l\left( d_1,d_2,\ldots,d_I;x \right) \right],
\end{equation}
incorporates bilevel RL applications, using additional examples similar to those in \cref{sec:biRL_app}. We also explain how the boundedness of $l$ assumed in \cref{sec:theory} can be verified individually for each case.

\textbf{Reward Shaping} From the perspective of bilevel RL, we can shape an auxiliary reward function $r(x)$ at the lower level for efficient agent training, while maintaining the original environment at the upper level to align with the initial task evaluation~\citep{hu2020rewardshaping}, which is established as follows,
\begin{equation*}
    \begin{aligned}
        \min _{x\in\mathbb{R}^n}&\ -V_{\bar{\mathcal{M}}_{\bar{\tau}}}^{\pi ^*\left( x \right)}\left( \bar{\rho} \right) = -\mathbb{E}\zkh{Q^{\pi^*(x)}_{\bar{\mdp}_{\bar{\tau}}}(s,a)\,|\,s\sim\bar{\rho},\pi^*(x)}
        \\
        \mathrm{s.\,t.}&\ \ \pi ^*\left( x \right) =\argmin _{\pi \in \Delta^{\abs{\mdpa}}}-V_{\mathcal{M} _{\tau}\left( x \right)}^{\pi}\left( \rho \right).
    \end{aligned}
\end{equation*}
Note that its upper-level function falls into the formulation~\eqref{eq:upper_f} with $I=1$, $H=1$, $l=-Q^{\pi}_{\bar{\mdp}_{\bar{\tau}}}\kh{s_0^1,a_0^1}$. Therefore, the boundedness and smoothness of $l$ can be seen from \cref{pro:QV_Lipz}.

\textbf{Reinforcement Learning from Human Feedback} The target of RLHF is to learn the intrinsic reward function that incorporates expert knowledge, from simple labels only containing human preferences. It optimizes a policy under $r(x)$ at the lower level, and adjusts $x$ to align the preference predicted by the reward model $r(x)$ with the true labels at the upper level.
\begin{equation*}
    \begin{array}{cl}
    \min \limits_{x\in\mathbb{R}^n} & \mathbb{E} _{y,d_1,d_2\sim \rho \left( d;\pi ^*\left( x \right) \right)}\left[ l_h\left( d_1,d_2,y;x \right)  \right]
    \\
    \mathrm{s.\,t.}& \pi ^*\left( x \right) =\argmin\limits_{\pi \in \Delta^{\abs{\mdpa}}}-V_{\mathcal{M} _{\tau}\left( x \right)}^{\pi}\left( \rho \right),  
    \end{array}
\end{equation*}
where each trajectory $d_i=\{\kh{s_h^i,a_h^i}\}_{h=0}^{H-1}$ ($i=1,2$) is sampled from the distribution~$\rho\kh{d;\pi^*(x)}$ generated by the policy $\pi^*(x)$ in the upper-level~$\bar{\mathcal{M}}_{\bar{\tau}}$, i.e., 
\begin{align*}
    P(d_i)=\bar{\rho}\kh{s_0^i}\zkh{\Pi_{h=0}^{H-2}\pi^*(x)\kh{a_h^i|s_h^i}\bar{P}\kh{s_{h+1}^i|s_h^i,a_h^i}}\times\pi^*(x)\kh{a_{H-1}^i|s_{H-1}^i},
\end{align*}
and the preference label $y\in\{0,1\}$, indicating preference for $d_1$ over $d_2$, obeys human feedback distribution $y\sim D_{human}(y| d_1,d_2)$. Moreover, $l_h$ is the binary cross-entropy loss, $l_h\left( d_1,d_2,y;x \right) = -y\log P\left( d_1\succ d_2;x \right)-\left( 1-y \right) \log P\left( d_2\succ d_1 ; x\right)$, with the preference probability $P\kh{d_1\succ d_2;x}$ built by the Bradley--Terry model,
\begin{equation}\label{eq:BT_model}
 P\left( d_1\succ d_2;x \right) =\frac{\exp \left( \sum_{h=0}^{H-1}{r_{s_{h}^{1},a_{h}^{1}}\left( x \right)} \right)}{\exp \left( \sum_{h=0}^{H-1}{r_{s_{h}^{1},a_{h}^{1}}\left( x \right)} \right) +\exp \left( \sum_{h=0}^{H-1}{r_{s_{h}^{2},a_{h}^{2}}\left( x \right)} \right)}. 
\end{equation}   
The upper-level function of RLHF is an instance of \eqref{eq:upper_f} with $I=2$, finite $H$, and $l=\mathbb{E}_y\zkh{l_h\kh{d_1,d_2,y;x}}$. Moreover, the boundedness of the reward $\abs{r_{sa}(x)}\le C_r$ will guarantee the boundedness of $P$, i.e., $\frac{1}{2}{\exp(-2HC_r)}\le{P( d_1\succ d_2;x )}\le 1$, and thus the boundedness of $l$.

\textbf{Contract Design} This problem involves designing payment mechanisms that enable a principal to influence the agent’s decision-making process \citep{wu2024contractual}. Through the state-contingent payment $x\in\mathbb{R}^{\abs{\mdps}\times\abs{\mdps}}_+$, the principal seeks to incentivize the agent to adopt policies that serve the principal’s interests. Specifically, when the agent takes action $a$ at state $s$, it costs $c(s,a)$ in the lower level while the principle only observes the transition $s\to s^\prime$ and receives the reward $R(s,s^\prime)$. To encourage the transition, the principal offers a positive payment $x(s,s^\prime)$, which is thus added to the lower-level objective and subtracted from the upper-level objective.
\begin{equation*}
    \begin{aligned}
        \min _{x\in\mathbb{R}^n}&\ -\mathbb{E}\zkh{\sum_{t=0}^{H} R(s_t,s_{t+1})-x(s_t,s_{t+1})\ |\ s_0\sim\rho,\pi^*(x)}
        \\
        \mathrm{s.\,t.}&\ \ \pi ^*\left( x \right) =\argmin _{\pi \in \Delta^{\abs{\mdpa}}}-\mathbb{E}\zkh{\sum_{t=0}^\infty\gamma^t(x(s_t,s_{t+1})-c(s_t,a_t)+\tau h(\pi_{s_t}))\ |\ s_0\sim\rho,\pi}.
    \end{aligned}
\end{equation*}
Note that its upper-level function is an instance of~\eqref{eq:upper_f} with $I=1$, finite $H$, and $l=\sum_{t=0}^{H} -R(s_t,s_{t+1})+x(s_t,s_{t+1})$. It is reasonable that the reward $R$ and the payment $x$ are bounded in practice, and thus the function $l$ is bounded.

\textbf{Efficient Robot Navigation} As formulated in \citep{chakraborty2024parl}, we consider a maze-world environment of size $N\times N$, and aim to navigate a robot to get close to the destination $(N,N)$ efficiently. To this end, the designer can introduce a modified goal $x\in\mathbb{R}^2$, which induces an associated reward $r_x\in\mathbb{R}^{\abs{\mdps}\times\abs{\mdpa}}$, and then use this goal to train the robot (in the lower level). In this manner, the upper-level objective includes both the received reward and the additional cost of moving the robot from $x$ to $(N,N)$. The task is formulated by the following bilevel problem,
\begin{equation*}
    \begin{aligned}
        \min _{x\in\mathbb{R}^2}&\ \omega\norm{x-(N,N)}^2-\mathbb{E}\zkh{\sum_{t=0}^{H} r_x(s_t,a_t)|\ s_0\sim\rho,\pi^*(x)}
        \\
        \mathrm{s.\,t.}&\ \ \pi ^*\left( x \right) =\argmin _{\pi \in \Delta^{\abs{\mdpa}}}-\mathbb{E}\zkh{\sum_{t=0}^\infty\gamma^t(r_x(s_t,a_t)+\tau h(\pi_{s_t}))\ |\ s_0\sim\rho,\pi},
    \end{aligned}
\end{equation*}
where $\omega>0$ is the weight balancing the deviation from the original destination $(N,N)$ and the accumulative reward measured under the modified goal $x$. In this case, the upper-level function matches~\eqref{eq:upper_f} with $I=1$, finite $H$, and $l=-\sum_{t=0}^{H} \kh{r_x(s_t,a_t)-\omega\norm{x-(N,N)}^2/H}$, which is bounded if the reward is bounded.

\section{\revise{EXTENSION TO ASYNCHRONOUS AND DISTRIBUTED BILEVEL RL SETTINGS}}\label{sec:distributed}
In \cref{alg:SoBiRL}, the task within the $k$-th outer iteration is divided into two components: one is dedicated to approximately solving the lower-level problem such that $\|\pi _k^* - \pi _k\|^2_2\le\epsilon$, and the other is used to collect rollouts via $\pi_k$. We denote the sample complexity as $c _{\mathrm{solver}}$ and $c _{\mathrm{roll}}$, respectively. Employing the asynchronous strategy, existing work~\citep{shen2023asynchro} shows the actor-critic algorithm achieves a sample complexity $c _{\mathrm{solver}}/{N}$ per worker, where $N$ is the number of workers. In a similar spirit, we can also take advantage of multiprocessing to collect rollouts amongst $N$ workers asynchronously and thus enjoy $c _{\mathrm{roll}}/N$ complexity per worker. Finally, aggregate the rollouts from all the workers to estimate the hyper-gradient and update the upper-level variable. In this manner, the error bound for the hyper-gradient estimation in our work still holds, and thus the convergence analysis can be generalized.

In the distributed setting (e.g., see \citep{chen2022distributedRL}), agents are connected via a fully decentralized network and interact with a shared environment. Since the agent can communicate only with its neighbors, estimating the (global) hyper-gradient becomes challenging. Therefore, developing a tailored communication strategy and establishing a convergence analysis is an interesting and meaningful topic with broad applications in multi-agent RL problems. To this end, recent advances in decentralized bilevel optimization offer valuable insights. \citep{kong2024decentralized,he2024distributed,zhu2024sparkle}.

\end{document}